\documentclass[10pt]{amsart}
\usepackage[initials]{amsrefs}
\usepackage{amssymb,graphicx,amscd,color,amsmath,amsfonts,amssymb,latexsym,amsthm,units,paralist,subfigure}
\usepackage[all]{xy}
\usepackage{xcolor}
\usepackage{hyperref}
\hypersetup{colorlinks=false,linkbordercolor=red,linkcolor=green,pdfborderstyle={/S/U/W 1}}

\usepackage{mathtools}

\newcommand{\C}{{\mathcal C}}

\renewcommand{\H}{{\mathcal H}}

\newcounter{ChartNo}

\newcounter{FigureNo}

\newcommand{\UZ}{\operatorname{UZ}}

\newcommand{\ignore}[1]{}

\newcommand{\real}{{\mathbb R}}
\newcommand{\R}{{\real}}
\newcommand{\rational}{{\mathbb Q}}
\newcommand{\Q}{{\rational}}

\newcommand{\N}{{\mathbb N}}
\newcommand{\Z}{{\mathbb Z}}

\newcommand{\la}{{\langle}}
\newcommand{\ra}{{\rangle}}

\newcommand{\A}{{\mathcal A}}

\newcommand{\D}{{\mathcal D}}

\newcommand{\F}{{\mathcal F}}
\newcommand{\G}{{\mathcal G}}

\newcommand{\J}{{\mathcal J}}

\newcommand{\e}{{\varepsilon}}

\newcommand{\charf}[1]{\mbox{\raise.48ex\hbox{$\chi$}$_{#1}$}}

\def\cl{{\rm cl}}
\def\diam{{\rm diam}}
\def\dist{{\rm dist}}

\def\continuum{{\mathfrak c}}

\def\Cantor{{\mathfrak C}}
\def\la{\langle}
\def\ra{\rangle}

\newcommand{\cov}{{\rm cov}}

\newtheorem{theorem}{Theorem}[section]
\newtheorem{corollary}[theorem]{Corollary}
\newtheorem{proposition}[theorem]{Proposition}
\newtheorem{lemma}[theorem]{Lemma}
\newtheorem{problem}[theorem]{Problem}
\newtheorem{example}[theorem]{Example}
\newtheorem{definition}[theorem]{Definition}
\newtheorem{remark}[theorem]{Remark}
\newtheorem{Fact}[theorem]{Fact}
\newtheorem{Claim}[theorem]{Claim}
\newtheorem{conjecture}{Conjecture}

\newcommand{\thm}[2]{\begin{theorem}\label{#1}{\sl #2}\end{theorem}}
\newcommand{\cor}[2]{\begin{corollary}\label{#1}{\sl #2}\end{corollary}}
\newcommand{\prop}[2]{\begin{proposition}\label{#1}{\sl #2}\end{proposition}}
\newcommand{\lem}[2]{\begin{lemma}\label{#1}{\sl #2}\end{lemma}}
\newcommand{\pr}[2]{\begin{problem}\label{#1}{\rm #2}\end{problem}}
\newcommand{\ex}[2]{\begin{example}\label{#1}{\rm #2}\end{example}}

\newcommand{\rem}[2]{\begin{remark}\label{#1}{\rm #2}\end{remark}}

\begin{document}
\title[Differentiability versus continuity]{Differentiability versus continuity: Restriction and extension theorems and 
monstrous examples}

\author[Ciesielski]{Krzysztof C. Ciesielski}
\address{Department of Mathematics,\newline \indent West Virginia University, Morgantown,\newline \indent WV 26506-6310, USA.\newline \indent
\textsc{ and }\newline \indent  \noindent Department of Radiology, MIPG,\newline \indent University of Pennsylvania,\newline 
\indent Philadelphia, PA 19104-6021, USA.}
\email{KCies@math.wvu.edu}

\author[Seoane]{Juan B. Seoane--Sep\'ulveda}
\address{Instituto de Matem\'atica Interdisciplinar (IMI), \newline\indent Departamento de An\'alisis y Matem\'atica Aplicada,\newline\indent Facultad de Ciencias Matem\'aticas, \newline\indent Plaza de Ciencias 3, \newline\indent Universidad Complutense de Madrid,\newline\indent 28040 Madrid, Spain.}
\email{jseoane@ucm.es}

\subjclass[2010]{
26A24, %26A24  	Differentiation (functions of one variable): general theory, generalized derivatives, mean-value theorems [See also 28A15]
54C30,                      %Real-valued functions [See also 26-XX]
46T20, %46T20  	Continuous and differentiable maps [See also 46G05]
58B10, %58B10  	Differentiability questions
54A35, %54A35  	Consistency and independence results [See also 03E35]
26A21, %26A21  	Classification of real functions; Baire classification of sets and functions [See also 03E15, 28A05, 54C50, 54H05]
26A27, %26A27  	Nondifferentiability (nondifferentiable functions, points of nondifferentiability), discontinuous derivatives
26A30, %	26A30  	Singular functions, Cantor functions, functions with other special properties
54C20, %54C20  	Extension of maps
41A05%41A05  	Interpolation [See also 42A15 and 65D05]
 }
\keywords{continuous function, differentiable function, points of continuity, points of differentiability, Whitney extension theorem, interpolation theorems, independence results, classification of real functions, Baire classification functions, Pompeiu derivative}
%\thanks{}

\begin{abstract}
	The aim of this expository article is to present recent developments in the centuries old discussion on the interrelations between 
	continuous and differentiable real valued functions of one real variable. The truly new results include, among others, 
	the $D^n$-$C^n$ interpolation theorem: {\em For every $n$-times differentiable $f\colon\R\to\R$ and perfect $P\subset \R$ 
	there is a $C^n$ 
	function $g\colon\R\to\R$ such that $f\restriction P$ and $g\restriction P$ agree on an uncountable set} 
	and an example of a differentiable function $F\colon\R\to\R$ (which can be nowhere monotone)
	and of compact perfect $\mathfrak{X}\subset\R$ such that $F'(x)=0$ for all $x\in \mathfrak{X}$ 
	while $F[\mathfrak{X}]=\mathfrak{X}$; thus, the map $\mathfrak{f}=F\restriction\mathfrak{X}$ is shrinking at every point while, 
	paradoxically, not globally. However, the novelty is even more prominent in the newly discovered simplified presentations 
	of several older results, including: a new short and elementary construction of {\em everywhere differentiable nowhere
	 monotone $h\colon \R\to\R$}\/
	and the proofs (not involving Lebesgue measure/integration theory) of the theorems of Jarn\'\i k: 
	{\em Every differentiable map $f\colon P\to\R$, with $P\subset \R$ perfect, 
	admits differentiable extension $F\colon\R\to\R$}\/ and of Laczkovich: 
	{\em For every continuous $g\colon\R\to\R$ there exists a perfect $P\subset\R$ such that 
	$g\restriction P$ is differentiable}. The main part of this exposition, concerning continuity 
	and first order differentiation, is presented in an narrative that answers two classical questions: 
	\textit{To what extend a continuous function must be differentiable?}
	and \textit{How strong is the assumption of differentiability of a continuous function?}
	In addition, we overview the results concerning higher order differentiation. 
	This includes the Whitney extension theorem and the higher order interpolation theorems related to 
	Ulam-Zahorski problem. Finally, we discuss the results concerning smooth functions that are independent 
	of the standard axioms ZFC of set theory. We close with a list of currently open problems related to this subject. 
\end{abstract}

%\date{\UpdateDate}

\maketitle

%%%%%%%%%%%%%%%%%%%%%%%%%%%%%%%%%%%%%%%%%%%%%%%%%%%%%%%%%%%%%%%%

\newpage

\tableofcontents

\section{Introduction and overview} 
Continuity and Differentiability are among the most fundamental concepts of differential calculus. 
Students often struggle to fully comprehend these deep notions. 
Of course, as we all know, differentiability is a much stronger condition than continuity. Also, continuity can behave in many strange ways. For instance, besides the classical definition of continuous function, there are many characterizations of this concept that, usually, are not taught at an undergraduate level. One of these characterization (due to Hamlett \cite{hamlett}, see also \cite{continuidad02,continuidad03,velleman,Utz,White} for further generalizations) states that a function $f$ from $\R$ to $\R$ is continuous if, and only if, $f$ maps continua (compact, connected sets) to continua.  On the contrary, and quite surprisingly, there are nowhere continuous functions mapping connected sets to connected sets, and (separately) compact sets to compact sets (\cite{continuidad01}). Somehow, a key direction of study of continuity and differentiability deals with trying to provide a clear structure of what the set of points of continuity/differentiability looks like.

The leading theme of this expository paper is to discuss the following two questions
concerning the functions from $\R$ to $\R$: 
\begin{itemize}
	\item[Q1:] How much continuity does differentiability imply?
	\item[Q2:] How much differentiability does continuity imply?
\end{itemize}
They will be addressed in Sections~\ref{sec:CfromD} and~\ref{sec:DfromC}, respectively. 
The main narrative presented in these sections 
is independent of any results from Lebesgue measure and/or integration theory.
In particular, it can be incorporated into an introductory course of real analysis. 
The question Q1, addressed in Section~\ref{sec:CfromD}, will be interpreted as a question on the continuity of the derivatives. 
We will recall that the derivatives must have the intermediate value property 
and must be continuous on a dense $G_\delta$-set.\footnote{As usual, we say that a set is a $G_\delta$-set whenever is the countable intersection of open sets.} 
The key novelty here will be a presentation of recently found simple construction of a {\em differentiable monster}, that is, 
a differentiable function that is nowhere monotone. The existence of such examples 
shows that none of the good 
properties of the derivatives discussed earlier can be improved. 
Towards the answer of Q2, presented in Section~\ref{sec:DfromC}, we start with a construction of a
{\em Weierstrass monster}, that is, 
a continuous function
which is differentiable at no point. Then, we proceed to show that 
for every continuous functions $f\colon\R\to\R$ there exists a compact perfect set $P\subset \R$ such that
the restriction $f\restriction P$ is differentiable. 
Moreover, there exists a $C^1$ (i.e., continuously differentiable) function $g\colon\R\to\R$
for which the set $[f=g]\coloneqq\{x\in \R\colon f(x)=g(x)\}$ is uncountable, 
the result known as $C^1$-interpolation theorem. 
This function $g$ will be constructed using a recently found version of a
Differentiable Extension Theorem of Jarn\'\i k. 
As a part of this discussions we also show that differentiable functions 
from a compact perfect set $P\subset\R$ into $\R$ can behave quite paradoxical,
by describing a simple construction of a perfect subset $\mathfrak{X}$
of the Cantor ternary set $\Cantor\subset \R$
and a differentiable surjection
$\mathfrak{f}\colon \mathfrak{X}\to\mathfrak{X}$ such that $\mathfrak{f}\,'\equiv 0$.  

In Section~\ref{sec:higher} we shall discuss the extension and interpolation theorems for the classes of functions from $\R$ to $\R$ of higher smoothness: 
$D^n$, of $n$-times differentiable functions, and $C^n$, of those functions from $D^n$
whose $n$th derivative is continuous.

Section \ref{sec:independence} shall  provide a set theoretical flavor to this expository paper by giving an overview of the results 
related to smooth functions which are independent of the axiomatic set system ZFC. We shall go through interpolations theorems via consistency results, and provide some consequences of 
the Martin's Axiom (MA), Continuum Hypothesis (CH), or Covering Property Axiom (CPA) on ``how many'' $C^n$ functions are needed in order to cover any $f \in D^n$. Furthermore, and in the spirit of the classical {\em Sierpi\'nski decomposition}, we will also consider 
when the plane can be covered by the graphs of few 
(i.e., fewer than continuum many) continuous/continuously differentiable functions.
Other consistency results will be presented. At the end, some open questions and final remarks will be provided in 
Section~\ref{sect6}.

\section{Continuity from differentiability}\label{sec:CfromD}

Clearly any differentiable function $F\colon\R\to\R$  is continuous. 
Thus, the true question we will investigate 
here is: 
\begin{quote}
Can we ``squeeze'' more continuity properties 
from the assumption of differentiability of $F$? 
\end{quote}
More specifically: 
\begin{quote}
{\em To what extent the derivative $F'$ of $F$ must be continuous?}
\end{quote}

Of course, for a differentiable complex function $F\colon{\mathbb C}\to{\mathbb C}$, its derivative
is still differentiable and, thus, continuous. The same is true for any real analytic function
$F(x)\coloneqq\sum_{n=0}^\infty a_n x^n$. Moreover, one of the most astonishing results, 
due to Augustin Louis Cauchy\footnote{We include the pictures and birth-to-death years of the main 
contributors to this story who are already deceased.} (1789-1857), \cite{cauchy}, 
in the theory of complex variables states that for functions of a single complex variable, 
a function being analytic is equivalent to it being holomorphic. 
Recall that, although the term analytic function is often used interchangeably with that of holomorphic function, 
the word analytic is defined in a much broader sense in order to denote any function 
(real, complex, or of more general type) that can be written as a convergent power series in a neighborhood of each point in its domain.

\begin{figure}[h!]
	\centering
	\includegraphics[width=0.3\textwidth]{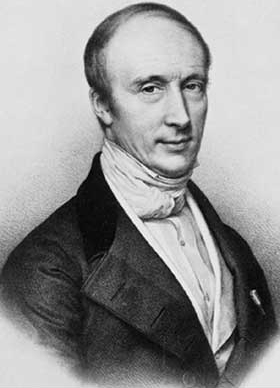}
	\caption{Augustin Louis Cauchy.}
	\label{pic_cauchy}
\end{figure}

For the functions from $\R$ to $\R$ the situation is considerably more complicated.
It is true that most of the derivatives we see in regular calculus courses are continuous. 
Nevertheless (in spite of what probably many calculus students believe)
there exist differentiable functions from $\R$ to $\R$ with discontinuous derivatives.
This is commonly exemplified by a map $h\colon\R\to\R$ given by (\ref{D1notC1}),
which appeared already in a 
1881 paper~\cite[p.~335]{Vol} of Vito Volterra
(1860-1940):
\begin{equation}\label{D1notC1}
	h(x)\coloneqq
	\begin{cases}
		x^2\sin\left(x^{-1}\right) & \mbox{ for $x\neq 0$},\\
		0 & \mbox{ for $x=0$},\\
	\end{cases}
\end{equation}
(see Fig.~\ref{Fig:C1notD1}) with the 
derivative\footnote{$h'(0)=0$ follows from the squeeze theorem, since 
	$\left|\frac{h(x)-h(0)}{x-0}\right|\leq \left|\frac{x^2-h(0)}{x-0}\right|=|x|$.}
\begin{equation*}
	h'(x)\coloneqq
	\begin{cases}
		2x\sin\left(x^{-1}\right)- \cos\left(x^{-1}\right) & \mbox{ for $x\neq 0$},\\
		0 & \mbox{ for $x=0$}.\\
	\end{cases}
\end{equation*}

\begin{figure}[h!]
	\centering
	\begin{tabular}{cc}
		\includegraphics[width=0.48\textwidth]{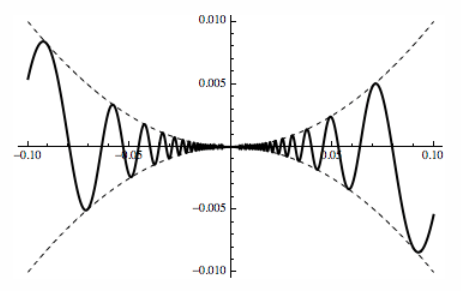} & \includegraphics[width=0.48\textwidth]{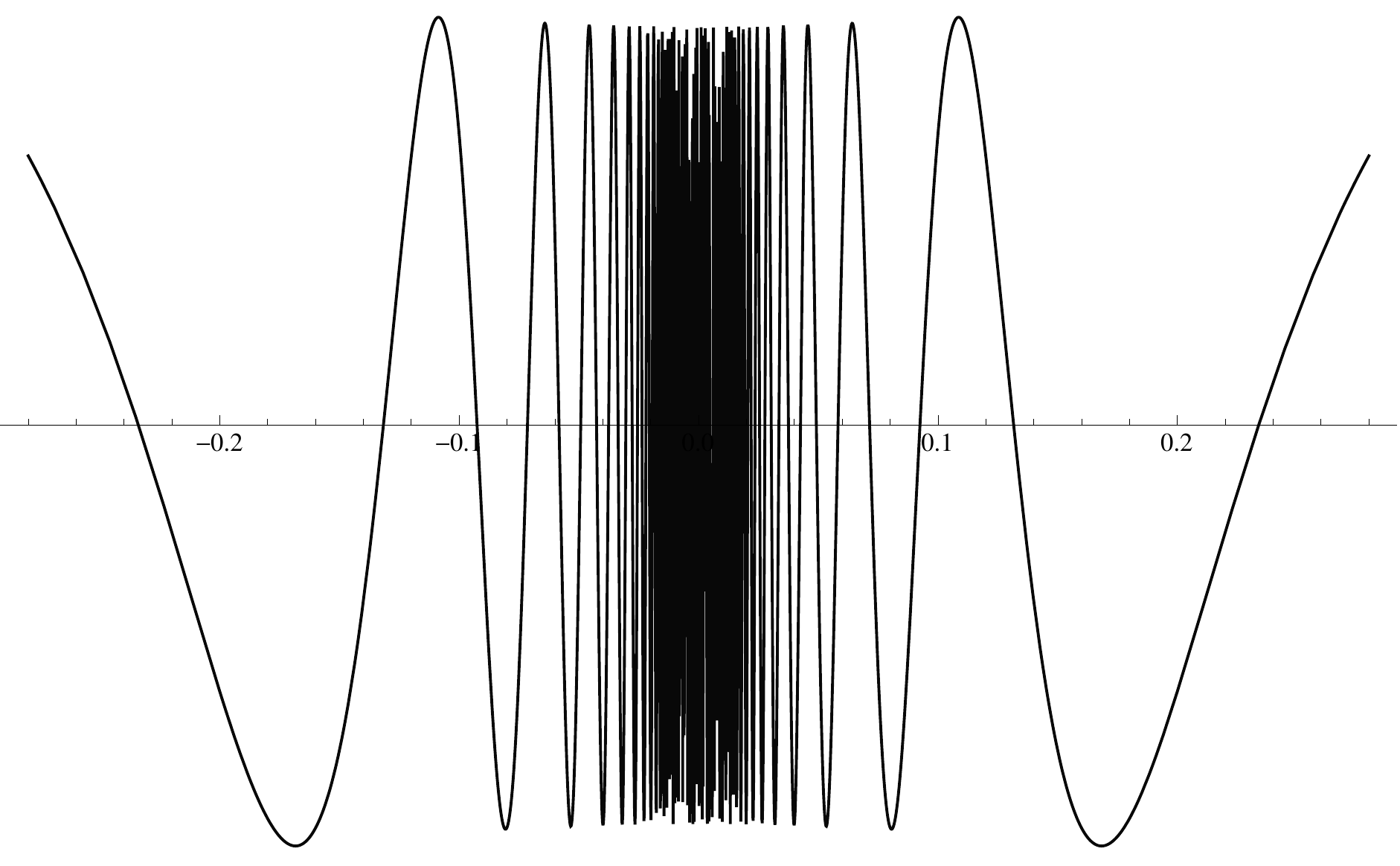}
	\end{tabular}
	\caption{On the left, the graph of function $h$ given by equation \eqref{D1notC1}. On the right, 
	the graph its derivative, $h'$.}\label{Fig:C1notD1}
\end{figure}

The existence of such examples lead mathematicians to the investigation of
which functions are the derivatives and, more generally, of the 
properties that the derivatives must have.
The systematic study of the structure of derivatives (of functions from $\R$ to $\R$)
can be traced back to at least the second half of the 19th century.
The full account of these studies can be found in the 250 page monograph \cite{Bru1978} of 
Andrew Michael Bruckner (1932--) from 1978 (with bibliography containing 218 items) and its 1994 updated edition \cite{Bru1994}. 
The limited scope of this article does not allow us to give a full account of such a vast material.
Instead, we will focus here on the results that are best suited to this narrative. 

We will start with discussing the ``nice'' properties of the class of derivatives, that is,
those that coincide with the properties of the class of all continuous functions. 

\subsection{Nice properties of derivatives}\label{sec:Nice}
Of course, by the Fundamental Theorem of Calculus, any continuous function $f\colon\R\to\R$ is
a derivative of $F\colon\R\to\R$ defined as 
$F(x)\coloneqq\int_0^x f(t)\; dt$. 
Also, in spite of the existence of discontinuous derivatives, 
the derivatives share several ``good" properties of continuous functions. 
For example, the derivatives form a vector space:\footnote{However, the product of a two derivatives need
	not be a derivative. For example, if $F(x)\coloneqq x^2\sin\left(1/x^{2}\right)$ for $x\neq 0$ and $F(0)\coloneqq 0$, then $F$ is differentiable,
	but $(F')^2$ is not a derivative (see, e.g., \cite[page 17]{Bru1978}).\label{footnote:prod} 
}
this follows from a basic differentiation formula, that $(a\, f(x)+ b\, g(x))'=a\, f'(x)+ b\, g'(x)$
for every differentiable functions $f,g\colon\R\to\R$. 
Also, similarly as any continuous function, any derivative has the intermediate value property,
as shown by Jean-Gaston Darboux (1842-1917) in his 1875 paper~\cite{Darb}. 

\begin{figure}[h!]
	\centering
	\begin{tabular}{cc}
		\includegraphics[width=0.355\textwidth]{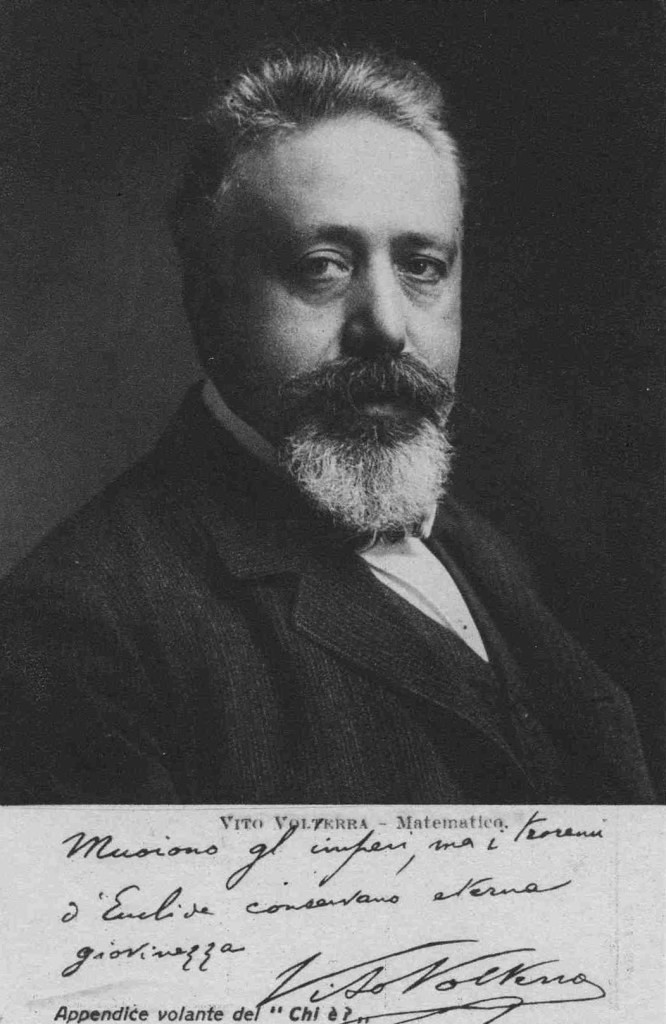} \hspace{.3cm} & \hspace{.3cm} \includegraphics[width=0.42\textwidth]{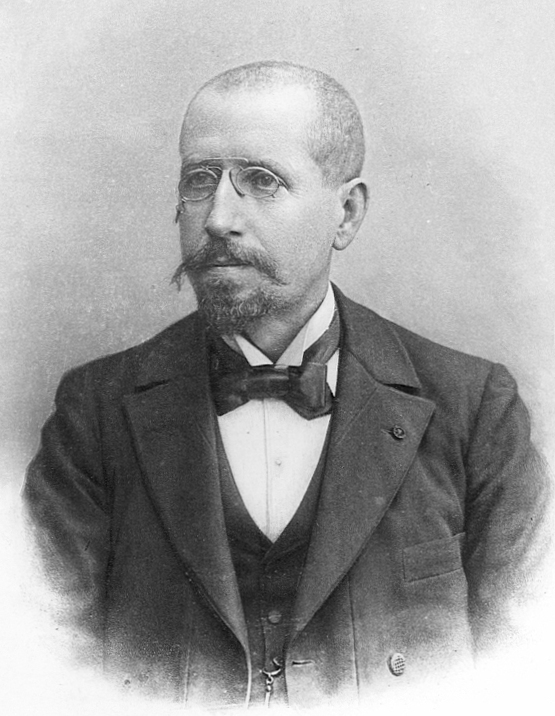}
	\end{tabular}
	\caption{Vito Volterra and Jean-Gaston Darboux.}
	\label{pic_VD}
\end{figure}

\thm{thm:Darboux}{Any derivative $f\colon\R\to\R$ has the intermediate value property, that is,
	for every $a<b$ and $y$ between $f(a)$ and $f(b)$ there exists an $x\in[a,b]$ with $f(x)=y$.
}

\begin{proof}
	Assume that $f(a)\leq y\leq f(b)$ (the case $f(b)\leq y\leq f(a)$ is similar).
	Let $F\colon \R\to\R$ be such that $F'=f$
	and define $\varphi\colon \R\to\R$ as %via 
	$\varphi(t)\coloneqq F(t)-yt$. 
	Then $\varphi'(t)=f(t)-y$  and 
	$\varphi'(a)=f(a)-y\leq 0\leq f(b)-y=\varphi'(b)$. 
	We need to find an $x\in[a,b]$ with $\varphi'(x)=0$.
	This is obvious, unless $\varphi'(a)< 0<\varphi'(b)$,
	in which case the maximum value of $\varphi$ on $[a,b]$,
	existing by the extreme value theorem,
	must be attained at some $x\in(a,b)$.
	For such $x$ we have $\varphi'(x)=0$.
\end{proof}

An alternative proof of Theorem~\ref{thm:Darboux} can be found in \cite{Olsen}. 
Because of Darboux's work, nowadays the functions from $\R$ to $\R$ that have the intermediate value property
(or, more generally, the functions from a topological space $X$ into a topological space $Y$
that map connected subsets of $X$ onto connected subsets of $Y$)
are often called {\em Darboux functions}.

It is easy to see that a composition of Darboux functions remains Darboux.
Thus, the composition of derivatives, similarly as the composition of  continuous functions on $\R$, still enjoys the intermediate value property, in spite the fact that a composition of two derivatives need not be a derivative.\footnote{For example,\label{footnote:comp}
	if $F$ is  as in the footnote \ref{footnote:prod}, then $(F')^2$ is a composition of two derivatives,
	$F'(x)$ and $g(x)\coloneqq x^2$, which is not a derivative.}

\begin{figure}[h!]
	\centering
	\includegraphics[width=0.4\textwidth]{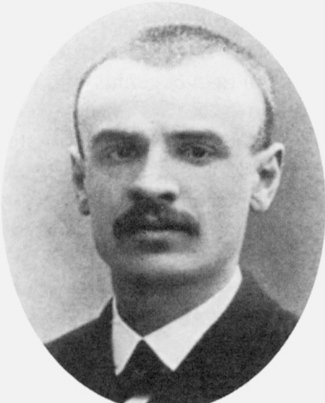}
	\caption{Ren\'e-Louis Baire.}
	\label{pic_Baire}
\end{figure}

The next result, that comes from the 1899 dissertation~\cite{Baire} of Ren\'e-Louis Baire (1874-1932),
shows that every derivative must have a lot of points of 
continuity. (See also \cite{MEDV}.) 
This research of Baire is closely related to another fascinating real analysis topic: the study of {\em separately continuous functions}\/ (i.e., of functions $g\colon\R^2\to\R$ for which the sectional maps $g(x,\cdot)$ and $g(\cdot,x)$ are continuous for all $x\in\R$),\footnote{It is easy to see that 
$f\colon\R\to\R$ is differentiable if, and only if, the difference quotient map 
$q(x,y)\coloneqq\frac{f(x)-f(y)}{x-y}$, defined for all $x\neq y$, can be extended to a separately continuous map $g\colon\R^2\to\R$,
with $g(x,x)=f'(x)$ for all $x\in\R$. In this context, we sometimes refer to the functions of more than one variable
as {\em jointly continuous}\/ if they are continuous in the usual (topological) sense.}
recently surveyed by the first author
%Krzysztof C. Ciesielski 
and David Alan Miller \cite{CiMi}. 
(Most interestingly, Cauchy, in his 1821 book {\em Cours d'analyse}~\cite{cours},
proves that separate continuity implies continuity, which is well known to be false.
Nevertheless, as explained in \cite{CiMi}, Cauchy was not mistaken!)

\thm{thm:DerIsB1}{The derivative of any differentiable function $F\colon\R\to\R$ is Baire class one, 
that is, it is a pointwise limit of continuous functions. In particular, since the set of points of continuity 
of any Baire class one function is a dense $G_\delta$-set,	the same is true for $F'$.}

\begin{proof} Clearly $F'$ is a pointwise limit of continuous functions, since we have
	$F'(x)=\lim_{n\to\infty} F_n(x)$, where functions $F_n(x)\coloneqq\frac{f(x+1/n)-f(x)}{1/n}$ are clearly continuous. 
	
	Also, for any function $g\colon\R\to\R$, the  set $C_g$ of points of continuity of $g$ 
	is  a $G_\delta$-set: 
	$C_g\coloneqq\bigcap_{n=1}^\infty V_n$, where the sets
	\[
	V_n\coloneqq\bigcup_{\delta>0}\{x\in\R\colon \mbox{$|g(s)-f(g)|<1/n$ for all $s,t\in(x-\delta,x+\delta)$}\}
	\]
	are open. At the same time, any pointwise limit $g\colon\R\to\R$ of continuous functions $g_n\colon\R\to\R$
	is continuous on a dense $G_\delta$-set $G\coloneqq\bigcap_{n=1}^\infty \bigcup_{N=1}^\infty U^n_N$, 
	where each $U^n_N$ is the interior of the closed set
	$$\{x\in\R\colon |f_k(x)-f_m(x)|\leq 1/n \mbox{ for all $m,k\geq N$}\}.$$ 
	(See e.g., \cite[theorem 48.5]{Munk}.)
	Thus, the set of points of continuity of any Baire class one function must be dense (containing a set in form of $G$) 
	and a $G_\delta$-set. 
\end{proof}

Another interesting property of continuous functions that is shared by derivatives is the fixed point property:

\prop{prop:fixed}{If $f\colon [-1,1]\to[-1,1]$ is a derivative, then it has a fixed point, that is,
	there exists an $x\in[-1,1]$ such that $f(x)=x$.}

\begin{proof}
	The function $g\colon [-1,1]\to\R$, defined as $g(x)\coloneqq f(x)-x$ is a derivative. Thus, by Theorem~\ref{thm:Darboux}, it has the intermediate value property. Since $g(-1)=f(x)+1\geq 0$ and $g(1)=f(1)-1\leq 0$, there exists $z\in[-1,1]$ such that $f(z)-z=g(z)=0$. Thus, $f(z)=z$, as needed.
\end{proof}

More interestingly, a composition of finitely many derivatives, from $[0,1]$ to $[0,1]$, also possesses the fixed point property. For the composition of two derivatives this has been proved, independently, in \cite{CO} and \cite{EKP},
while the general case is shown in \cite{Szuca}.
This happens, in spite the fact that a composition of two derivatives need not be
of Baire class one (so, by Theorem~\ref{thm:DerIsB1}, a derivative).  
This is shown by the following example, which is a variant of one from \cite{BrCi}.\footnote{The fixed point property of the composition 
of derivatives could be easily deduced, if it could be shown that the composition of the derivatives has a connected
graph. However, it remains unknown if such a result is true, see e.g., \cite{BrCi} and Problem~\ref{pr:comp}.}

\ex{Ex:Andy}{There exist derivatives 
	$\varphi,\gamma\colon[-1,1]\to[-1,1]$ such that their composition $\psi\coloneqq\varphi\circ\gamma$ is not of Baire class one. 
}

\begin{proof}
	Let $\gamma(x)\coloneqq\cos(x^{-1})$ for $x\neq 0$ and $\gamma(0)\coloneqq 0$. 
	It is a derivative, as a sum of two derivatives: 
	a continuous function $f(x)\coloneqq x\sin\left(x^{-1}\right)$ for $x\neq 0$, $f(0)\coloneqq 0$,
	and function $-h'$, where $h$ is as in (\ref{D1notC1}). 
	The graph of $\gamma$ is of the form from the right hand side of Figure~\ref{Fig:C1notD1}.
	
	Let $\varphi(x)\coloneqq m\ h'(x-b)$  for $x\in[-1,1]$, where $h$ is  an inverse of 
	Pompeiu's function  from 
	Proposition~\ref{pr:Pom}, $b\in\R$ is such that $h'(b)=0$, and $m>0$ is such that $\varphi[-1,1]\subset[-1,1]$.  
	Clearly, $\varphi$ is also a derivative. 
	
	To see that $\psi=\varphi\circ\gamma$ is not of Baire class one, by Theorem~\ref{thm:DerIsB1}
	it is enough to show that it is discontinuous on the set
	$G\coloneqq\{x\in(-1,1)\colon \varphi(x)=0\}$, which is a dense $G_\delta$. 
	Thus, fix $x\in G$ and note that $\psi(x)=0$. 
	It suffices to show that, for every $\e>0$ such that $(x-\e,x+\e)\subset (-1,1)$, 
	there exists $y\in (x-\e,x+\e)$ with $\psi(y)=1$. 
	Indeed, since $\varphi$ is nowhere constant, there exists $z\in(x,x+\e)$ such that 
	$\varphi(z)>0$. Then, since $\gamma[(0,\varphi(z))]=[-1,1]$, 
	there exists $p\in (0,\varphi(z)]$ such that $\gamma(p)=1$. 
	Also, since $\varphi$ is Darboux and $p\in (0,\varphi(z)]=(\varphi(x),\varphi(z)]$,
	there exists $y\in(x,z)\subset(x,x+\e)$ such that $\varphi(y)=p$. Thus, $\psi(y)=\varphi(\gamma(y))=\varphi(p)=1$, as needed. 
\end{proof}

Next we turn our attention to show, by means of illuminating examples, that the nice properties of the derivatives we stated above cannot be improved in any essential way.

\subsection{Differentiable monster and other examples} 

We start here with showing that all that can be said about the set of points of continuity of
derivatives is already stated in Theorem~\ref{thm:DerIsB1}. 

\thm{thm:DiscDer}{For every dense $G_\delta$-set 
	$G\subset \R$  there exists a differentiable function $f\colon\R\to\R$ such that
	$G$ is the set of points of continuity of the derivative $f'$.}

\begin{proof}[Sketch of proof]  
	Choose $d\in(0,1)$ such that $h'(d)=0$, where $h$ is as in 
	(\ref{D1notC1}). 
	For every closed nowhere dense $E\subset \R$ define $f_E\colon \R\to[-1,1]$ via 
	\[
	f_E(x)\coloneqq
	\begin{cases}
	\displaystyle h\left(\frac{2d\, \dist(x,\{a,b\})}{b-a}\right) & \mbox{ for $x$ in a component $(a,b)$ of $\R\setminus E$,}\\
	0 & \mbox{ for $x\in E$}.
	\end{cases}
	\] 
	This is a variant of a Volterra function from~\cite{Vol}.
	It is easy to see that $f_E$ is differentiable with the derivative continuous on $\R\setminus E$
	and having oscillation 2 at every $x\in E$. 
	
	Now, if $G\subset \R$  is a dense $G_\delta$, then $\real\setminus G\coloneqq \bigcup_{n<\omega}E_n$,
	some some family $\{E_n\colon n<\omega\}$
	of closed nowhere dense sets in $\R$. In particular (see e.g. \cite[Theorem 7.17]{Rudin}), 
	the function $f\coloneqq \sum_{n<\omega} 3^{-n}f_{E_n}$ is differentiable and its derivative 
	$f'=\sum_{n<\omega} 3^{-n}f_{E_n}'$ is continuous precisely at the points of $G$. 
\end{proof} 

Theorems~\ref{thm:DerIsB1} and~\ref{thm:DiscDer} immediately give the following characterization
of the sets of points of continuity of the derivatives.

\cor{cor:ContOfDer}{For a $G\subset \R$, there exists a differentiable function $f\colon\R\to\R$ such that
	$G$ is the set of points of continuity of its derivative $f'$ if, and only if, $G$ is a dense $G_\delta$-set.}

Next, we will turn our attention to an example of a differentiable function which 
is one of the strangest, most mind-boggling
examples in analysis.

\subsubsection{Differentiable nowhere monotone maps}
We will start with the following elementary example showing that,
for a differentiable function $\psi\colon \R\to\R$, a non-zero value of the derivative at a point does 
not imply the monotonicity of the function in a neighborhood of the point. For instance (see \cite{dover2003}), this is the case for $\psi(x)\coloneqq x+2h(x)$, where $h$ is given by (\ref{D1notC1}), that is, 
\begin{equation*}
\psi(x)\coloneqq
\begin{cases}
x + 2 x^2 \sin (x^{-1}) & \mbox{ for $x\neq 0$,}\\
0 & \mbox{ for $x=0$,}\\
\end{cases}
\end{equation*}
see Fig.~\ref{f2}. Indeed $\psi'(0)=1$, while $\psi$ is not monotone on any non-empty interval $(-\delta,\delta)$.

\begin{figure}[h!]
	\centering
	\begin{tabular}{cc}
		\includegraphics[width=0.45\textwidth]{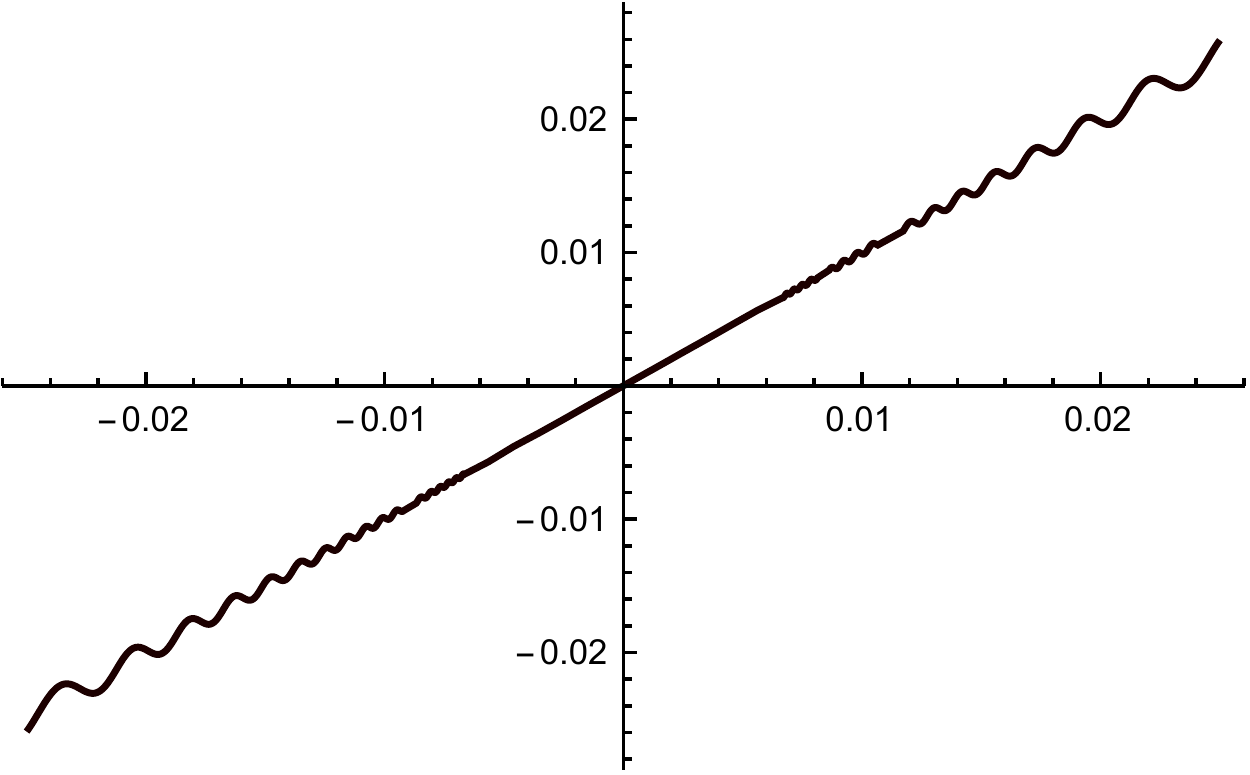} & \includegraphics[width=0.45\textwidth]{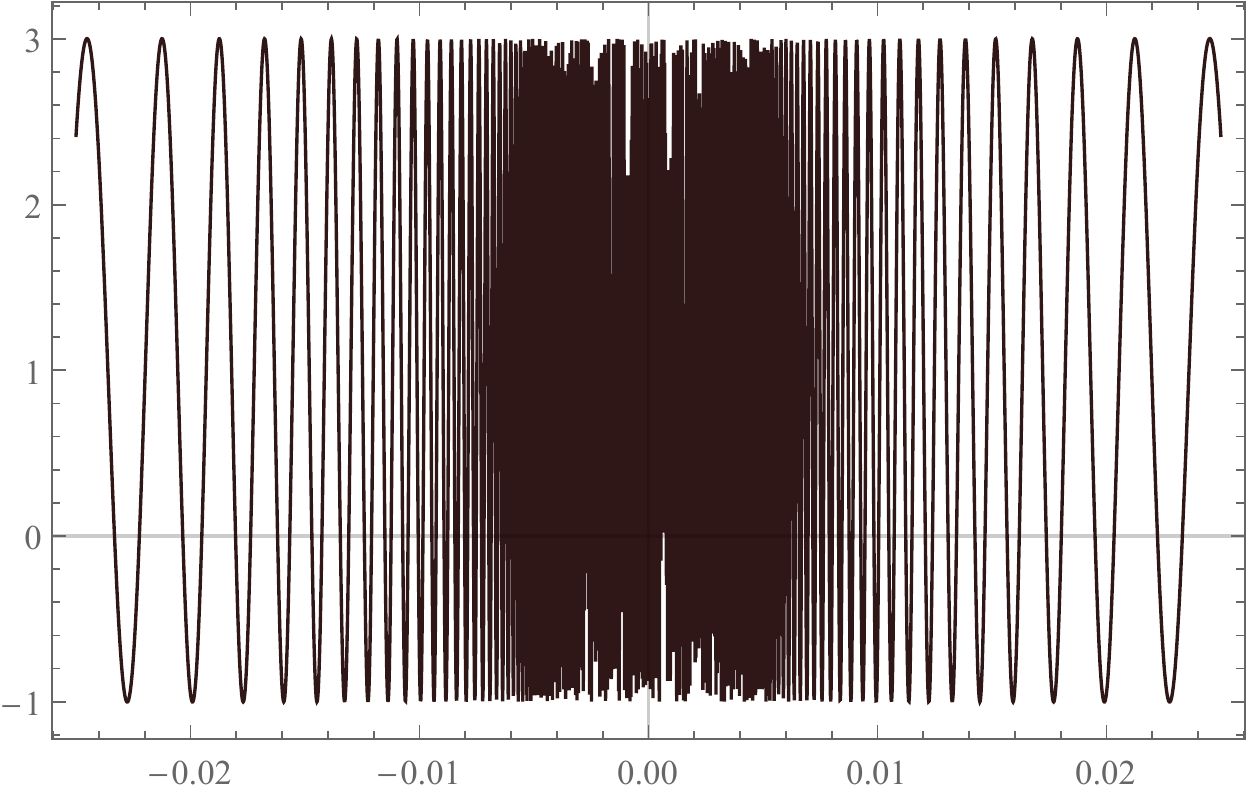}
	\end{tabular}
	\caption{The graphs of $\psi$ and $\psi'$.}\label{f2}
\end{figure}

The example from the next theorem pushes this pathology to the extreme.

\thm{th:Dmonster}{There exists a differentiable function $f\colon\R\to\R$ which is monotone on no non-trivial interval. In particular, each of the sets 
	\begin{equation*}
Z_f\coloneqq\{x\in\R\colon f'(x)=0\} \text{ and } Z_f^c\coloneqq\{x\in\R\colon f'(x)\neq 0\}
	\end{equation*} 
is dense and $f'$ is discontinuous at every $x\in Z_f^c$.}

The graph of such a function is 
simultaneously smooth and very rugged. This sounds (at least to us) like an oxymoron.
Therefore, we started to refer to such functions as {\em differentiable monsters},
see \cite{KC:Monthly}. 
Of course, a differentiable monster $f$ must attain local maximum (as well as local minimum) on every interval.
In particular, the set $Z_f$ must indeed be dense. Its complement, $Z_f^c$, must also be dense, since 
otherwise $f$ would be constant on some interval. 

\begin{figure}[h!]
	\centering
	\begin{tabular}{cc}
		\includegraphics[width=0.4\textwidth]{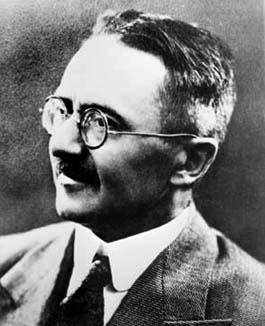} \hspace{.3cm} & \hspace{.3cm} \includegraphics[width=0.34\textwidth]{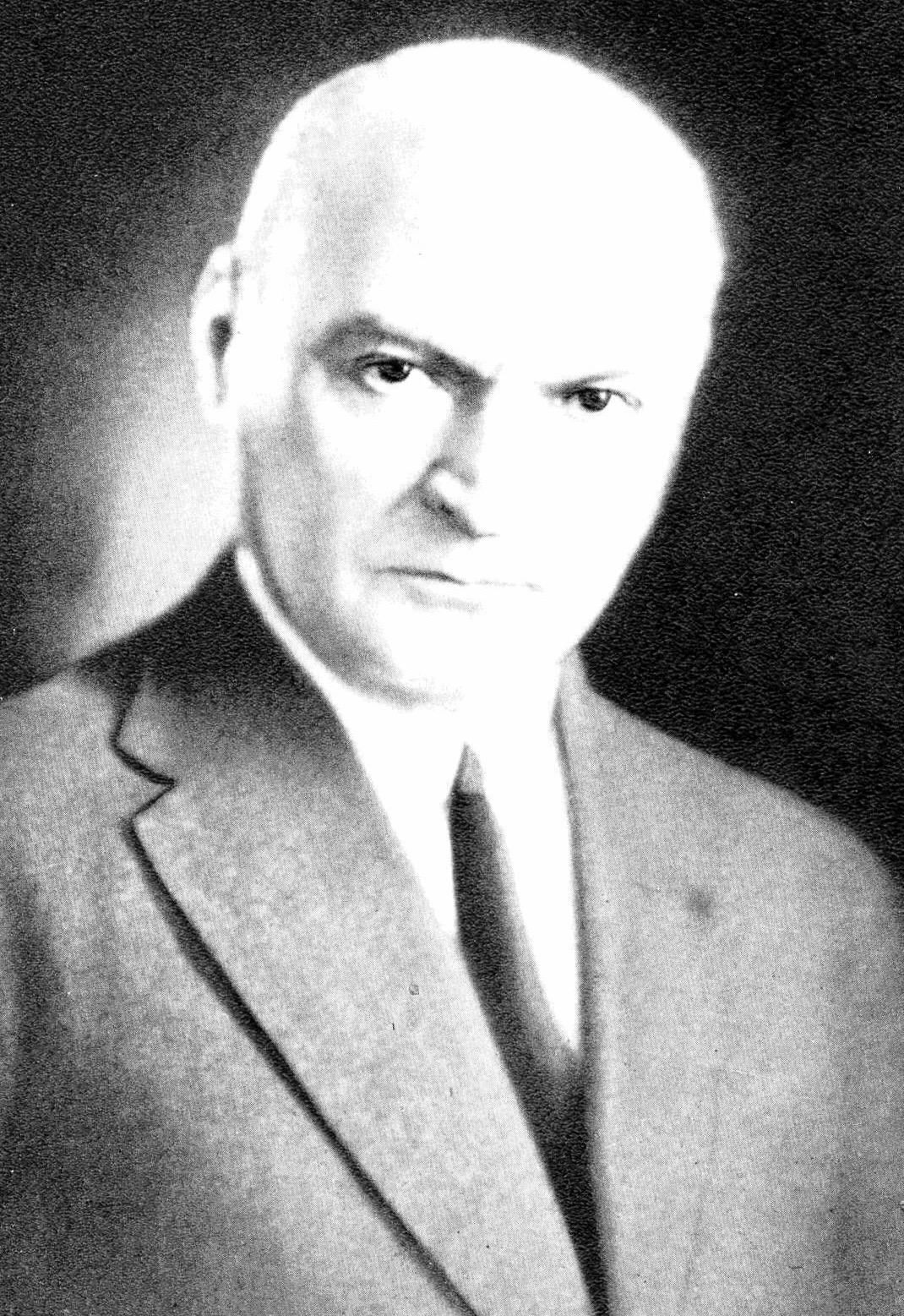}
	\end{tabular}
	\caption{Arnaud Denjoy and Dimitrie Pompeiu.}
	\label{pic_DenPom}
\end{figure}

The history of differentiable monsters is described in detail in the  1983 paper of A.~M.~Bruckner~\cite{Br}. 
The first construction of such a function was given in 1887 (see \cite{Ko1}) by
Alfred~K\"opcke (1852-1932). (There was, originally, a gap in \cite{Ko1}, but it was subsequently corrected in~\cite{Ko2,Ko3}.)
The most influential study of this subject is a 1915 paper~\cite{De} of Arnaud Denjoy (1884--1974). 
The construction we describe below comes from  a recent paper~\cite{KC:Monthly} of the first author. 
It is simpler than the constructions from 
the 1974 paper~\cite{KS} of Yitzhak~Katznelson \mbox{(1934--)} and Karl Stromberg (1931--1994)
%Karl Robert Stromberg
%https://www.math.ksu.edu/events/newsletter/news95.pdf on Stromberg
and 
from the 1976 article~\cite{Weil} of Clifford Weil.  There are quite a few different technical constructions of such monster, some of them including (for instance) the use of wavelets (see the work by R.M. Aron, V.I. Gurariy and the second author, \cite{AGS} or, also, \cite{PAMS2010}).

As a matter of fact (and with the appropriate background) our construction of a differentiable monster 
can be reduced to a single line:
\[
\mbox{\em $f(x)\coloneqq h(x-t)-h(x)$, where $h$ is the inverse of a Pompeiu's map $g$}
\]
as described below and $t$ is an arbitrary number from some residual set
(i.e., a set containing   a dense $G_\delta$-set).
A more detailed argument follows.

\subsubsection*{Pompeiu's functions.}
Let $Q\coloneqq \{q_i\colon i\in\N\}$ be an enumeration of a countable dense
subset of $\R$ (e.g., the set $\Q$ of rational numbers) such that $|q_i|\leq i$ for all $i\in \N$.
Fix an $r\in(0,1)$ and let $g\colon\R\to\R$ be defined by 
$g(x)\coloneqq \sum_{i=1}^\infty r^i (x-q_i)^{1/3}.$
This $g$ is essentially\footnote{Usually $g$ is defined on $[0,1]$ as
$\hat g(x)\coloneqq \sum_{i=1}^\infty a_i (x-q_i)^{1/3}$, where 
$Q$ is an {\em arbitrary}\/ enumeration of $\Q\cap [0,1]$ and numbers 
$a_i>0$ are arbitrary so that $\sum_{i=1}^\infty a_i$ converges.
This $\hat g$ can also be transformed into the function $g\colon\R\to\R$ we need, rather than using our definition.
Specifically, if $[c,d]=\hat g[[0,1]]$, then we can take $g\coloneqq h_1\circ \hat g\circ h_0$, 
where $h_0$ and $h_1$ are the diffeomorphisms from $\R$ onto $(0,1)$ and from $(c,d)$ onto $\R$, respectively. 
}
 the function constructed in the 1907 paper~\cite{Po} of Dimitrie Pompeiu (1873--1954), as seen in Fig.~\ref{Pompeiu}. 
(See~\cite{KC:Monthly}. 
Compare also~\cite{Br} or~\cite[section 9.7]{TBB}.)

Intuitively, $g$ has a non-horizontal tangent line at every point,
vertical at each $(q_i,g(q_i))$, since the same is true for the map $(x-q_i)^{1/3}$.
The graph of its inverse, $h$, is the reflection of the graph of $g$ with respect to the line $y=x$.
So, $h$ also has a tangent line at every point, 
the vertical tangent lines of $g$ becoming horizontal tangents of $h$.

\begin{figure}[h!]
	\centering
\begin{tabular}{cc}
	\includegraphics[width=.45\textwidth]{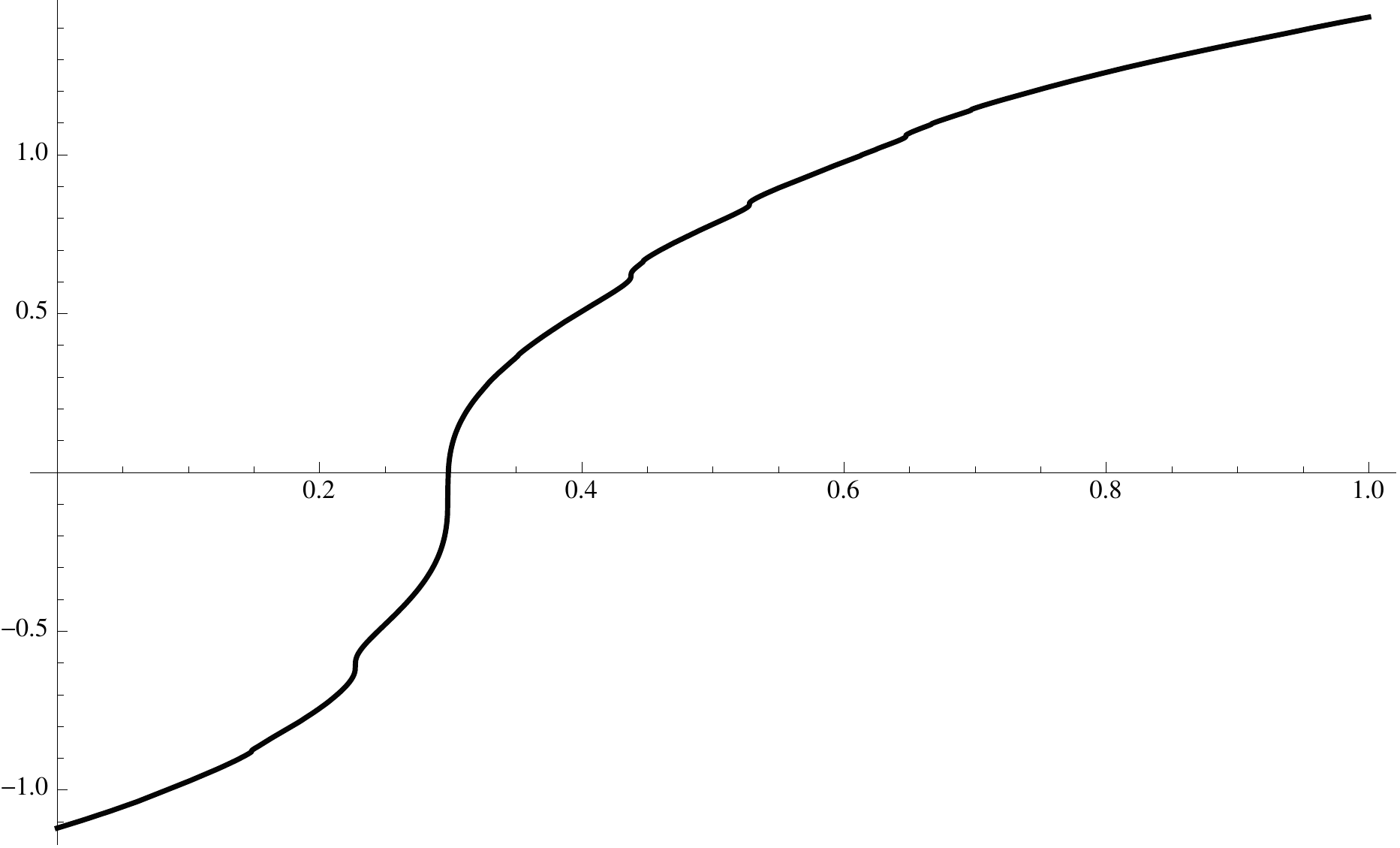} & \includegraphics[width=.45\textwidth]{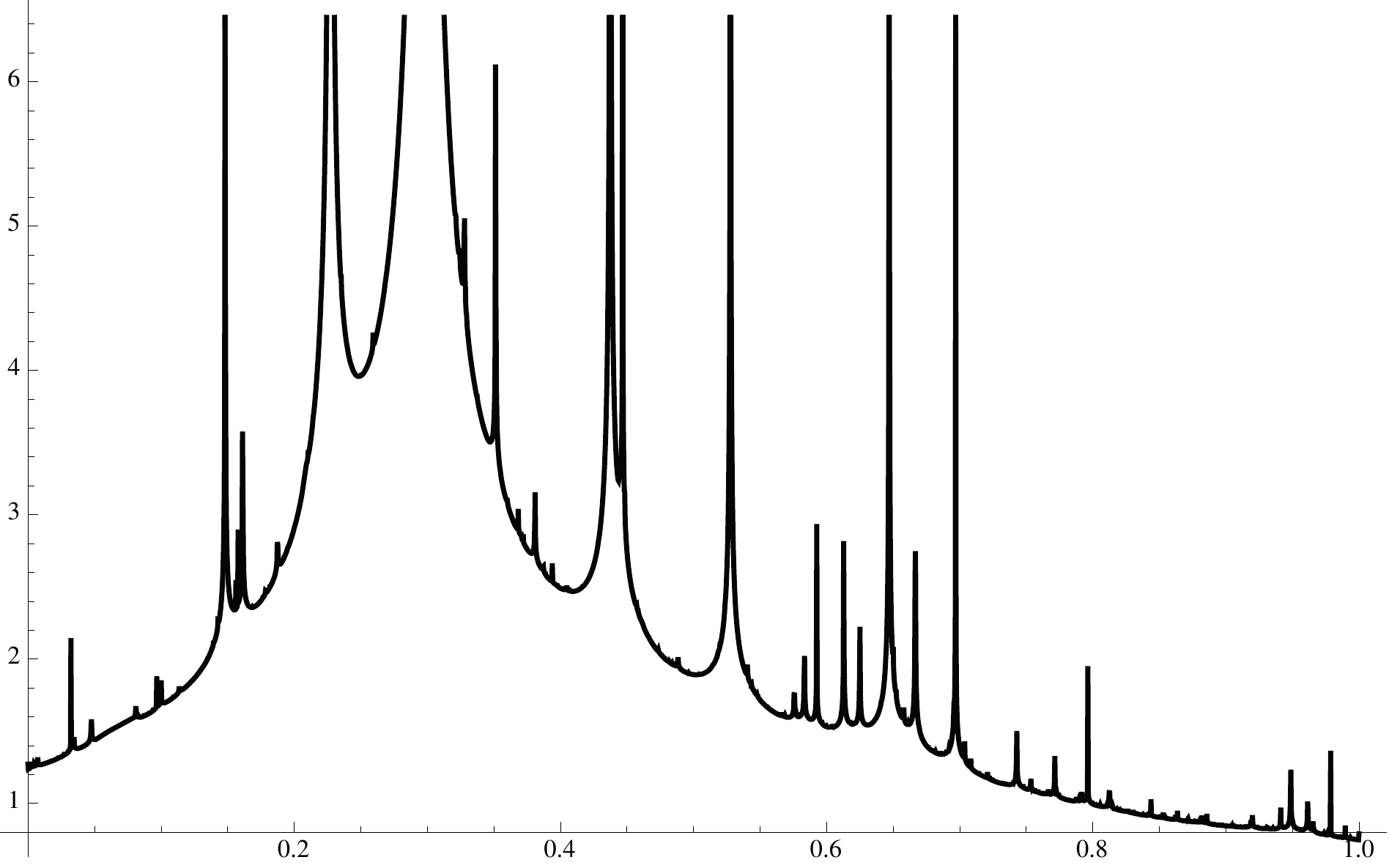}
\end{tabular}
	\caption{Rough sketch of the graph of a partial sum of Pompeiu's original example,
	function $g$ and its derivative, constructed over the interval $[0,1]$.}\label{Pompeiu}
\end{figure}

More formally,  

\prop{pr:Pom}{The map $g$ is continuous, strictly increasing, 
	and  onto $\R$. Its inverse $h\coloneqq g^{-1}\colon\R\to\R$ 
	is everywhere differentiable with $h'\geq 0$ and such that the sets
	$Z\coloneqq\{x\in \R\colon h'(x)=0\}$ and $Z^c\coloneqq\{x\in \R\colon h'(x)\neq 0\}$
	are both dense in~$\R$.
}

\begin{proof}
	The series 
	$$g(x)\coloneqq \sum_{i=1}^\infty r^i (x-q_i)^{1/3}$$
	converges uniformly on every bounded set: 
	$$|g(x)|\leq\sum_{i=1}^\infty r^i(|x|+i+1)$$ since 
	$$\left|(x-q_i)^{1/3}\right|\leq (|x|+|q_i|+1)^{1/3}\leq |x|+|q_i|+1\leq |x|+i+1.$$
	Thus, $g$ is continuous. 
	It is strictly increasing and onto $\R$, since 
	that is true of
	every term $\psi_i(x)\coloneqq r^i (x-q_i)^{1/3}$. 
	
	The trickiest part is to show that 
	\begin{equation}\label{eq1}
		g'(x)=\sum_{i=1}^\infty \psi_i'(x)\left(= \sum_{i=1}^\infty r^i \frac13 \frac{1}{(x-q_i)^{2/3}}\right).
	\end{equation}
	However, this clearly holds when $\sum_{i=1}^\infty \psi_i'(x)=\infty$, since then, for every $y\neq x$, we have
	$$\frac{g(x)-g(y)}{x-y}=\sum_{i=1}^\infty \frac{\psi_i(x)-\psi_i(y)}{x-y}\geq \sum_{i=1}^n \frac{\psi_i(x)-\psi_i(y)}{x-y},$$ 
	and the last expression is arbitrarily large for large enough $n$ and $y$ close enough to $x$. 
	On the other hand, when $\sum_{i=1}^\infty \psi_i'(x)<\infty$, then
	(\ref{eq1}) follows
	from the fact that 
	$$0<\frac{\psi_i(x)-\psi_i(y)}{x-y}\leq6\psi_i'(x)$$
	for every $y\neq x$.\footnote{It is enough to prove this for $\psi(x)\coloneqq x^{1/3}$. It holds for $x=0$, as $\psi'(0)=\infty$. 
		Also, since $\psi(x)$ is odd and concave on $(0,\infty)$, we can assume that $x>0$ and $y<x$. 
		Then $L(y)<\psi(y)$, where $L$ is the line passing through $(x,\psi(x))$ and $(0,-\psi(x))$. 
		So,
		$0<\frac{\psi(x)-\psi(y)}{x-y}<\frac{L(x)-L(y)}{x-y}=\frac{2x^{1/3}}{x}=6\psi'(x)$, as needed.}
	Indeed, given $\e>0$ and $n\in \N$ for which  
	$\sum_{i=n+1}^\infty \psi_i'(x)<\e/14$, 
	
	\begin{align*}
	\left|\frac{g(x)-g(y)}{x-y}-\sum_{i=1}^\infty \psi_i'(x)\right| & \le \sum_{i=1}^n \left|\frac{\psi_i(x)-\psi_i(y)}{x-y}-\psi_i'(x)\right|+7\left| \sum_{i=n+1}^\infty \psi_i'(x)\right|\\
	& \le \sum_{i=1}^n \left|\frac{\psi_i(x)-\psi_i(y)}{x-y}-\psi_i'(x)\right|+\frac\e 2
	\end{align*}
	which 
	is less than $\e$ for $y$ close enough to $x$.

	Now, by 
	(\ref{eq1}),  
	$g'(x)=\infty$ on the dense set $Q$. 
	Therefore, the inverse $h\coloneqq g^{-1}$ is strictly increasing and differentiable, with $h'\geq 0$.
	The set $Z\coloneqq\{x\in \R\colon h'(x)=0\}$ is dense in $\R$,
	since it contains the dense set $g[Q]$. The complement of $Z$ must be dense since, otherwise, $h$ would be constant 
	on some interval. 
\end{proof}

\begin{proof}[Construction of $f$ from Theorem~\ref{th:Dmonster}]

	Let $h$, $Z$, and $Z^c$ be as in 
	Proposition~\ref{pr:Pom} and let $D\subset \R\setminus Z$ be countable 
	and dense.
	Since $h'$ is discontinuous at every $x\in Z^c$, we conclude from Theorem~\ref{thm:DerIsB1} that the set 
	$G\coloneqq \bigcap_{d\in D} \bigl((-d+Z)\cap(d-Z)  \bigr)$ is residual. 
	We claim that, for any $t\in G$, the function $f\colon \R\to\R$ defined as $f(x)\coloneqq h(x-t)-h(x)$
	is a differentiable monster.
	
	Indeed, clearly $f$ is differentiable with $f'(x)=h'(x-t)-h'(x)$.
	Also, $f'>0$ on $t+D$, since for every $d\in D$ we have $t+d\in Z$, so that 
	$$f'(t+d)=h'(d)-h'(t+d)=h'(d)>0.$$
	Similarly, 
	$f'<0$ on $D$, since for every $d\in D$ we have $d-t\in Z$, so that 
	$f'(d)=h'(d-t)-h'(d)=-h'(d)<0$. 
\end{proof} 

\subsubsection{More examples related to calculus}
One of the first things we learn in calculus is that if for a differentiable function $\varphi\colon\R\to\R$
its derivative changes sign at a point $p$, then $\varphi$ has a (local) proper extreme value 
at $p$.

Our fist example here (see also \cite{dover2003}) shows that the change of sign of the derivative is not necessary for the existence of 
a proper extreme value 
at $p$.
For instance, this is the case for $\varphi(x)\coloneqq 2x^4 + x^2h(x)$, where $h$ is given by (\ref{D1notC1}), that is, 
\begin{equation*}
	\varphi(x)\coloneqq
	\begin{cases}
		x^4\left( 2 + \sin (x^{-1}) \right) & \mbox{ for $x\neq 0$,}\\
		0 & \mbox{ for $x=0$,}\\
	\end{cases}
\end{equation*}
see Fig.~\ref{f1}. It is easy to see that it has a proper minimum at $p=0$, while 
its derivative does no change sign at $0$. 

\begin{figure}[h!]
	\centering
	\begin{tabular}{cc}
		\includegraphics[width=0.48\textwidth]{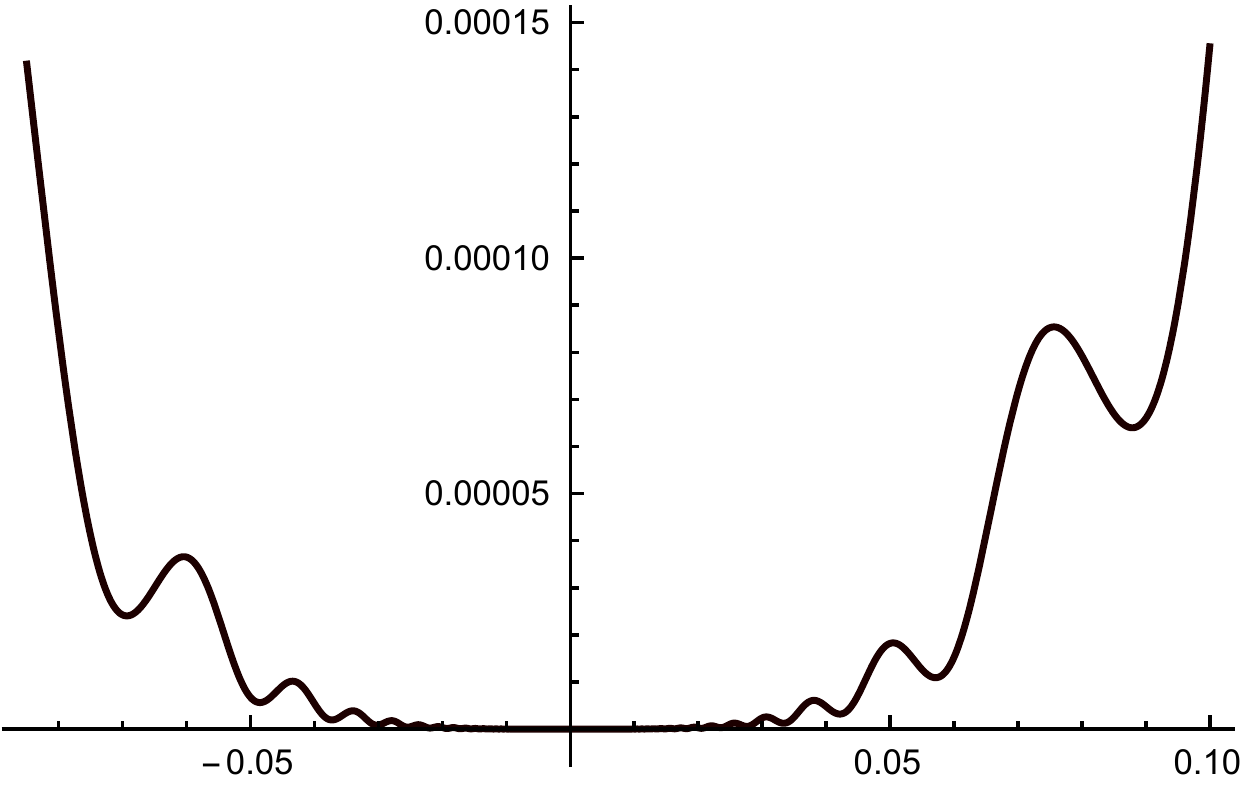} & \includegraphics[width=0.48\textwidth]{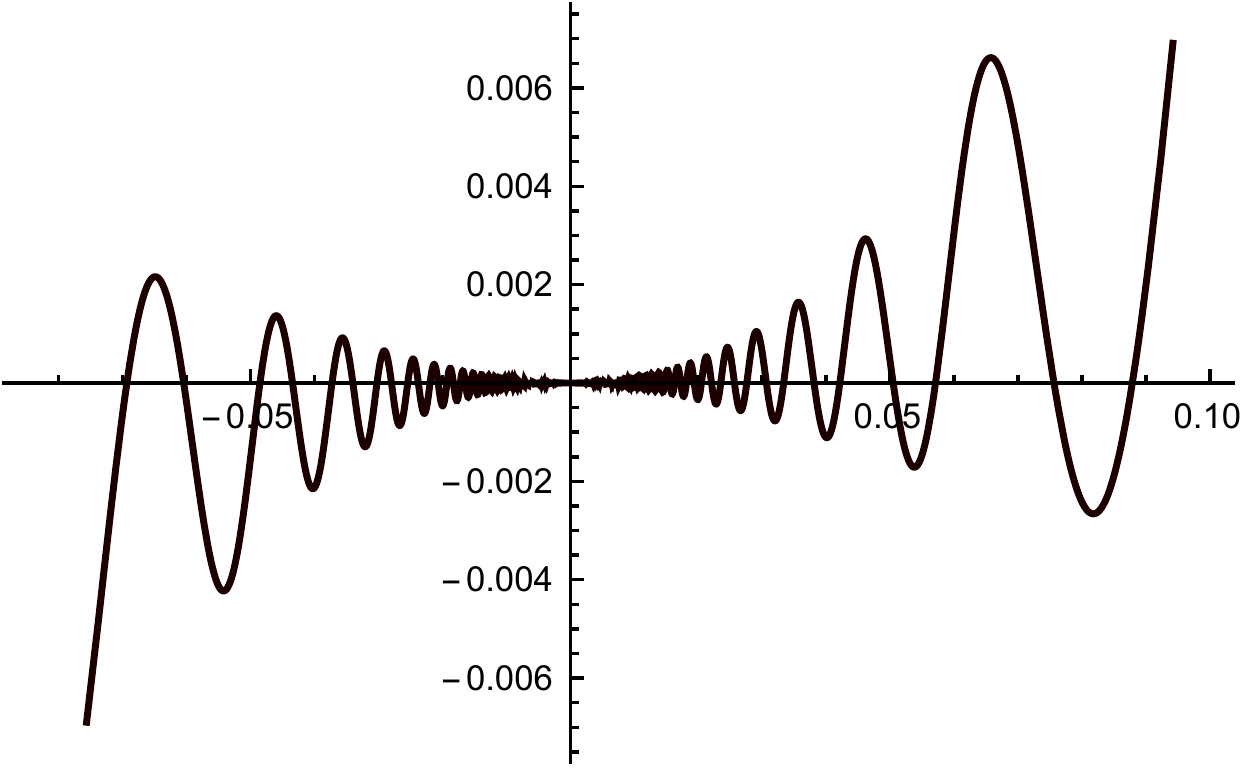}
	\end{tabular}
	\caption{The graphs of $\varphi$ and $\varphi'$.}\label{f1}
\end{figure}

The second example here (see also \cite{dover2003}) shows that a derivative of a differentiable function 
$\eta\colon\R\to\R$ need not satisfy the Extrema Value Theorem (even though it satisfies the Intermediate Value Theorem). 
For instance, this is the case for $\eta(x)\coloneqq e^{-3x}h(x)$, where $h$ is given by (\ref{D1notC1}), that is, 
\begin{equation*}
	\eta(x)\coloneqq
	\begin{cases}
		e^{-3x} x^2 \sin (x^{-1})  & \mbox{ for $x\neq 0$,}\\
		0 & \mbox{ for $x=0$,}\\
	\end{cases}
\end{equation*}
see Fig.~\ref{f3}.

The Extrema Value Theorem fails for $\eta'$, since $\inf\eta'[[0,1/3]]=-1\notin\eta'[[0,1/3]]$.
To see this, note that  
$\eta'(x)=-3e^{-3x} h(x) + e^{-3x} h'(x)$, that is,  
$\eta'(0)=0$ and
$$\eta'(x)=e^{-3x}\left(-3x^2 \sin (x^{-1})+ 2x\sin (x^{-1})- \cos (x^{-1})\right)$$ for $x\neq 0.$

In particular, 
on $(0,1/3]$, we have
$$|\eta'(x)|\leq e^{-3x}\left(3x^2 + 2x+1\right)\leq e^{-3x}\left(1 + 3x\right)<1,$$
where the last inequality holds since, by Taylor formula,  
$$e^{3x}=\sum_{i=0}^\infty \frac{(3x)^i}{i!}>1+3x.$$
This justify $-1\notin\eta'([0,1/3])$. 
As the same time, 
for $x_n\coloneqq (2\pi n)^{-1}$, we have $\lim_{n\to\infty} \eta'(x_n)=\lim_{n\to\infty} -e^{-x_n^2}=-1$, that is, 
indeed $\inf\eta'([0,1/3])=-1$.

\begin{figure}[h!]
	\centering
	\begin{tabular}{cc}
		\includegraphics[width=0.45\textwidth]{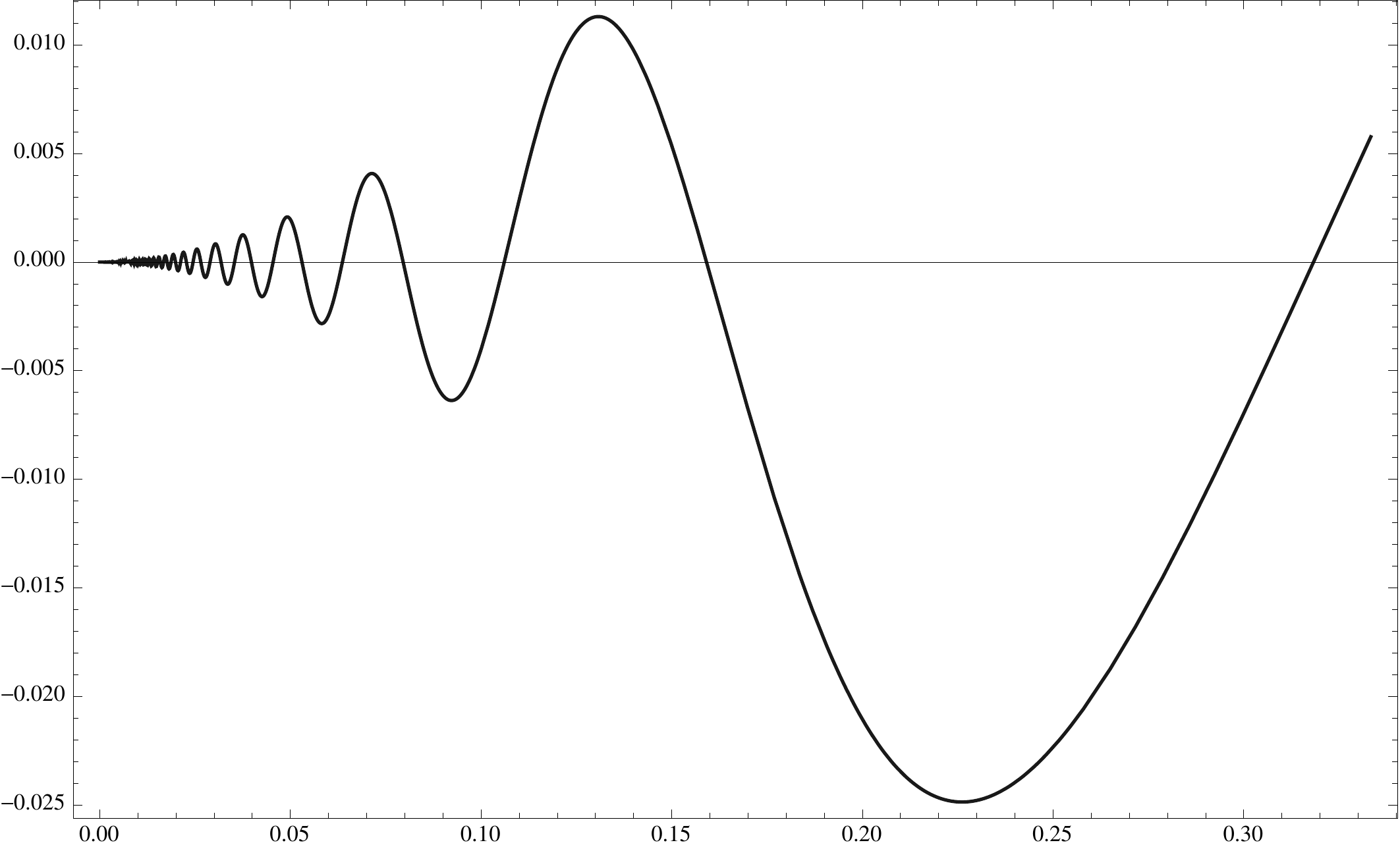} & \includegraphics[width=0.45\textwidth]{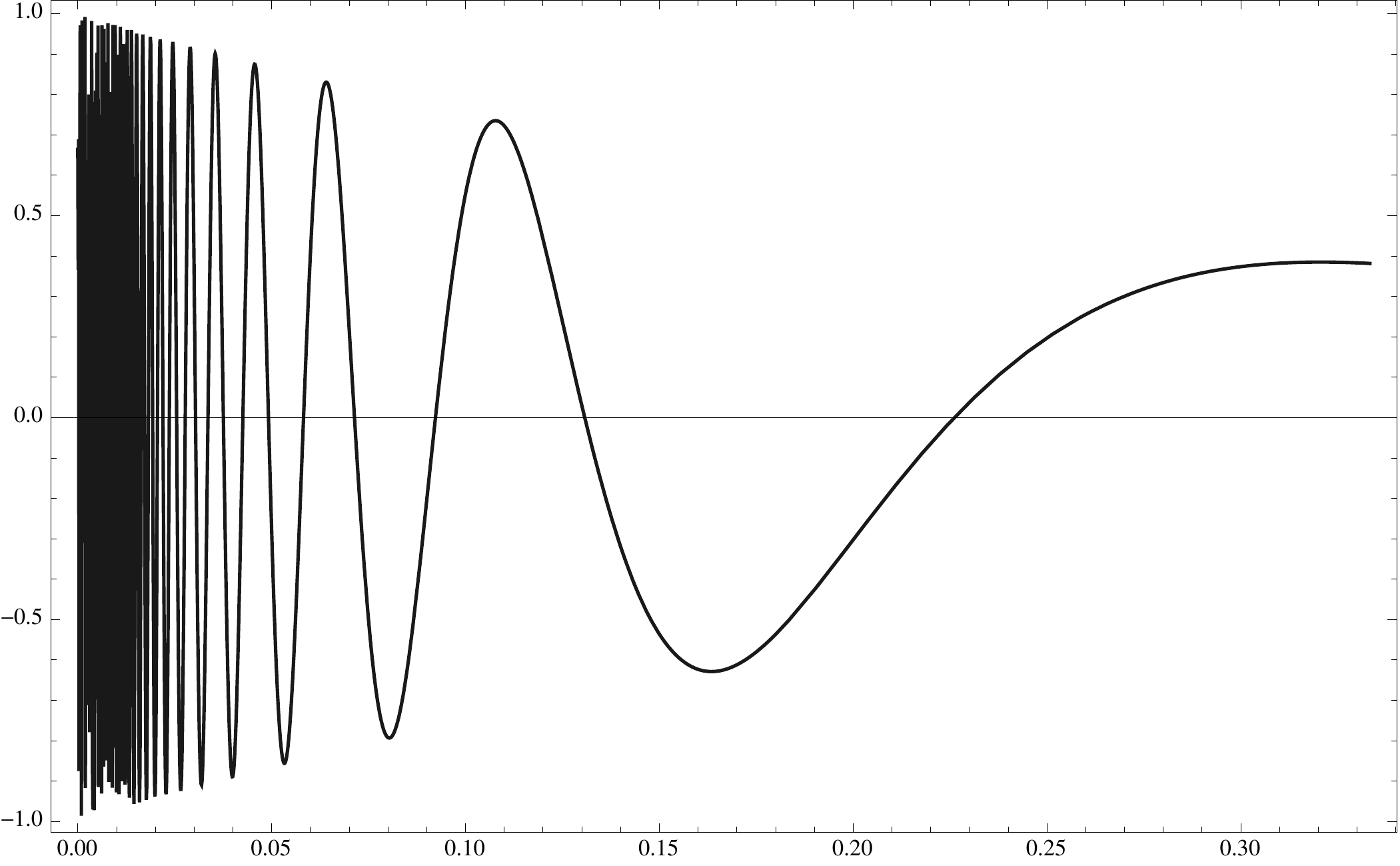}
	\end{tabular}
	\caption{Graphs of $\eta$ and $\eta'$.}\label{f3}
\end{figure}

\medskip

This last example, $\eta$, can be pushed even further, as shown by A. W. Bruckner in \cite[\S VI, Theorem 3.1]{Bru1978}:
there are derivatives defined on a closed interval that achieve no local extrema (i.e., maximum or minimum). 
However, their constructions is considerably more complicated. 
Many more \textit{pathologies} of this kind can be found in the excellent monograph \cite{dover2003}. 
See also \cite{bams2014,dover2003,amm2014,book2016,enfloseoane,carielloseoane,strangefunctions3,strangefunctions2} 
for a long list  of pathologies enjoyed 
by continuous and differentiable functions in $\mathbb{R}^\mathbb{R}$. 
 
It is appropriate to finish this section 
recalling that there is no characterization of the derivatives simpler than the trivial one: 
\begin{quote}
{\em $f$ is a derivative if, and only if, there exists an 
$F$ for which $f=F'$.} 
\end{quote}
Perhaps, a simpler characterization can never be found. 
(See, e.g., \cite{Freiling}.)

\section{Differentiability from continuity} 
\label{sec:DfromC}

In this section we will address the question 
\begin{center}
Q2: {\em How much differentiability does continuity imply?} 
\end{center}
We will argue, in spite of what the next example shows, that
actually every continuous function has some traces of differentiability. 

\begin{figure}[h!]
	\centering
	\begin{tabular}{cccc}
		\includegraphics[width=0.3\textwidth]{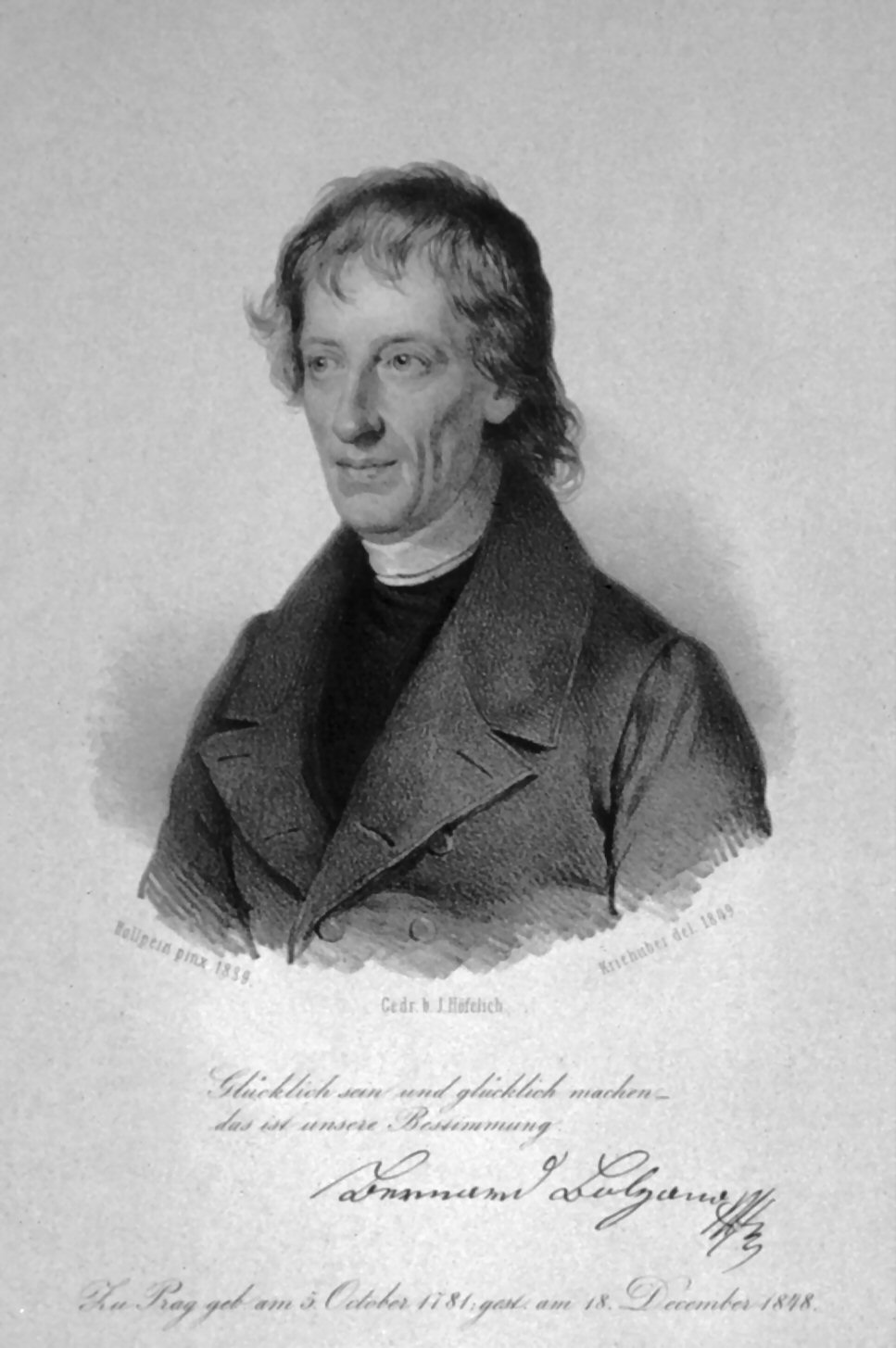} \hspace{.75cm} & \hspace{.75cm} \includegraphics[width=0.3\textwidth]{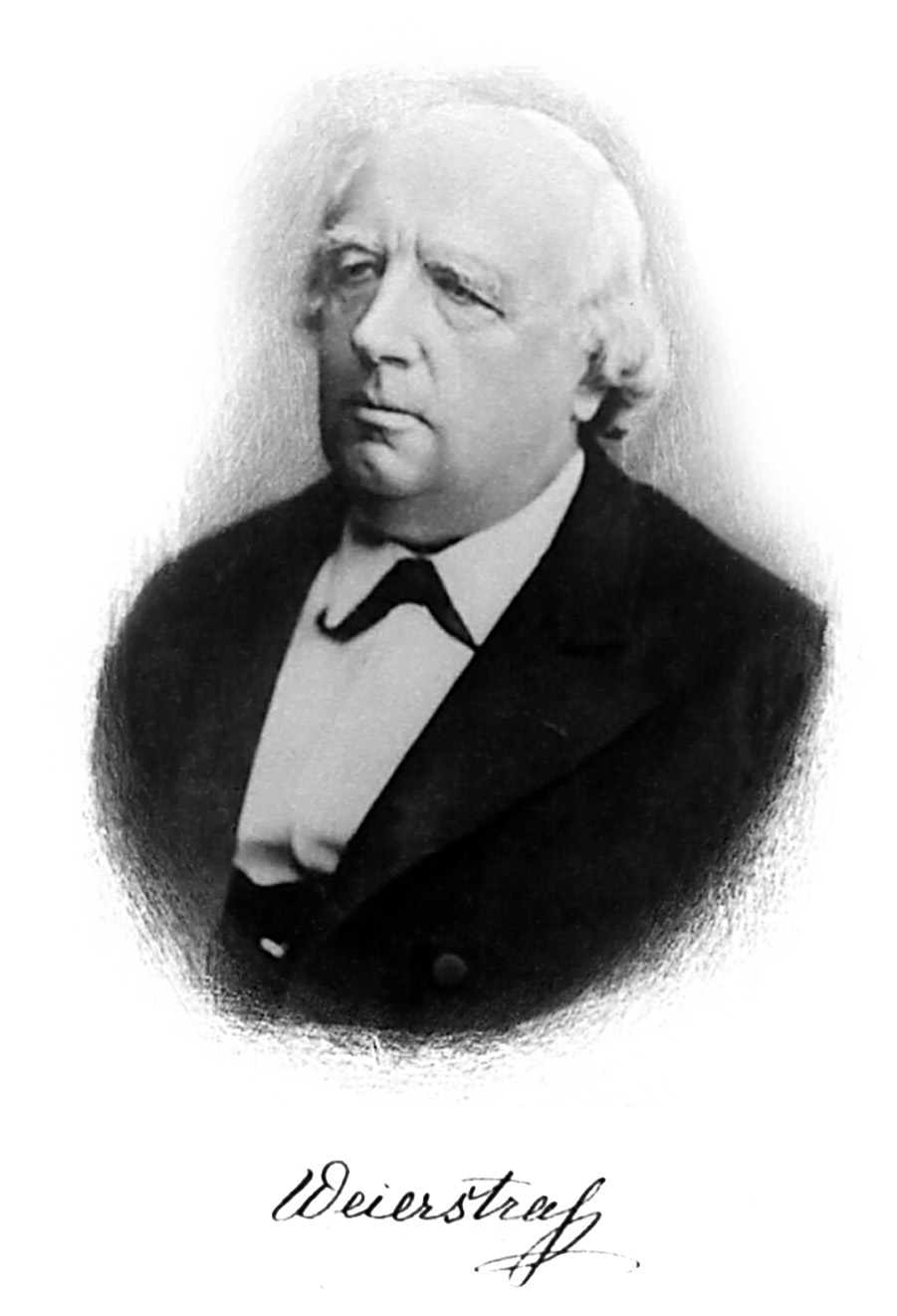}
	\end{tabular}
	\caption{Bernard Bolzano and Karl Theodor Wilhelm Weierstrass.}\label{pic_BW}
\end{figure}

\subsection{Weierstrass monsters} 
At a first glance, an answer to the question Q2 should be 
{\em none}, since there exist  
continuous functions $f\colon\R\to\R$ that are differentiable at no point $x\in \R$.

Although the first known example of this kind dates back to Bernard Bolzano (1781-1848) in 1822 (see, e.g., \cite{Bolzano1822}), the first published example of such an $f$  appeared in 1886 paper~\cite{Wei} (for the English translation see \cite{Edgar2004}) by Karl Weierstrass (1815--1897). The example, defined as
$$
W(x)\coloneqq \sum^\infty_{n=0} \frac{1}{2^n} \cos(13^{n}\pi x),
$$
was first presented by Weierstrass to Prussian Academy of Sciences in 1872. 
At that time, mathematicians commonly believed that a continuous function must have a derivative at a ``significant'' set of points.
Thus, the example came as a general {\em shock}\/  to the audience and was received with such disbelief
that  continuous nowhere differentiable functions became known as {\em Weierstrass's monsters}. 

% % % % % % % monster here
\begin{figure}[h!]
	\centering
	\includegraphics[height=.4\textwidth,keepaspectratio=true]{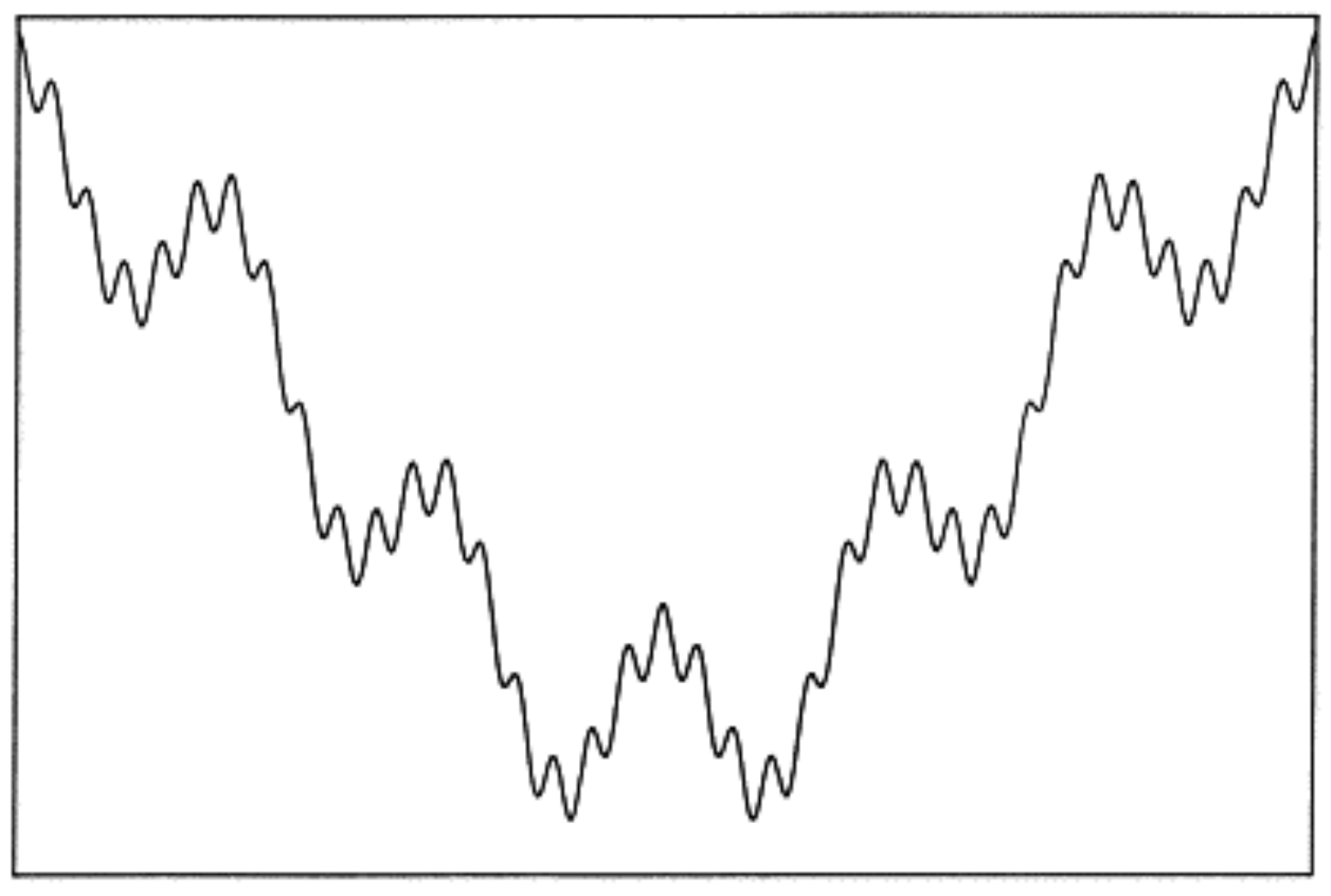}
	\caption{A sketch of the famous Weierstrass' Monster.}
	\label{figure}
\end{figure}

A large number of simple constructions of Weierstrass's monsters 
have appeared in the literature; see, for example, \cite{JP,JMSBBMS,thim2003}.
Our favorite is the following variant of a 1930 example of Bartel Leendert van der Waerden  (1903--1996)~\cite{vW}
(compare \cite[Thm. 7.18]{Rudin}), described already in a 1903
paper~\cite{Takagi} of Teiji Takagi (1875--1960), 
since the proof of its properties requires only the standard tools of one-variable calculus: $f(x)\coloneqq \sum_{n=0}^\infty 4^n f_n(x)$, where 
$f_n(x)\coloneqq \min_{k\in\Z}\left|x-\frac{k}{8^{n}}\right|$ is the distance from $x\in\R$ to the set 
$\frac{1}{8^{n}}\Z=\{\frac{k}{8^{n}}\colon k\in\Z\}$. 
(See Fig.~\ref{figure}.)
It is continuous at each 
$x\in \R$, since 
$$|f(y)-f(x)|\leq \left|\sum_{i=0}^n 4^i f_i(y)-\sum_{i=0}^n 4^i f_i(x)\right|+\frac{1}{2^{n}}$$ for every $y\in\R$ and $n\in\N$. 
It is not differentiable at any 
$x\in\R$, since for every $n\in\N$ and $k\in \Z$ with $x\in\left[\frac{k}{8^n},\frac{k+1}{8^n}\right]$
there exists a $y_n\in \left\{\frac{k}{8^n},\frac{k+1}{8^n}\right\}\setminus\{x\}$ 
such that
\begin{align*}
\left|\frac{f(x)-f(y_n)}{x-y_n}\right| & \ge \left|\frac{f(\frac{k+1}{8^n})-f(\frac{k}{8^n})}{\frac{k+1}{8^n}-\frac{k}{8^n}}\right|\\
& = 8^n \left|\sum_{i=0}^n 4^i f_i(\frac{k+1}{8^n})-\sum_{i=0}^n 4^i f_i(\frac{k}{8^n})\right|
\geq \frac 23 4^{n-1}.
\end{align*}

\begin{figure}[h!]
	\centering
	\begin{tabular}{cccc}
		\includegraphics[width=0.3\textwidth]{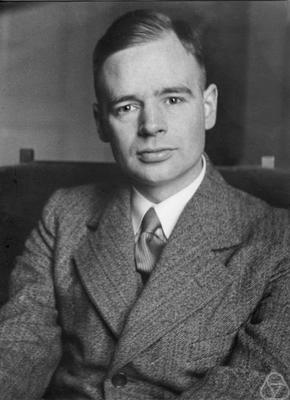} \hspace{.75cm} & \hspace{.75cm}\includegraphics[width=0.3277\textwidth]{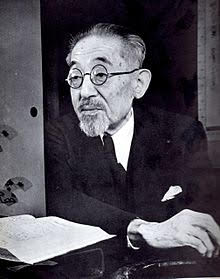}
	\end{tabular}
	\caption{Bartel Leendert van der Waerden and Teiji Takagi.}
	\label{pic_LT}
\end{figure}

\rem{Rem:W}{In fact, any typical continuous function\footnote{Term {\em typical}\/ is used the Baire category sense, that is, it holds for 
every map from a co-meager set $G\subset C([a,b])$ of functions.} 
is a Weierstrass monster, see e.g. \cite{Munk}. 
Actually, a typical
continuous function 
agrees with continuum many constant functions on perfect sets. Also, in \cite{BrucknerGarg}, the authors studied the structure of the sets in which the graphs of a residual set of continuous functions intersect with different straight lines. They showed that any {\em typical} continuous functions $f$ has a certain structure with respect to intersections with straight lines. In particular, all but at most two of the constant functions that agree anywhere with $f$, agree on uncountable sets (and the same also holds with lines in all non-vertical directions).}

\subsection{Lipschitz and differentiable restrictions}

The goal of this section is to prove the following 1984 theorem
of Mikl\' os Laczkovich \mbox{(1948--)} from~\cite{La1984}. 

\thm{thML}{For every continuous $f\colon \R\to\R$ there is a perfect 
set $Q\subset\R$ such that $f\restriction Q$ is differentiable.}

In the statement of the theorem, 
the differentiability of $h\coloneqq f\restriction Q$ is understood as existence of its derivative: 
the function $h'\colon Q\to\real$ where 
$$
h'(p)\coloneqq\lim_{x\to p,\ x\in Q}\frac{h(x)-h(p)}{x-p}
$$
for every $p\in Q$.
The proof of this theorem presented below comes from~\cite{KCprep2017}.
Its advantage over the original proof is that none of its steps require the tools of Lebesgue measure/integration theory. 

Before we turn our attention to the proof, we like to make two remarks, which lay outside of the main narrative here
(as they use Lebesgue measure tools), but give a better perspective on Theorem~\ref{thML}.
The first one is a more general version of the theorem, in the form presented in the original paper~\cite{La1984}.

\rem{rem1}{For every $E\subset \R$ of positive Lebesgue measure and every 
	continuous function $f\colon E\to\R$ there exists a perfect $Q\subset E$ such that $f\restriction Q$ is differentiable.
}

It is worth to note, that the proof we present below can be easily adapted to deduce also the statement from Remark~\ref{rem1}. 

The next remark shows that, in general, we cannot expect that the set $Q$ in Theorem~\ref{thML} can be either of positive Lebesgue
measure or of second category. 

\rem{rem2}{There exists a continuous function $f\colon \R\to\R$ such that $f\restriction Q$ can be differentiable only when
	$Q$ is both first category and of Lebesgue measure 0. 
}

A function $f\colon [0,1]\to\R$ with such property can be chosen as one of the coordinates of the classical 
Peano curve that maps continuously $[0,1]$ onto $[0,1]^2$. Indeed, such coordinates are nowhere approximately and $\mathcal I$-approximately differentiable,
as proved in 1989 by the first author, %Krzysztof C. Ciesielski
Lee Larson, and Krzysztof Ostaszewski~\cite{CLO1990}. (See also \cite[example 1.4.5]{CLObook}.)
Clearly, such $f$ cannot have a differentiable restriction to a set $Q$ of positive measure (second category)
since then $f$ would be approximately ($\mathcal I$-approximately) differentiable
at any density ($\mathcal I$-density) point of $Q$.

The main tool in the presented proof of Theorem~\ref{thML} is the following result, which is of independent interest. 
Notice that Theorem~\ref{thm:L} follows also easily from the interpolation 
Theorem~\ref{Thm:Inter}, whose original proof was not using 
Theorem~\ref{thm:L}.\footnote{If $g\colon\R\to\R$ is a $C^1$ function from Theorem~\ref{Thm:Inter}, then,
	by the Mean Value Theorem, its restriction $g\restriction [-M,M]$ is Lipschitz for every $M>0$.
	Thus, for some $M>0$, the set $P\coloneqq[f=g]\cap[-M,M]$ is perfect non-empty
	and $f\restriction P=g\restriction P$ is Lipschitz. 
}
However, our proof of Theorem~\ref{Thm:Inter} uses Theorem~\ref{thm:L}.

\thm{thm:L}{Assume that $f\colon\R\to\R$ is monotone and continuous on a non-trivial interval $[a,b]$.
	For every $L> \left| \frac{f(b)-f(a)}{b-a}\right|$ there exists a closed uncountable %perfect 
	set $P\subset [a,b]$ such that $f\restriction P$ is Lipschitz with constant $L$.
}

The difficulty in proving Theorem~\ref{thm:L} without measure theoretical tools comes from the fact
that there exist functions like Pompeiu map from Proposition~\ref{pr:Pom}: 
strictly increasing continuous maps $f\colon\R\to\R$
which posses finite or infinite derivative at every point, but such that the derivative of $f$ is infinite on a dense $G_\delta$-set.
These examples show that a perfect set in Theorem~\ref{thm:L} should be nowhere sense. 
Thus, we will use a measure theoretical approach, in which the measure theoretical tools will be present only implicitly
or, as in case of Fact~\ref{fact1}, given together with a simple proof. 

\begin{figure}[h!]
	\centering
	\begin{tabular}{cc}
		\includegraphics[width=0.37\textwidth]{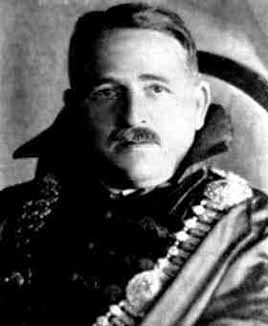} \hspace{.3cm} & \hspace{.3cm} \includegraphics[width=0.3\textwidth]{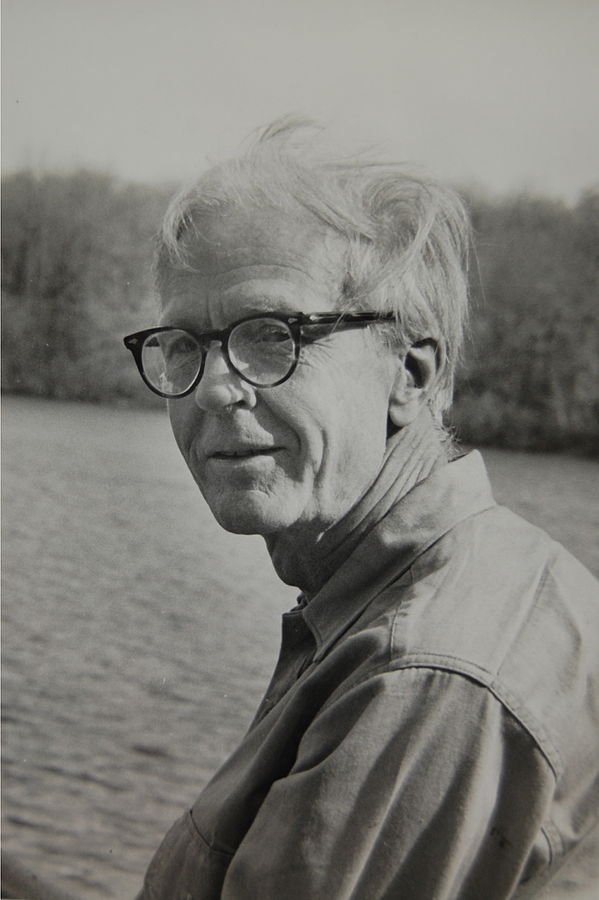}
	\end{tabular}
	\caption{Frigyes Riesz and Hassler Whitney.}
	\label{pic_RW}
\end{figure}

\begin{figure}[h!]
	\centering
	\includegraphics[width=.6\textwidth]{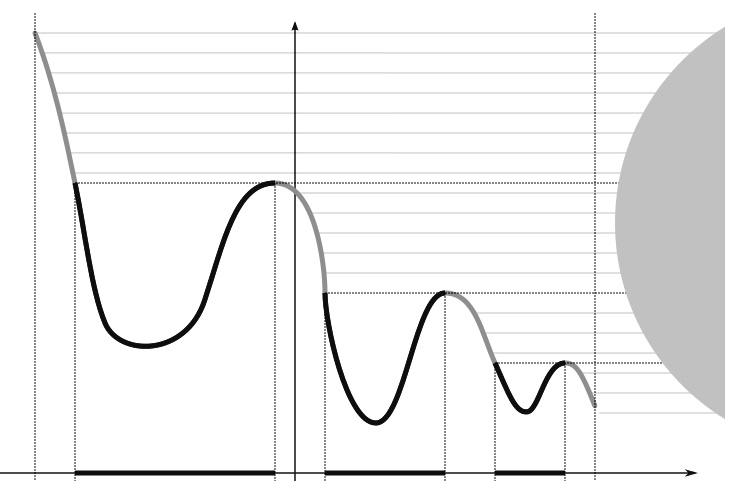}
	\caption{Illustration of the Rising Sun Lemma of Frigyes Riesz (Lemma \ref{RSL}). The colorful name of this lemma comes from imagining the graph of the function as a mountainous landscape, having the sun shining horizontally from the right. The points in the set $U \cap (a,b)$ are those lying in the shadow.}\label{rising-sun}
\end{figure}

The presented proof of Theorem~\ref{thm:L} is extracted from a proof of a Lebesgue theorem
that every monotone function $f\colon\R\to\R$ is differentiable almost everywhere.
It is based on the following 1932 result of Frigyes Riesz (1880-1956) from~\cite{Riesz}, known as 
the rising sun lemma, see Fig. \ref{rising-sun}. 
See also \cite{Tao}.

\lem{lem2}{If $g$ is a continuous function from a non-trivial interval $[a,b]$ into $\R$, 
	then the set 
	\mbox{$U\coloneqq\{x\in[a,b)\colon g(x) < g(y) \mbox{ for some }y \in (x,b]\}$} is open in $[a,b)$ and
	$g(c)\leq g(d)$ for every connected component  $(c,d)$ of $U$. \label{RSL}
}

\begin{proof}
	It is clear that $U$ is open in $[a,b)$. To see the other part, let $(c,d)$ be a component of $U$. 
	By continuity of $g$, 
	it is enough to prove that $g(p)\leq g(d)$ for every $p\in(c,d)$.
	Assume by way of contradiction that $g(d)<g(p)$ for some $p\in(c,d)$
	and let $x\in [p,b]$ be a point at which $g\restriction [p,b]$ achieves the maximum.
	Then $g(d)<g(p)\leq g(x)$ and so we must have
	$x\in [p,d)\subset U$, as otherwise $d$ would belong to $U$. 
	But $x\in U$ contradicts the fact that $g(x)\geq g(y)$ for every $y\in (x,b]$.
\end{proof}

For an interval $I$ let $\ell(I)$ be its length. We need 
the following simple well-known observations. 

\begin{Fact}\label{fact1}
Let $a<b$ and $\J$ be a family of open intervals with $\bigcup\J\subset(a,b)$. 
	\begin{itemize}
		\item[(i)] If $[\alpha,\beta]\subset\bigcup\J$, then  $\sum_{I\in\J} \ell(I)> \beta-\alpha$.
		\item[(ii)] If the intervals in $\J$ 
		are pairwise disjoint, 
		then $\sum_{I\in\J} \ell(I)\leq b-a$.
	\end{itemize}
\end{Fact}

\begin{proof}
	(i): By compactness of $[\alpha,\beta]$ we can assume that $\J$ is finite, say of 
	size $n$. 
	Then (i) follows by an easy induction on $n$: if $(c,d)=J\in \J$ contains $\beta$, then either $c\leq \alpha$,
	in which case (i) is obvious, or $\alpha<c$ and, by induction,
	$\sum_{I\in\J} \ell(I)=\ell(J)+\sum_{I\in\J\setminus\{J\}} \ell(I)> \ell([c,\beta])+\ell([\alpha,c])=\beta-\alpha$.
	
	(ii): Once again, it is enough to show (ii) for finite $\J$, say of size $n$, by induction. 
	Then, there is $(c,d)=J\in\J$ to the right of any $I\in \J\setminus\{J\}$.
	So, by induction, 
	$\sum_{I\in\J} \ell(I)=\ell(J)+\sum_{I\in\J\setminus\{J\}} \ell(I)\leq (b-c)+(c-a)=b-a$.
\end{proof}

\begin{proof}[Proof of Theorem~\ref{thm:L}] 
	If there exists a nontrivial interval $[c,d]\subset [a,b]$ on which
	$f$ is constant, then clearly $P\coloneqq [c,d]$ is as needed. 
	Thus, we can assume that $f$ is strictly monotone on $[a,b]$.  
	Also, replacing $f$ with $-f$, if necessary, we can also assume that $f$ 
	is strictly increasing. 
	
Fix  $L>|q(a,b)|=\frac{f(b)-f(a)}{b-a}$ and define $g\colon \R\to\R$ as $g(t)\coloneqq f(t)-Lt$. 
Then $g(a)=f(a)-La>f(b)-Lb=g(b)$.

Let $m\coloneqq \sup\{g(x)\colon x\in[a,b]\}$ 
and $\bar a\coloneqq\sup\{x\in[a,b]\colon g(x)=m\}$.
Then $f(\bar a)-L\bar a=g(\bar a)\geq g(a)>g(b)=f(b)-Lb$, 
so $a\leq \bar a<b$ and we still have $L>|q(\bar a,b)|=\frac{f(b)-f(\bar a)}{b-\bar a}$.
Moreover, $\bar a$ does not belong to the set 
	\[
	U\coloneqq\{x\in[\bar a,b)\colon g(y)> g(x) \mbox{ for some }y \in (x,b]\}
	\]
	from Lemma~\ref{lem2} applied to $g$ on $[\bar a,b]$. 
	In particular, $U$ is open in $\R$
	and the family $\J$ of all connected components of $U$
	contains only open intervals $(c,d)$ for which, by Lemma~\ref{lem2}, $g(c)\leq g(d)$. 
	
	The set $P\coloneqq[\bar a, b]\setminus U\subset[a,b]$ is closed and for any $x<y$ in $P$
	we have $f(y)-Ly=g(y)\leq g(x)=f(x)-Lx$, that is, $|f(y)-f(x)|=f(y)-f(x)\leq Ly-Lx=L|y-x|$. 
	In particular, $f$ is Lipschitz on $P$ with constant $L$. 
	It is enough to show that $P$ is uncountable.
	
	In order to see this notice that for every $J\coloneqq(c,d)\in\J$ we have $f(d)-Ld=g(d)\geq  g(c)=f(c)-Lc$, that is,
	$\ell(f[J])=f(d)-f(c)\geq L(d-c)=L\ell(J)$. 
	Since the intervals in the family \mbox{$\J^*\coloneqq\{f[J]\colon \J\in\J\}$} 
	are pairwise disjoint and contained in the interval $(f(\bar a),f(b))$, by Fact~\ref{fact1}(ii) we have 
	$\sum_{J^*\in\J^*} \ell(J^*)\leq f(b)-f(\bar a)$.
	Thus,
	$$\sum_{J\in\J} \ell(J)\leq \frac{1}{L}\sum_{J\in\J} \ell(f[J])= \frac{1}{L}
	\sum_{J^*\in\J^*} \ell(J^*)\leq \frac{f(b)-f(\bar a)}{L}<b-\bar a.$$
	Therefore, by Fact~\ref{fact1}(i),
	$P\coloneqq[\bar a,b]\setminus U=[\bar a,b]\setminus\bigcup\J\neq\emptyset$. 
	However, we need more, that $P$ cannot be contained in any countable set, say $\{x_n\colon n\in\N\}$.
	To see this, fix $\delta>0$ such that $\frac{f(b)-f(\bar a)}{L}+\delta<b-\bar a$, 
	for every $n\in \N$ choose an interval $(c_n,d_n)\ni x_n$ of length $2^{-n}\delta$,
	and put $\hat\J\coloneqq\J\cup\{(c_n,d_n)\colon n<\omega\}$.
	Then  
	$$
	\sum_{J\in\hat\J} \ell(J)=\sum_{J\in\J} \ell(J)+\sum_{n\in\N} \ell((c_n,d_n))\leq\frac{f(b)-f(\bar a)}{L}+\delta<\beta-\alpha
	$$
	so, by Fact~\ref{fact1}(i), 
	$U\cup\bigcup_{n\in\N}(c_n,d_n)\supset U\cup \{x_n\colon n\in\N\}$ does not contain $[\bar a,b]$.
	In other words, $P=[\bar a,b]\setminus U$ is uncountable, as needed. 
\end{proof}

The last step needed in the proof of Theorem~\ref{thML} is the following proposition, in which 
$\Delta\coloneqq\{\la x,x\ra\colon x\in\R\}$. 

\prop{propMain}{For every continuous 
	$f\colon\R\to\R$
	there exists a perfect set $Q\subset\R$ such that the quotient map $q\colon\R^2\setminus \Delta\to \R$,
	$q(x,y)\coloneqq\frac{f(x)-f(y)}{x-y}$, restricted to 
	$Q^2\setminus \Delta$ 
	is bounded and uniformly continuous. 
}

\begin{proof}
	If $f$ is monotone on some non-trivial interval $[a,b]$,
	then, by Theorem~\ref{thm:L}, there exists a perfect set $P\subset\R$ such that
	$f\restriction P$ is Lipschitz with some constant $L\in[0,\infty)$.
	In particular, the values of $q\restriction P^2\setminus\Delta$ are in the bounded interval $[-L,L]$.
	Therefore, 
	by a theorem of Micha{\l} Morayne \mbox{(1958--)} from \cite{Morayne85} applied to $F\coloneqq q\restriction P^2\setminus\Delta$,
	there exists a perfect $Q\subset P$ for which $F\restriction Q^2\setminus \Delta$ is uniformly continuous. 
	In such case, $q\restriction Q^2\setminus \Delta=F\restriction Q^2\setminus \Delta$ is clearly 
	bounded and uniformly continuous.
	
	On the other hand, if 
	$f$ is monotone on no non-trivial interval, then, by a 1953 theorem of Komarath Padmavally~\cite{Padm}
	(compare also \cite{Mina1940,Markus1958,Grag1963}),
	there exists 
	a perfect set $Q\subset\R$ on which $f$ is constant. 
	Of course, the quotient map on such $Q$ is as desired.
\end{proof}

Note that the results from papers~\cite{Padm} and~\cite{Morayne85}, which we used above,  
have simple topological proofs that  do not
require any measure theoretical tools. 

\begin{proof}[Proof of Theorem~\ref{thML}] 
	Let $Q\subset\R$ be as in Proposition~\ref{propMain}.
	Then the uniformly continuous $q\restriction Q^2\setminus \Delta$ can be extended
	to the uniformly continuous $\bar q$ on $Q^2$.
	Therefore, 
	for every $x\in Q$, the limit 
	$$\lim_{y\to x, y\in Q} \frac{f(y)-f(x)}{y-x}=	\lim_{y\to x, y\in Q} q(y,x)=\bar q(x,y)$$
	is well defined and equal to the derivative of $f\restriction Q$ at $x$. 
\end{proof}

\subsection{Differentiable extension: Jarn\'\i k and Whitney theorems} 

By Theorem~\ref{thML}, any continuous functions $f\colon \R\to\R$ has a differentiable restriction to a perfect set.
In the next subsection we will prove an even stronger theorem, that there always exists
a $C^1$ function $g\colon \R\to\R$ which agrees with $f$ on an uncountable set. 
The main tool proving this result is the Whitney's $C^1$ extension theorem, proved
in the 1934 paper~\cite{Wh} of Hassler Whitney (1907-1989). 
We will use this result in form of Theorem~\ref{thmMC&KC}  from a
recent paper of Monika Ciesielska and the first author~\cite{MC&KC} (check also \cite{KCprep2017}), 
since we can sketch here its elementary proof. 
(See also \cite{KC&JS2018} for another simple proof of the Whitney's $C^1$ extension theorem.)
To state it, we need the following notation. 
For a bounded open interval $J$ let $I_J$ be the closed middle third of $J$
and, given a perfect set $Q\subset \R$, we let 
\[
\hat Q\coloneqq Q\cup\bigcup\left\{I_J\colon J \mbox{ is a bounded connected component of } \R\setminus Q\right\}.
\]

\thm{thmMC&KC}{Let $f\colon Q\to\R$, where $Q$ is a perfect subset of $\R$.
	\begin{itemize}
		\item[(a)] If $f$ is differentiable, then there exists a differentiable extension $F\colon\R\to\R$ of $f$.
		\item[(b)] If  $\hat f\coloneqq\bar f\restriction \hat Q$, where $\bar f\colon \R\to\R$ is a linear interpolation of $f$, then
		$\hat f$ is differentiable and 
		\[
		\mbox{$f$ admits a $C^1$ extension if, and only if, $\hat f$ is continuously differentiable.
		}
		\]
	\end{itemize}
}

Part (b) of Theorem~\ref{thmMC&KC} gives a criteria, in term of $\hat f$, on admission of a $C^1$ extension of $f$.
In particular, it can be viewed as a version of the ($C^1$-part of) Whitney's extension theorem
(for functions of one variable). 

\begin{figure}[h!]
	\centering
		\includegraphics[width=0.35\textwidth]{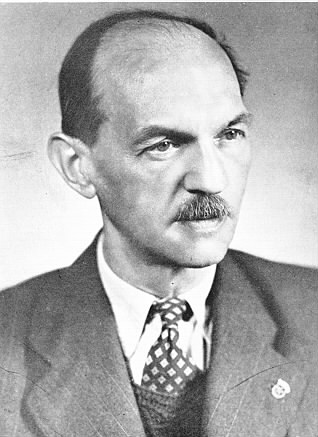}
	\caption{Vojt\v ech Jarn\'\i k.}
	\label{pic_jarnik}
\end{figure}

Part (a) of Theorem~\ref{thmMC&KC} has long and interesting story. 
It first appeared in print in a 1923 paper~\cite{Jarnik} of Vojt\v ech Jarn\'\i k (1897--1970), for the case when $Q\subset\R$ is 
compact. Unfortunately, \cite{Jarnik} appeared in the not so well known journal 
{\em Bull. Internat. de l'Acad\' emie des Sciences de Boh\^eme},
was written in French, and it only sketched the details of the construction.
A more complete version of the proof, that appeared in 
\cite{JarnikCzech}, was written in Czech and was even less readily accessible.
Therefore, this result of Jarn\'\i k was unnoticed by the mathematical community
until the mid 1980's. 
Theorem~\ref{thmMC&KC}(a) was rediscovered by Gy\" orgy Petruska (1941--) and Mikl\'os Laczkovich 
and published in 1974 paper~\cite{PeLa}.
Its proof in~\cite{PeLa} is quite intricate and embedded in a deeper, more general research.
A simpler proof of the theorem appeared in a 1984 paper~\cite{Marik} by Jan Ma\v r\'\i k (1920--1994);  however, it is considerably more complicated  than the one we presented below and it employs Lebesgue integration tools. 
Apparently, the authors of neither~\cite{PeLa} nor~\cite{Marik} had been aware of Jarn\'\i k's paper~\cite{Jarnik}
at the time of publication of their articles. 
However~\cite{Jarnik} is cited in 1985 paper~\cite{ALP} 
%of Aversa, Laczkovich, and Preiss 
that discusses multivariable version of Theorem~\ref{thmMC&KC}(a).
Also, two recent papers \cite{NZ,Koc} that address generalizations of Theorem~\ref{thmMC&KC}(a) mention~\cite{Jarnik}.
(See also \cite{KocKolar2016,KocKolar2017}). 

\begin{figure}[h!]
	\begin{center}
		\includegraphics[width=10.0cm]{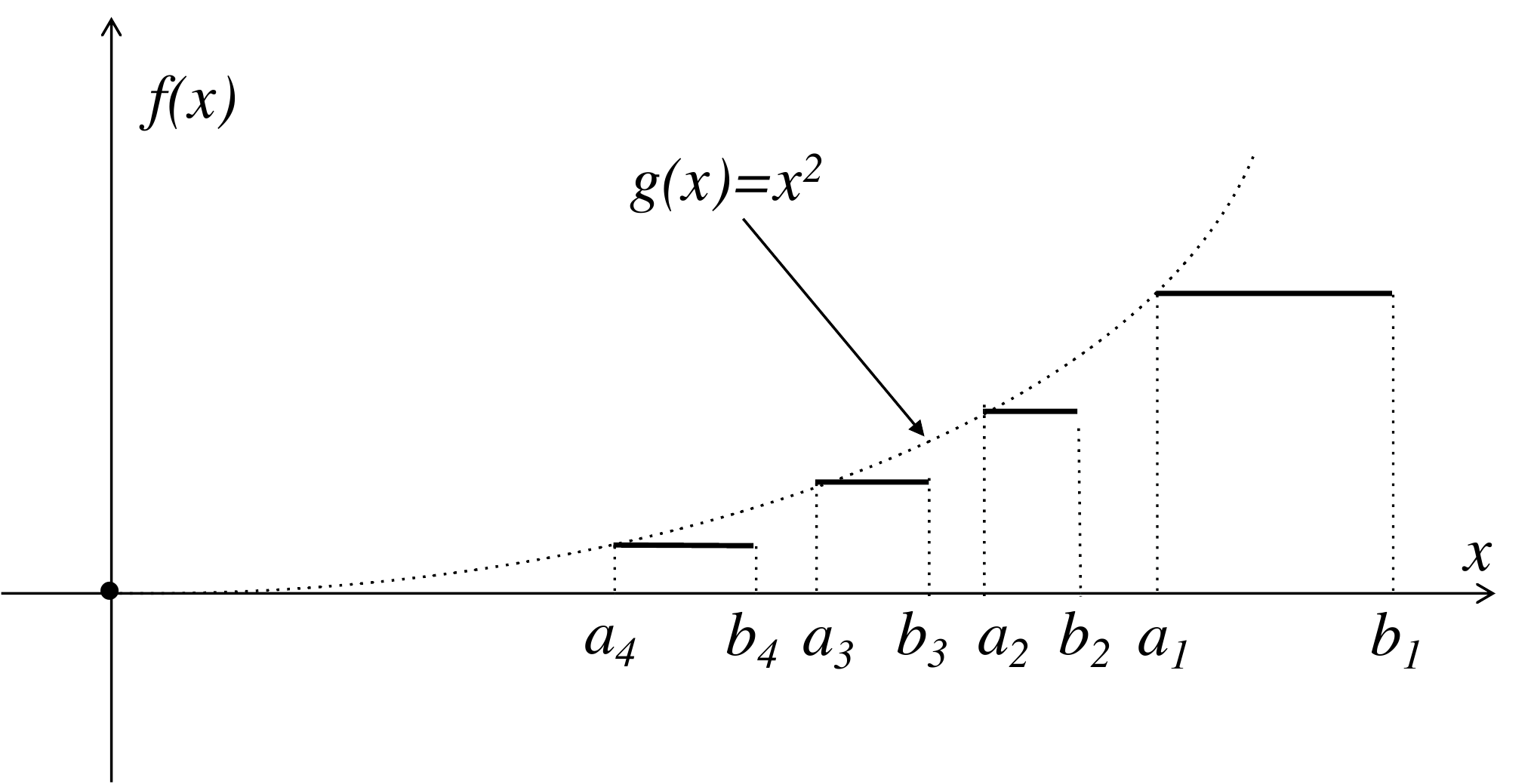}
	\end{center}
	\caption{Graph of 
	$f\colon P\to\R$, horizontal thick segments, with $f'=0$ on $P$.
		No differentiable extension $F\colon\R\to\R$ of $f$ has continuous derivatives,
		unless $\displaystyle \frac{f(a_n)-f(b_{n+1})}{a_n-b_{n+1}}\stackrel{n\to\infty}{\longrightarrow} 0$. 
	}
	\label{3Fig}
\end{figure}

	\begin{figure}[h!]
		\begin{center}
			\includegraphics[width=9.0cm]{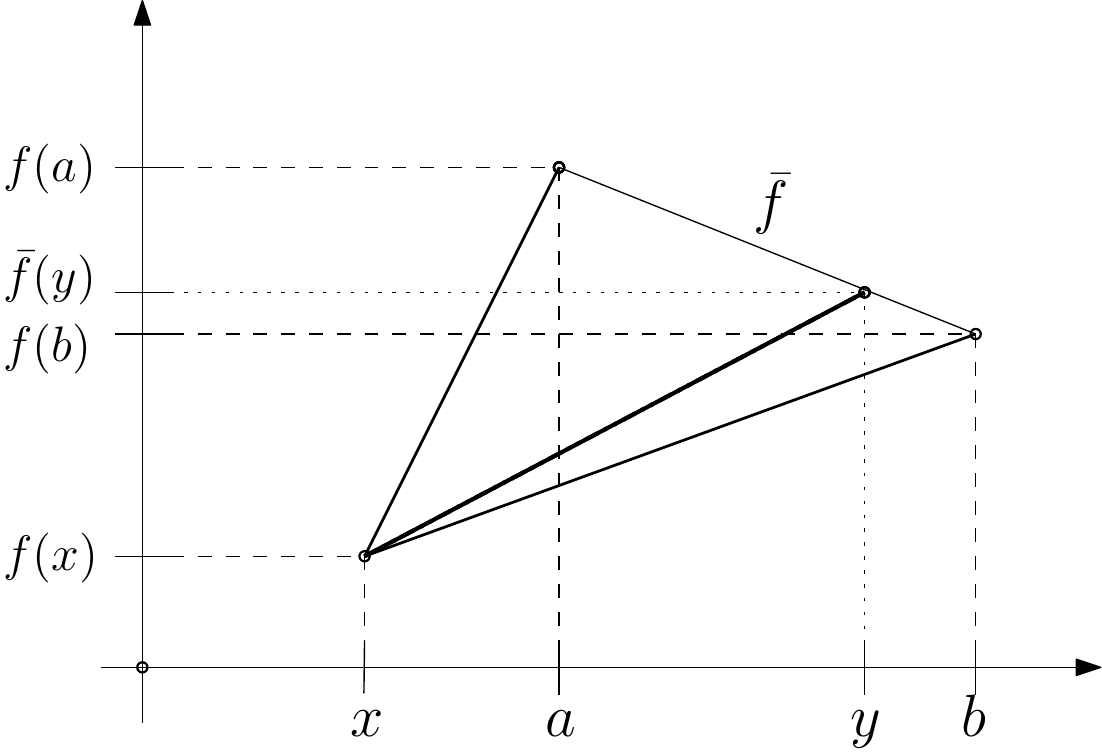}
		\end{center}
		\caption{Illustrating the situation presented in statement \eqref{eq:Est1Diff}.}
	\end{figure}

It is well known and easy to see that 
function $f$ from Jarn\'\i k's theorem need not to admit $C^1$  extension, 
even when $f'$ is constant. See, for example, a map from Fig.~\ref{3Fig}.

Notice also, that Theorem~\ref{thmMC&KC}(a) is actually true for $Q$ being any arbitrary closed subset of $\R$.
Such a version was proved in all cited papers on the subject that appeared after the original works by Jarn\'\i k.
We skip such generality in order to avoid a problem of defining the notion of the derivative for the set with isolated points as well as 
some additional technical issues. But the general theorem can be easily deduced from the version
from Theorem~\ref{thmMC&KC}(a), since for any differentiable function $f$ on a closed subset $P$ of $\R$
there exists a perfect set $Q$ containing $P$ such that a linear interpolation of $f$
restricted to $Q$ is differentiable.

\begin{proof}[Sketch of proof of Theorem~\ref{thmMC&KC}] 
	Extending slightly $f$, if necessary, we can assume that 
	the perfect set $Q$ is unbounded from both sides. 
	Then a linear interpolation $\bar f\colon\R\to\R$ of $f$  is uniquely determined. 
	
	First notice that the unilateral (i.e., one sided) derivatives $D^-\bar f$
	and $D^+\bar f$ of $\bar f$ exists at every point. 
	Indeed, $D^+\bar f(x)$ and $D^-\bar f(x)$ clearly exist for every $x\in \R\setminus Q$. 
	They also exist for every $x\in  Q$, since for every component $I\coloneqq(a,b)$ of $\R\setminus Q$
	with $x\notin[a,b]$, we have that
	\begin{equation}\label{eq:Est1Diff} 
	\frac{|\bar f(y)-\bar f(x)|}{y-x} \text{ lies between } \frac{|f(a)-f(x)|}{a-x} \text{ and } \frac{|f(b)-f(x)|}{b-x}	\quad \forall y \in (a,b),
	\end{equation}
	thus 
	$$
	\mbox{$\left|f'(x)-\frac{|\bar f(y)-\bar f(x)|}{y-x} \right|\leq \max\left\{\left|f'(x)-\frac{|f(a)-f(x)|}{a-x} \right|,\left|f'(x)-\frac{|f(b)-f(x)|}{b-x} \right|\right\}$.}
	$$

	In particular, $\hat f\coloneqq\bar f\restriction \hat Q$ is differentiable,
	$\bar f$ is the linear interpolation of $\hat f$, and $\bar f$ is differentiable 
	at all points $x\in\R$ that do not belong to 
	the set $E_Q$ of all end-points of connected components of $\real\setminus \hat Q$. 
	The function $F$ we are after is defined as $\bar f+g$ for some small adjustor map $g\colon \R\to\R$ 
	such that $g=0$ on $\hat Q$. 
	
	Let $\kappa\leq\omega$ be the cardinality of the family 
	$\J$ of all connected components of $\R\setminus \hat Q$
	and let $\{(a_i,b_i)\colon 1\leq i\leq\kappa\}$ be an enumeration of $\J$.
	Since we assumed that $g=0$ on $\hat Q$, it is enough to define $g$ on each 
	interval in $\J$. 
	
	Therefore, for every $1\leq i\leq\kappa$, define $\ell_i\coloneqq\min\{1,b_i-a_i\}$ and let 
	$\e_i\in(0,3^{-i}\ell_i)$ be such that  
	\begin{itemize}
		\item[(a)] $\displaystyle \left|f'(a_i)-\frac{f(x)-f(a_i)}{x-a_i}\right|<3^{-i}$ for every $x\in P\cap [a_i-\e_i,a_i)$; 
		\item[(b)] $\displaystyle \left|f'(b_i)-\frac{f(x)-f(b_i)}{x-b_i}\right|<3^{-i}$ for every $x\in P\cap (b_i,b_i+\e_i]$.
	\end{itemize}
	Now, define $g$ on $(a_i,b_i)$ 
	as $g(x)\coloneqq\int_{a_i}^x h_i(r)\; dr$, where $h_i\colon [a_i,b_i]\to\R$,
	depicted in Fig.~\ref{3FigHi}, is such that $h_i\coloneqq 0$ on $[a_i+\e_i^2,b_i-\e_i^2]$,  
	\begin{itemize}
		\item there exists $s_i\in (a_i,a_i+\e_i^2)$ such that 
		$h_i$ is linear on $[a_i,s_i]$ \linebreak
		with $h_i(s_i)\coloneqq 0$ and 
		$h_i(a_i)\coloneqq f'(a_i)-\frac{f(b_i)-f(a_i)}{b_i-a_i}$,
		while \linebreak
		$\int_{a_i}^{s_i} |h_i(r)|\; dr=\frac 12 |h_i(a_i)| (s_i-a_i)<\e_i^2$; 
		on $[s_i,a_i+\e_i^2]$ it is defined as 
		$h_i(x)\coloneqq A_i\dist(x,\{s_i,a_i+\e_i^2\})$, 
		where the constant $A_i$ is chosen so that $\int_{a_i}^{a_i+\e_i^2} h_i(r) \; dr=0$;
		
		\item there exists $t_i\in (b_i-\e_i^2,b_i)$ such that 
		$h_i$ is linear on $[t_i,b_i]$ \linebreak
		with $h_i(t_i)\coloneqq 0$  and 
		$h_i(b_i)\coloneqq f'(b_i)-\frac{f(b_i)-f(a_i)}{b_i-a_i}$,
		while \linebreak $\int_{t_i}^{b_i} |h_i(r)|\; dr=\frac 12 |h_i(b_i)| (b_i-t_i)<\e_i^2$; 
		on $[b_i-\e_i^2,t_i]$ it is defined as 
		$h_i(x)\coloneqq  B_i\dist(x,\{b_i-\e_i^2,t_i\})$, 
		where the constant $B_i$ is chosen so that $\int_{b_i-\e_i^2}^{b_i} h_i(r) \; dr=0$.
	\end{itemize}
	\begin{figure}[h!]
		\begin{center}
			\includegraphics[width=12.0cm]{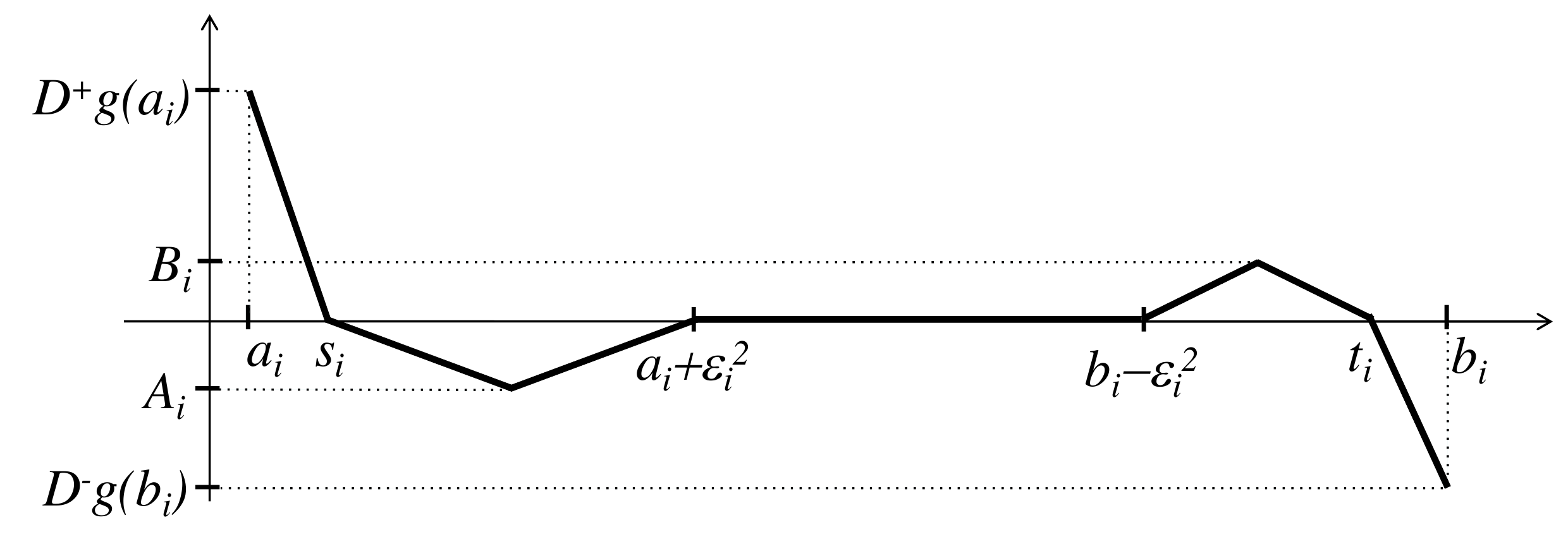}
		\end{center}
		\caption{A sketch of a map $h_i$.}
		\label{3FigHi}
	\end{figure}
	
	It is easy to see that such definition ensures that {$g\restriction[a_i,b_i]$} is  $C^1$,
	\begin{itemize}
		\item
		$D^+g(a_i)=f'(a_i)-\frac{f(b_i)-f(a_i)}{b_i-a_i}$ 
		and $D^-g(b_i)=f'(b_i)-\frac{f(b_i)-f(a_i)}{b_i-a_i}$,
		\item%[(c)] 
		$g=0$ on $[a_i+\e_i^2,b_i-\e_i^2]$ and $|g(x)|\leq \e_i^2$ for $x\in [a_i,b_i]$,
		\item%[(d)] 
		$|g(x)|\leq |g'(a_i)(x-a_i)|$ for $x\in [a_i,a_i+\e_i^2]$, and
		\item%[(e)] 
		$|g(x)|\leq |g'(b_i)(x-b_i)|$ for $x\in [b_i-\e_i^2,b_i]$.
	\end{itemize}
	
	We claim that, if $g$ satisfies all these requirements, then $F\coloneqq\bar f+g$ is differentiable. 
	To see this, it suffices to show that the unilateral derivatives $D^+F(x)$ and $D^-F(x)$ exist for all $x\in\R$. 
	Indeed, if they exist, then they are equal: for $x\in \R\setminus \hat Q$ this is obvious whereas, for $x\in\hat Q\setminus \bigcup_{1\leq i\leq\kappa}\{a_i,b_i\}$, this is ensured by the fact that $D^+F(x)=D^+f(x)=D^-f(x)=D^-F(x)$,
	while for $x\in\bigcup_{1\leq i\leq\kappa}\{a_i,b_i\}$ by the first of four items above. 
	
	By symmetry, it suffices to show the existence of $D^+F(x)$.
	It clearly exists, unless $x\in \hat Q\setminus \{a_i\colon 1\leq i\leq\kappa\}$.
	For such an $x$ we have $F(x)=f(x)$. Choose $\e>0$.	It is enough to find $\delta>0$ such that 
	\begin{equation}\label{eq:Fdiff}
		\mbox{$\left|f'(x)- \frac{F(y)-f(x)}{y-x}\right|<5 \e$ whenever $y\in(x,x+\delta)$.}
	\end{equation}
	For this, pick $m\in\N$ with $3^{-m}<\e$ and choose $\delta>0$ 
	such that $( x,x+\delta)$ is disjoint with $\bigcup_{i<m}[a_i,b_i]$ and 
	$\left|f'(x)- \frac{\bar f(y)-\bar f(x)}{y-x}\right|<\e$
	when  
	$0<|y-x|<\delta$. 
	An elementary, although tedious estimation (for details, see~\cite{MC&KC}),
	shows that such a choice of $\delta$ ensures 
	(\ref{eq:Fdiff}). Thus, indeed, $F=\bar f+g$ is differentiable,
	finishing the proof of part (a).  
	
	Now we turn our attention to the proof of part (b).
	First, assume that $f$ admits a $C^1$ extension, say  $G\colon\R\to\R$. We need to show 
	that  $\hat f$ is continuously differentiable.
	Since, by the proof of part (a), $\hat f\coloneqq F\restriction \hat Q$ is differentiable, it is enough to show that 
	$\hat f'$ is continuous.
	Clearly $\hat f'$ is continuous on $\hat Q\setminus Q$, since 
	$\hat f$ is locally linear on $\hat Q\setminus Q\coloneqq \bigcup_{J\in\J} I_J$.
	Thus, we need to show that $\hat f'$ is continuous on $Q$.
	Notice that $G=f$ and $G'=f'$ on $Q$.
	
	Fix an $x\in Q$ and $\e>0$. 
	It is enough to find a $\delta>0$ such that 
	\begin{equation}\label{eq:4}
		\mbox{$|\hat f'(x)-\hat f'(y)|<\e$ whenever $y\in \hat Q\cap (x-\delta,x+\delta)$. }
	\end{equation}
	Let $\delta_0\in(0,1)$  be such that $|G'(x)-G'(y)|<\e$ whenever $|x-y|<\delta_0$.
	Choose $\delta\in(0,\delta_0)$ such that for every $J\coloneqq(a,b)\in\J$:
	if $x\in[a,b]$, then $\delta<\frac{b-a}{3}$; if $x\notin[a,b]$ and $[a,b]\not\subset (x-\delta_0,x+\delta_0)$,
	then $(x-\delta,x+\delta)$ is disjoint with $[a,b]$.
	To see that such $\delta$ satisfies (\ref{eq:4}) pick $y\in \hat Q\cap (x-\delta,x+\delta)$.
	If $y\in Q$, then (\ref{eq:4}) holds, since then 
	$|\hat f'(x)-\hat f'(y)|=|G'(x)-G'(y)|<\e$.
	Thus, assume that $y\notin Q$. Then, there exists a $J=(a,b)\in \J$ such that $y\in I_J$. 
	Note that $x\notin[a,b]$, since in such case $\delta<\frac{b-a}{3}$, preventing $y\in I_J$.
	Therefore, $[a,b]\subset (x-\delta_0,x+\delta_0)$, as $(x-\delta,x+\delta)$ is not disjoint with $[a,b]$,
	both containing $y$. By the mean value theorem, there exists a $\xi\in(a,b)\subset (x-\delta_0,x+\delta_0)$
	such that $G'(\xi)=\frac{G(b)-G(a)}{b-a}$.
	Thus, $|G'(x)-G'(\xi)|<\e$. Also, $$\hat f'(y)=\bar f'(y)=\frac{f(b)-f(a)}{b-a}=\frac{G(b)-G(a)}{b-a}=G'(\xi).$$
	Therefore, $|\hat f'(x)-\hat f'(y)|=|G'(x)-G'(\xi)|<\e$, proving (\ref{eq:4}). 
	So, indeed, $\hat f'$ is continuous.
	
	To finish the proof, assume that the derivative of $\hat f$ is continuous.
	We need to show that, in such case, 
	there is a continuously differentiable extension $F\colon\R\to\R$ of $\hat f$.
	This $F$ is constructed by a small refinement of the construction
	of $F$ extracted from part (a). 
	More specifically, for every $1\leq i\leq\kappa$, 
	let $\alpha_i$ and $\beta_i$ be the endpoints of $[a_i,b_i]$ 
	such that $\hat f'(\alpha_i)\leq \hat f'(\beta_i)$ and, 
	when choosing 
	maps $h_i$, ensure that their range is contained in the interval 
	$$\left[\hat f'(\alpha_i)-\frac{f(b_i)-f(a_i)}{b_i-a_i}-3^{-i},\hat f'(\beta_i)-\frac{f(b_i)-f(a_i)}{b_i-a_i}+3^{-i}\right].$$
	This can be achieved by choosing $s_i$ and $t_i$ so close
	to, respectively, $a_i$ and $b_i$ that the resulted constants $A_i$ and $B_i$ have magnitude
	$\leq 3^{-i}$. 
	We claim, that  such constructed $F$ has continuous derivative.
	To see this, choose an $x\in\R$. 
	We will show only that $F'$ is right-continuous at $x$, the argument for 
	left-continuity being similar.

	Clearly, the definition of $F$ ensures that $F'$ is right-continuous at $x$
	if  there exists a $y>x$ such that $(x,y)\cap \hat Q=\emptyset$. 
	So, assume that there is no such $y$.  
	Choose an $\e>0$.
	It is enough to find a $\delta>0$ such that 
	\begin{equation}\label{eq:5}
		\mbox{$|F'(x)-F'(y)|<2\e$ whenever $y\in (x,x+\delta)$. }
	\end{equation}
	Let $\delta_0>0$
	be such that $|\hat f'(x)-\hat f'(y)|<\e$ whenever $y\in (x,x+\delta_0)\cap \hat Q$. 
	Choose $n\in\N$ such that $3^{-n}<\e$
	and let $\delta\in(0,\delta_0)$ such that: $(0,\delta)$ is disjoint with 
	every $(a_i,b_i)$ for which $i<n$; if $(a_i,b_i)$ intersects 
	$(0,\delta)$, then $[a_i,b_i]\subset (0,\delta_0)$. 
	To see that such $\delta$ satisfies (\ref{eq:5}) pick $y\in (x,x+\delta)$.
	If $y\in \hat Q$, then (\ref{eq:5}) holds, since then 
	$|F'(x)-F'(y)|=|\hat f'(x)-\hat f'(y)|<\e$.
	So, assume that $y\notin \hat Q$. Then, $y\in (a_i,b_i)$ for some $i\geq n$.
	Since $\beta_i\in [a_i,b_i]\subset (0,\delta_0)$,
	we have $|\hat f'(x)-\hat f'(\beta_i)|<\e$ and 
	$\hat f'(\beta_i)< \hat f'(x)+\e$.
	So, 
	$F'(y)=\bar f'(y)+g'(y)=
	\frac{f(b_i)-f(a_i)}{b_i-a_i}+h_i(y)\leq \hat f'(\beta_i)+3^{-i}< \hat f'(x)+\e+3^{-i}\leq F'(x)+2\e$.
	Similarly, 
	$F'(y)\geq \hat f'(\alpha_i)-3^{-i}> \hat f'(x)-\e-3^{-i}\geq F'(x)-2\e$. So, (\ref{eq:5}) holds. 
\end{proof}

Finally, let us note that 
there is no straightforward generalization of part (a) of Theorem~\ref{thmMC&KC},
that is, of Jarn\'\i k's theorem, to the differentiable functions $f$ defined on closed subsets $P$ of $\R^n$.
This is the case, since the derivative of any extension $F\colon \R^n\to\R$ 
is Baire class one, as it is a pointwise limit of continuous functions
$F_n(x)\coloneqq n\left(F\left(x+\frac1n\right)-F(x)\right)$. 
Therefore the derivative $f'$ of any differentiable extendable map
$f'\colon P\to\R$ must be also Baire class one.
However, there exists a differentiable function $f\colon P\to\R$,
with $P\subset\R^2$ being closed, for which $f'$ is not Baire class one, see~\cite[thm 5]{ALP}.
Clearly this $f$ admits no differentiable extension to $\R^2$. 
However, in~\cite{ALP} the authors prove that this is the only obstacle
to generalize Jarn\'\i k's theorem to multivariable functions.
More specifically, they prove that 
a differentiable function
$f\colon P\to\R$, with $P$ being a closed subset of $\R^n$,
admits differentiable extension $F\colon \R^n\to\R$ if, and only if, 
$f'\colon P\to\R$ is Baire class one. 

Also, a straightforward generalization of Jarn\'\i k's theorem, Theorem~\ref{thmMC&KC}(a), 
to the higher order smoothness is false, since,  by Example~\ref{ex111} that comes from 
\cite{KC&JS2018}, 
there are a perfect $Q\subset\R$ and a twice differentiable $f\colon Q\to\R$ such that $f$ admits no 
extension $F\in C^2(\R)$, in spite that it admits a $C^1$ extension $F\colon\R\to\R$. 
Compare also Section~\ref{sec:higherExt} and Problem~\ref{probJarnik}.

\subsection{$C^1$-interpolation theorem and Ulam-Zahorski problem}\label{sec:Interpolation} 

The ma\-in re\-sult we li\-ke to discuss here is the following theorem
of Steven J. Agronsky, Andrew M. Bruckner, Mikl\'os Laczkovich, and David Preiss
from 1985, \cite{ABLP}. 

\thm{Thm:Inter}{For every continuous function 
	$f\colon\R\to\R$ there exists a continuously differentiable function 
	$g\colon\R\to\R$ with the property  
	that the set $[f=g]\coloneqq \{x\in \R\colon f(x)=g(x)\}$ is uncountable.
	In particular, $[f=g]$ contains a perfect set $P$ and the restriction $f\restriction P$ is continuously differentiable. 
}

The story behind Theorem~\ref{Thm:Inter} spreads over a big part of the 20th century and 
is described in detail in \cite{Ol} and \cite{Br1}. 
Briefly, around 1940 Stanis\l aw Ulam (1909--1984) asked, in the Scottish Book (Problem 17.1, see \cite{Ulam} or \cite{ScotB})  whether every continuous $f\colon\R\to\R$ agrees with some real analytic function on an uncountable set.
Zygmunt Zahorski (1914--1998) showed, in his 1948 paper~\cite{Za}, that the answer is no: there exists a
$C^\infty$ (i.e., infinitely many times differentiable) function  which can agree with every real analytic function on at most countable set of points. 
At the same paper Zahorski stated a problem known latter as {\em Ulam-Zahorski problem}:
does every $C^0$ (i.e., continuous) function  
$f\colon\R\to\R$ agree with some $C^\infty$ (or possibly $C^n$ or $D^n$) function on an uncountable set?
Clearly, Theorem~\ref{Thm:Inter} shows that Ulam-Zahorski problem has an affirmative
answer for the class of $C^1$ functions. 
This is the best possible result in this direction, since Alexander 
Olevski\v{\i} (1939--) constructed, in his 1994 paper~\cite{Ol},
a continuous function $f\colon [0,1]\to\R$ 
which can agree with every $C^2$ function on at most countable set of points.

The results related to the Ulam-Zahorski problem for the higher order differentiable functions 
are discussed in Section~\ref{sec:InterpolationHigherOrder}.

\begin{figure}[h!]
	\centering
	\begin{tabular}{cc}
\includegraphics[width=0.3661\textwidth]{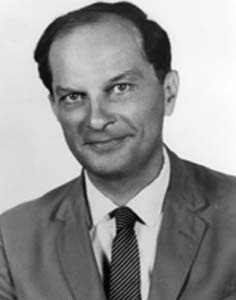} \hspace{.3cm} & \hspace{.3cm} \includegraphics[width=0.32\textwidth]{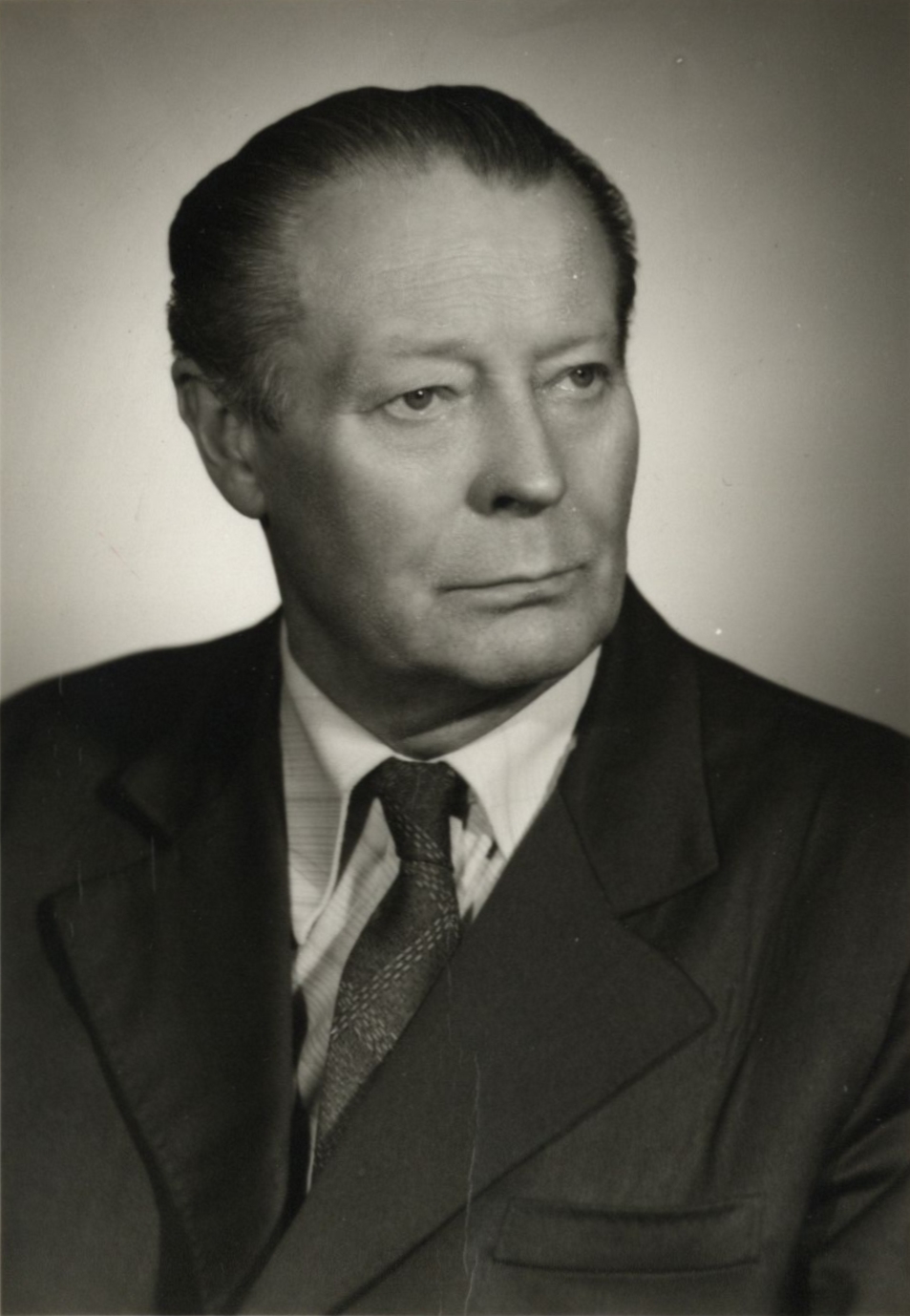}
	\end{tabular}
	\caption{Stanis\l aw Ulam and Zygmunt Zahorski.}
	\label{pic_UZ}
\end{figure}

\begin{proof}[Proof of Theorem~\ref{Thm:Inter}] 
	If	the perfect set $Q\subset\R$ is the one from Proposition~\ref{propMain}, then the quotient map 
	$q_0\coloneqq q\restriction Q^2\setminus\Delta$ can be extended to a uniformly continuous map $\bar q$ on $Q^2$ and $f\colon Q\to\R$ is continuously differentiable with $(f\restriction Q)'(x)=\bar q(x,x)$ for every $x\in Q$. By 
	part (b) of Theorem~\ref{thmMC&KC}, the extension $\hat f$ of $f\restriction Q$ is differentiable.
	In particular, 
	$\hat f'(x)= (f\restriction Q)'(x)$ for every $x\in Q$
	and $\hat f'(x)=\bar q(c,d)$ whenever $x\in I_J$, where $J\coloneqq (c,d)$ is a bounded connected component of $\R\setminus Q$.
	
	By Theorem~\ref{thmMC&KC}(b), we need to show that $\hat f'$ is continuous. 
	Clearly $\hat f'$ is continuous on $\hat Q\setminus Q$, since it is locally constant on this set. 
	So, let $x\in Q$ and fix $\e>0$. We need to find an open $U$ containing $x$ such that $|\hat f'(x)-\hat f'(y)|<\e$ whenever 
	$y\in \hat Q\cap U$. Since $\bar q$ is continuous, there exists an open $V\in \R^2$ containing $\la x,x\ra$
	such that $|\hat f'(x)-\bar q(y,z)|=|\bar q(x,x)-\bar q(y,z)|<\e$ whenever $\la y,z\ra\in Q^2\cap V$. 
	Let $U_0$ be open interval containing $x$ such that $U_0^2\subset V$ and 
	let $U\subset U_0$ be an open set containing $x$
	such that: if $U\cap I_J\neq\emptyset$ for some bounded connected component $J=(c,d)$ of $\R\setminus Q$, then $c,d\in U_0$.
	We claim that $U$ is as needed. Indeed, let $y\in \hat Q\cap U$. If $y\in Q$, then
	$\la y,y\ra\in U^2\subset V$ and $|\hat f'(x)-\hat f'(y)|=|\bar q(x,x)-\bar q(y,y)|<\e$. 
	Also, if $y\in I_J$ for some bounded connected component $J=(c,d)$ of $\R\setminus Q$, then $\la c,d\ra\in U_0^2\subset V$
	and, once again,  
	$|\hat f'(x)-\hat f'(y)|=|\bar q(x,x)-\bar q(c,d)|<\e$.
\end{proof}

\subsection{Differentiable maps on a perfect set $P\subset\R$: another monster}\label{sec:dynamics}

According to Theorem~\ref{thML}, every 
continuous $f\colon \R\to\R$ has a differentiable restriction to a perfect subset of $\R$.
Thus, a natural question is: 
\begin{quote}
{\em What can be said about differentiable functions $f\colon P\to\R$, where $P$ is a perfect subset of $\R$?}
\end{quote}
If $P$ has a non-zero Lebesgue measure, than quite a bit can be said about $f$.
(For example, its derivative will have a continuous, even Lipschitz, restriction
to a subset of $P$ of positive measure, see Remark~\ref{rem1}.) However, little seem to be known, in general case,
when  the perfect sets $P$ could have Lebesgue measure $0$.

In the ``positive'' direction, we have the following generalization of Theorem~\ref{thm:DerIsB1}. We would like to point out that the technic employed in our proof of Proposition~\ref{pr:DerIsB1partial}  is of a much simpler nature than that from the work of Mikl\'os Laczkovich on $C^1$ interpolation.

\prop{pr:DerIsB1partial}{If $P\subset \R$ is perfect and $F\colon P\to\R$ is differentiable, then $F'$ is  Baire class one.
	In particular, $F'$ is continuous on a dense $G_\delta$ subset of~$P$.}

\begin{proof}
	Let $\bar F\colon \R\to\R$ be a differentiable extension of $F$, which exists by 
	Theorem~\ref{thmMC&KC}(a). Then, by Theorem~\ref{thm:DerIsB1}, 
	$\bar F'$ is  Baire class one, and so is $F'=\bar F'\restriction P$. 
\end{proof}

There is little else we can say about the derivatives of differentiable functions $F\colon P\to\R$. 
The next example, first constructed in 2016 by the first author % Krzysztof C. Ciesielski
 and Jakub Jasinski in~\cite{Ci121},
shows how counterintuitively such maps can behave.
We use the symbol $\Cantor$ to denote the Cantor ternary set, that is, 
$\Cantor\coloneqq \left\{\sum_{n=0}^\infty 
\frac{2s(n)}{3^{n+1}}\colon s\in 2^\omega\right\}$, where $2^\omega$ is the 
set of all functions from $\omega\coloneqq\{0,1,2,\ldots\}$ into $2\coloneqq\{0,1\}$.

\begin{figure}[h!]
	\centering
	\includegraphics[width=0.9\textwidth]{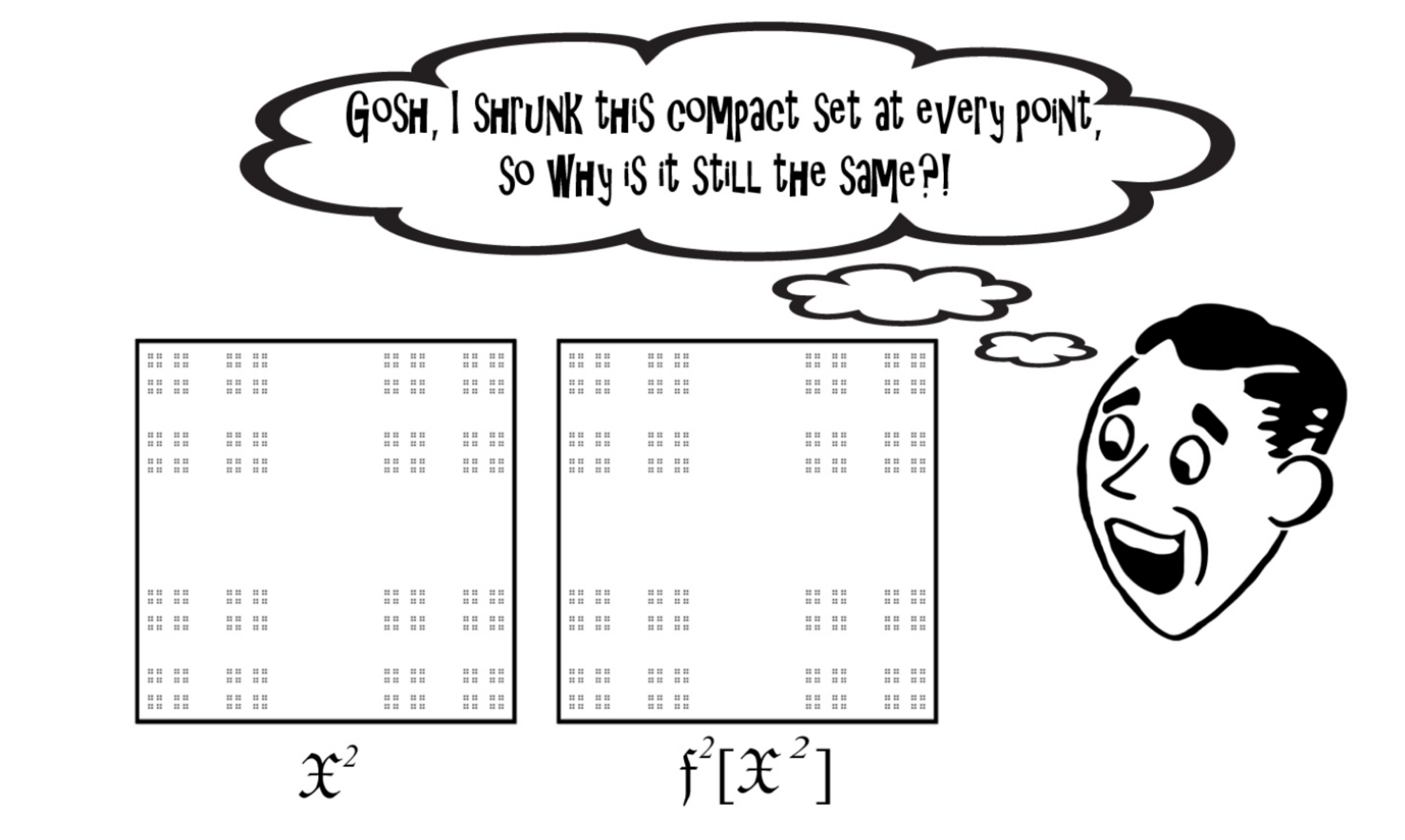}
	\caption{Illustration for Example \ref{Ex:Monster}. Result of the action of $\frak{f}^2 = \langle \frak{f}, \frak{f} \rangle$ on $\frak{X}^2 = \frak{X}\times  \frak{X}$.}\label{fX2}
\end{figure}

\ex{Ex:Monster}{There exists a perfect set $\mathfrak{X}\subset\Cantor$ and a differentiable bijection 
	${\mathfrak{f}}\colon \mathfrak{X}\to \mathfrak{X}$ such that ${\mathfrak{f}}\,'(x)=0$ for every $x\in \mathfrak{X}$. 
	 (See Fig.~\ref{fX2}.) Moreover, ${\mathfrak{f}}$ does not have any periodic points.
}

Of course, by Theorem~\ref{thmMC&KC}(a), function 
${\mathfrak{f}}$ can be extended to a differentiable map $F\colon\R\to\R$. 
However, no function ${\mathfrak{f}}$ as in the example admits 
{\em continuously}\/ differentiable extension $F\colon\R\to\R$, as proved in \cite[lemma 3.3]{Ci112}.

There is something counterintuitive about such function. Having derivative 0 at every point, 
it is {\em pointwise contractive}\/ with every constant $\lambda\in(0,1)$:
for every $x\in \mathfrak{X}$ there is an open subset $U$ of $\mathfrak{X}$ 
containing $x$ such that $|\mathfrak{f}(x)-\mathfrak{f}(y)|\leq \lambda  |x-y|$ for all $y\in U$.
Thus,  $\mathfrak{f}$
is pointwise contractive but globally stable (in the sense that  
$\mathfrak{f}[\mathfrak{X}]=\mathfrak{X}$).
The functions that have such property globally
cannot map any compact perfect set onto itself:
if $\mathfrak{f}$ is {\em shrinking}\/ (i.e., such that $|\mathfrak{f}(x)-\mathfrak{f}(y)|<  |x-y|$ for all distinct $x,y\in \mathfrak{X}$),
then  $\diam(\mathfrak{f}[\mathfrak{X}])<\diam(\mathfrak{X})$. 
The map $\mathfrak{f}$ also cannot be 
{\em locally shrinking}, in a sense that for every $x\in \mathfrak{X}$ there exists an open $U\ni x$ in $\mathfrak{X}$ 
such that $\mathfrak{f}\restriction U$ is shrinking. Indeed, 
by a theorem of Michael  Edelstein (1917-2003) from his 1962 paper~\cite{Edel2}, every locally shrinking self-map of a compact space must have a periodic point. 
Of course, $\mathfrak{X}$ must have Lebesgue measure zero, since ${\mathfrak{f}}\,'\equiv0$
implies that $\mathfrak{f}[\mathfrak{X}]$ must have measure zero, see for example~\cite[p. 355]{Foran}.

\begin{figure}[h!]
	\centering
	\includegraphics[width=0.33\textwidth]{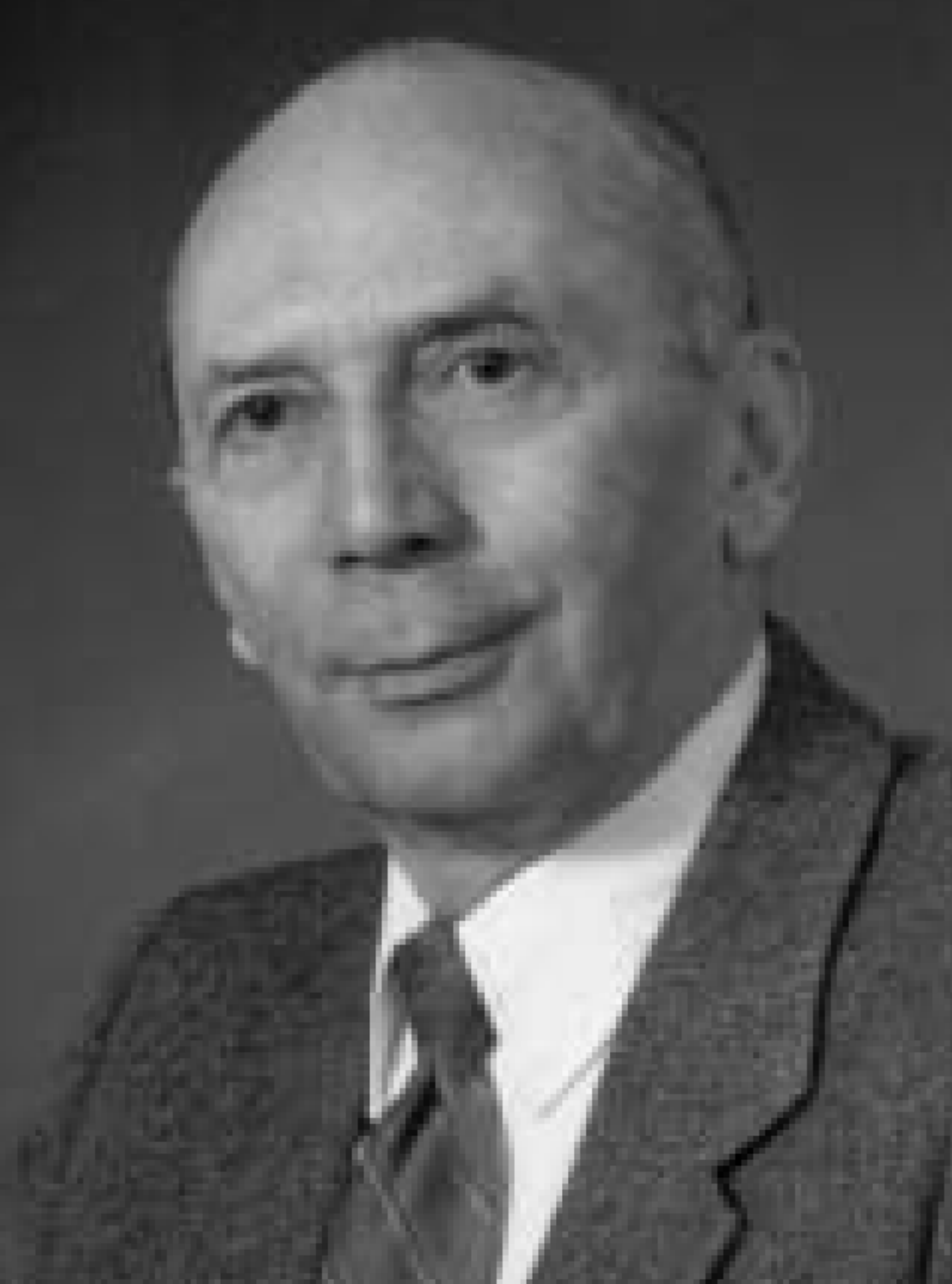}
	\caption{Michael Edelstein.}
	\label{pic_ME}
\end{figure}

The construction of $\mathfrak{f}$ we present below comes from 
\cite{KC:Monthly} and is based on its variants from~\cite{Ci121} and~\cite{BKO}.

\begin{proof}[Construction of $\mathfrak{f}$ from Example~\ref{Ex:Monster}] 
	Let 
	$\sigma\colon 2^\omega\to 2^\omega$ be the {\em add-one-and-carry adding machine},
	that is, defined, 
	for every  
	$s\coloneqq\la s_0,s_1,s_2,\ldots\ra\in2^\omega$, as 
	\[
	\sigma(s)\coloneqq
	\begin{cases} 
	\la 0,0,0,\ldots\ra & \mbox{ if $s_i=1$ for all $i<\omega$},\\
	\la 0,0,\ldots,0,1,s_{k+1},s_{k+2},\ldots\ra & \mbox{ if $s_k=0$ and $s_i=1$ for all $i<k$}.
	\end{cases}
	\]
	For more on adding machines, see survey~\cite{Do}.
	
	The map $\mathfrak{f}$ is defined as 
	$\mathfrak{f}\coloneqq h\circ\sigma\circ h^{-1}\colon h[2^\omega]\to h[2^\omega]$, where
	$h\colon 2^\omega\to\R$ 
	is an appropriate embedding that ensures that $\mathfrak{f}\,'\equiv 0$. 
	Thus, $\mathfrak{X}\coloneqq h[2^\omega]$. 
	
	We define embedding $h$ via formula:
	$$
	h(s)\coloneqq  \sum_{n=0}^\infty 2s_n 3^{-(n+1)N(s\restriction n)},
	$$ 
	where 
	$N(s\restriction n)$ is the natural 
	number for which the following 0-1 
	sequence\footnote{
		$\nu(s,n)$ is obtained from $s\restriction n=\la s_0,\ldots,s_{n-1}\ra$ by: ``flipping'' its last digit $s_{n-1}$ to 
		$1-s_{n-1}$, appending 1 at the end, and reversing the order. The `appending 1' step is to ensure that 
		$2^{n}\leq N(s\restriction n)$. The ``flipping'' step is the key new trick,
		that comes from~\cite{BKO}.}
	$\nu(s,n)\coloneqq\la 1,1-s_{n-1},s_{n-2},\ldots,s_{0}\ra$ 
	is its binary representation, that is, we have 
	$N(s\restriction n)\coloneqq  \sum_{i<n-1}s_i2^i +(1-s_{n-1}) 2^{n-1}+2^{n}$.
	
	Clearly, $2^{n}\leq N(s\restriction n)\leq\sum_{i\leq n}2^i<2^{n+1}$ for every $s\in 2^\omega$ and $n<\omega$. 
	Hence, the sequence $\la N(s\restriction n)\colon n<\omega\ra$ 
	is strictly increasing and $h$  is an embedding into $\Cantor$. 
	So, $\mathfrak{X}=h[2^\omega]\subset\Cantor$. 
	
	The proof that $\mathfrak{f}\,'\equiv 0$ 
	follows from  two observations: 
	\begin{itemize}
		\item[(a)] for every $s\in 2^\omega
		$ there is a $k<\omega$ such that  
		$N(\sigma(s)\restriction n)= N(s\restriction n)+1$ for every $n>k$;
		\item[(b)]  if $n\coloneqq\min\{i<\omega\colon s_i\neq t_i\}$ for some distinct 
		$s\coloneqq\la s_i\ra$ and $t\coloneqq\la t_i\ra$ from 
		$2^\omega$,  then 
		$3^{-(n+1)N(s\restriction n)}\leq |h(s)-h(t)|\leq 3 \cdot 3^{-(n+1)N(s\restriction n)}$. 
	\end{itemize}
	Indeed, to see that $\mathfrak{f}\,'(h(s))=0$ for an $s\in 2^\omega$, choose a $k<\omega$ satisfying (a)
and let $\delta>0$ be such that the inequality $0<|h(s)-h(t)|<\delta$ implies that 
$n=\min\{i<\omega\colon s_i\neq t_i\}$ is greater than $k$. 
Then, for any $t\in 2^\omega$ for which $0<|h(s)-h(t)|<\delta$, we have 
$$n=\min\{i<\omega\colon s_i\neq t_i\}=\min\{i<\omega\colon \sigma(s)_i\neq \sigma(t)_i\}$$ and,
using (a) and (b) for the pairs $\la s,t\ra$ and $\la \sigma(s),\sigma(t)\ra$, we obtain 
$$
\frac{|\mathfrak{f}(h(s))-\mathfrak{f}(h(t))|}{|h(s)-h(t)|}=
\frac{|h(\sigma(s))-h(\sigma(t))|}{|h(s)-h(t)|}\leq
\frac
{3 \cdot 3^{-(n+1)N(\sigma(s)\restriction n)}}
{3^{-(n+1)N(s\restriction n)}}=3 \cdot 3^{-(n+1)}.
$$
Thus, indeed $\mathfrak{f}\,'(h(s))=0$, as 
$3 \cdot 3^{-(n+1)}$ is arbitrarily small for $\delta$ small enough.

To see (a) let $s=\la s_i\ra_i$ and notice that, for every $0<n<\omega$,
\begin{equation}\label{one}
\mbox{$N(\sigma(s)\restriction n)= N(s\restriction n)+1$, unless $s_0=\cdots=s_{n-2}=1$ and $s_{n-1}=0$.}
\end{equation}
Indeed, if $s_i=0$ for some $i<n-1$, then $\sum_{i<n-1}\sigma(s)_i2^i=1+\sum_{i<n-1}s_i2^i$ 
and $\sigma(s)_{n-1}=s_{n-1}$, giving (\ref{one}). 
Otherwise, 
$s_0=\cdots=s_{n-2}=1$ and, by our assumption, 
also $s_{n-1}=1$. This implies that 
$\sigma(s)_0=\cdots=\sigma(s)_{n-1}=0$.
Thus  
$N(\sigma(s)\restriction n)=\sum_{i<n-1}2^i +2^{n}=2^{n-1}-1+2^{n}$ and 
$N(s\restriction n)= 2^{n-1}+2^{n}$, again giving (\ref{one}). 

Since for every $s=\la s_i\ra_i\in 2^\omega$ there is at most one $0<n<\omega$
for which $s_0=\cdots=s_{n-2}=1$ and $s_{n-1}=0$, any $k$ greater than this number satisfies~(a). 

To see property (b), first notice that for every $s\coloneqq\la s_i\ra_i \in 2^\omega$ and $n<\omega$, if 
$H(s\restriction n)\coloneqq\sum_{k< n} 2s_k 3^{-(k+1)N(s\restriction k)}$ is the initial partial sum of $h(s)$, then 
\begin{equation}\label{two}
\mbox{$H(s\restriction n)+ 2s_n 3^{-(n+1)N(s\restriction n)}\leq h(s)\leq H(s\restriction n)+ (2s_n +1)3^{-(n+1)N(s\restriction n)}$.}
\end{equation}
Indeed, 
$h(s)= H(s\restriction n) + 2s_n 3^{-(n+1)N(s\restriction n)}+2\sum_{k> n} 3^{-(k+1)N(s\restriction k)}$ while also 
$0\leq 2\sum_{k> n} 3^{-(k+1)N(s\restriction k)}
\leq 2\sum_{i=1}^\infty 3^{-[(n+1)N(s\restriction n)+i]}
=3^{-(n+1)N(s\restriction n)}$,
where the second inequality holds, since 
the sequence $\la (k+1)N(s\restriction k)\ra_k$  is strictly increasing
(as $2^k\leq N(s\restriction k)<2^{k+1}$). 
This clearly implies (\ref{two}). 

To prove the inequalities in (b) we can assume that $s_n=0$ and $t_n=1$. 
Then, by (\ref{two}) used for $s$ and $t$, 
we have 
$h(s)\leq H(s\restriction n)+ 3^{-(n+1)N(s\restriction n)}$ and \linebreak
$H(t\restriction n)+ 2 \cdot 3^{-(n+1)N(t\restriction n)}\leq h(t)$. 
Using these inequalities and our assumption that 
$t\restriction n=s\restriction n$,
we obtain $h(t)-h(s)\geq  3^{-(n+1)N(s\restriction n)}>0$.
In particular, we get the lower bound 
$|h(s)-h(t)|=h(t)-h(s)\geq  3^{-(n+1)N(s\restriction n)}$. 
Next, using just proved fact that $h(s)<h(t)$
and property (\ref{two}) for $s$ and $t$,
we obtain 
$$H(t\restriction n)=H(s\restriction n)\leq h(s)<h(t)\leq H(t\restriction n)+ 3 \cdot 3^{-(n+1)N(t\restriction n)}.$$
In particular, $$|h(s)-h(t)|=h(t)-h(s)\leq 3 \cdot 3^{-(n+1)N(t\restriction n)}=3 \cdot 3^{-(n+1)N(s\restriction n)},$$
the desired upper bound. 
\end{proof} 

We like to stress, once more, that the compactness 
of $\mathfrak{X}$ in Example~\ref{Ex:Monster} is what makes it so paradoxical. 
It is relatively easily to believe in the existence of the perfect unbounded subsets of 
$\R$ that admit similar mappings. 
Actually, it has been proved by the first author %Krzysztof C. Ciesielski
 and Jakub Jasinski in~\cite{Ci112}
that there exists a $C^\infty$ function $g=\la g_1,g_2\ra\colon\R\to\R^2$ and a perfect unbounded $P\subset \R$ 
such that $g_1'\restriction P=g_2'\restriction P\equiv  0$ and 
$g\restriction P$ is Peano-like in a sense that $g[P]=P^2$. 
On the other hand, it is unknown
(see Problem~\ref{prPeano} and \cite[problem 1]{Ci112}),
whether there exists a $D^1$ function $h$ 
(i.e., with $D^1$ coordinates) from a compact perfect $P\subset\R$ onto $P^2$. 
Of course, by Theorem~\ref{thmMC&KC}(a), such a map could be extended to a $D^1$ map from $\R$ to $\R^2$.
(However, there is no such an $h$ that could be extended to a $C^1$ map from $\R$ to $\R^2$, since
it has been proved in 
\cite[thm 3.1]{Ci112} that $P^2\not\subset f[P]$ 
for every $C^1$ function $f\colon\R\to\R^2$  and compact perfect $P\subset\R$.)

It has been recently proved, by the first author and his student Cheng-Han Pan \cite{CiPan}, that 
the function ${\mathfrak{f}}\colon \mathfrak{X}\to \mathfrak{X}$
from Example~\ref{Ex:Monster} can be also extended to functions $F_1,F_2\colon\R\to\R$ such that 
$F_1$ is a Weierstrass monster, while $F_2$ is a differentiable monster.
This squeezes three paradoxical examples to just two functions. 
The existence of a differentiable monster $F_2\colon\R\to\R$ extending ${\mathfrak{f}}$ follows immediately from the 
following ``twisted'' version of Jarn\'\i k's Differentiable Extension Theorem, our Theorem~\ref{thmMC&KC}(a), that comes from~\cite{CiPan}. 

\thm{mainCHP}{
For every perfect $P\subseteq\mathbb{R}$ and differentiable $f\colon P\to\mathbb{R}$, 
there exists a differentiable extension $\hat{f}\colon\mathbb{R}\to\mathbb{R}$ of $f$ such that 
$\hat{f}$ is nowhere monotone on $\mathbb{R}\setminus P$.
In particular, if $P$ is nowhere dense in $\R$, then $\hat{f}$ is monotone on no interval.}

The function $\hat{f}$ in Theorem~\ref{mainCHP} is constructed by using Theorem~\ref{thmMC&KC}(a)
to find an arbitrary differentiable extension $F\colon\mathbb{R}\to\mathbb{R}$ of $f$,
choosing differentiable $g\colon\R\to[0,\infty)$ with $g^{-1}(0)=P$,
and using the existence of a differentiable monster to 
find an extension $\hat{f}\colon\mathbb{R}\to\mathbb{R}$ of $f$
which differentiable nowhere monotone on $[a,b]$ for every $(a,b)\subset\R\setminus P$ 
and such that 
$\left|\hat f(x)-F(x)\right|\leq g(x)$  for every $x\in\R$. 
Such $\hat{f}$ is also differentiable on $P$, which is verified by a simple application the squeeze theorem. 
(See \cite[lemma 4]{CiPan}.)

\subsection{A few words on monotone restrictions}\label{sec:monotone}
Of course, the concept of monotonicity is closely related to both continuity and differentiability.
Therefore, we like to finish this section with some facts concerning monotone restrictions of continuous functions.

We start with the following 1966 theorem of 
Franciszek Miros{\l}aw Filipczak \cite{filipczak1966}. 

\thm{thm:filipczak}{For every continuous function $f\colon\R\to\R$ 
and every perfect $P\subset \R$ there exists a perfect set $Q\subset P$
such that $f\restriction Q$ is monotone. 
} 

\begin{proof}
For $P=\R$ the set $Q$ we constructed in our proof of Theorem~\ref{thML} 
is as needed. More specifically,
if $f$ is monotone on some non-trivial interval $[a,b]$, then $Q=[a,b]$ is as needed.
Otherwise, by a theorem of Komarath Padmavally~\cite{Padm},
there is perfect set $Q\subset\R$ on which $f$ is constant, so monotone.  
However, even for a general perfect set $P\subset\R$, one can prove the theorem
by the following simple argument.

If there is a non-empty open subset $U$ of $P$ on which $f$ is monotone,
then $Q=\cl_P(U)$ is as needed.
Otherwise, construct (by induction on $n<\omega$) the closed non-empty intervals $\{I_s\colon s\in 2^n \ \& \ {n<\omega}\}$,
such that for every $s\in  2^n$:
\begin{itemize}
\item $I_s$ is of length $\leq 2^{-n}$ and $I_s\cap P$ is perfect;
\item $I_{s0}$ and $I_{s1}$ are disjoint subsets of $I_s$ such that $f(x)<f(y)$
for every $x\in I_{s0}\cap P$ and $y\in I_{s1}\cap P$.
\end{itemize}
Then $Q=\bigcap_{n<\omega} \bigcup_{s\in 2^n} I_s$ is as needed. 
\end{proof}

Notice that, in Theorem~\ref{thML} the perfect set $Q$ (with differentiable $f\restriction Q$)
cannot be chosen inside a given perfect set $P$, unless $P$ is of positive Lebesgue measure
(compare with Remark~\ref{rem1}.) This holds true when $f$ has infinite derivative on some perfect set $P$.
For example, such an $f$ can be chosen as a Pompeiu  function $g$ from Proposition \ref{pr:Pom}
with $P$ being a subset of the dense $G_\delta$ set $\{x\in \R\colon g'(x)=\infty\}$.

The above discussion shows that finding a monotone restriction $f\restriction Q$,
of continuous $f$, is a harder problem than that of finding a differentiable restriction.
In fact, by Theorem~\ref{thm:filipczak},
any differentiable restriction $f\restriction Q$ can be further refined, so that 
$f\restriction Q$ is also monotone (which, clearly, cannot be done in the reversed order).  
Thus, one may wonder, if for $P=\R$ the set $Q$ in Theorem~\ref{thm:filipczak}
can be always chosen having positive Lebesgue measure.
A relatively easy counterexample for this assertion is a continuous $f$ which is 
nowhere approximately differentiable, e.g., a function $f$ from Remark~\ref{rem2}. 
Indeed, for such an $f$ and any perfect $Q$ of positive Lebesgue measure
the restriction $f\restriction Q$ cannot be monotone.
(Otherwise, a monotone extension $\bar f\colon\R\to\R$ of $f\restriction Q$, say its linear interpolation,
is, by a theorem of  Lebesgue, differentiable almost everywhere. Hence, $f\restriction Q$
has many points of differentiability, at which $f$ is approximately differentiable, a contradiction.) 

A considerably stronger counterexample was given in a 2009 paper \cite{kahanekatznelson2009}
of Jean-Pierre Kahane (1926--2017) and Yitzhak Katznelson
by constructing a continuous function $f\colon\R\to\R$ 
such that $f \restriction E$ is not monotone unless $E$ has Hausdorff dimension~0. 
More on monotone restrictions can be found in a 2011 paper \cite{kharazishvili2011}
of Alexander B. Kharazishvili (1949--) and 2017 article \cite{buczolich2017} of 
Zolt\'an Buczolich (1961--).

\section{Higher order differentiation}\label{sec:higher}

In this section we will discuss, in more detail, the higher order versions of Ulam-Zahorski interpolation problem and of the differentiable extension theorems of  Whitney and Jarn\'\i k. 

\subsection{Extension theorems}\label{sec:higherExt}
The original 1934 Whitney's Extension Theorem \cite{Wh}
provides the necessary and sufficient conditions for a  function $f$ from 
a closed subset $P$ of $\R^k$ ($k\in\N$) into $\R$ 
to have a $C^n$ ($n\in\N$) extension $\bar f\colon \R^k\to\R$. 
This theorem has been studied extensively, see e.g. \cite{Federer,WET2,WET3}. 
Here, we discuss it only for $k=1$ and $P\subset\R$ being perfect.
These assumptions ensure that the notion of the derivative of $f$ 
is well defined at each $a\in P$, what allows a simpler statement of the 
theorem\footnote{In terms of Taylor polynomials, rather than some implicitly given polynomials.} 
and a relatively simple proof of it, both coming from \cite{KC&JS2018}.

For an $n<\omega$, a perfect set $P\subset\R$, a $D^n$ function $f\colon P\to\R$, and 
an $a\in P$ let $T^n_a f(x)$ denote the $n$-th degree Taylor polynomial of $f$ at $a$:
\begin{equation*}
T^n_a f(x)\coloneqq\sum_{i=0}^{n} \frac{f^{(i)}(a)}{i!}(x-a)^{i}.
\end{equation*}
Also, define the map $q_f^n\colon P^2\to\R$ as 
\begin{equation*}
q_f^n(a,b)\coloneqq
\begin{cases}
\displaystyle \frac{T^n_b f(b)-T^n_a f(b)}{(b-a)^n} & \mbox{ if $a\neq b$,}\\
0& \mbox{ if $a= b$}.
\end{cases}
\end{equation*}

\thm{th:Whitney}{{\bf [Whitney's Extension Theorem]}
Let $P\subset \R$ be perfect, $n\in\N$, and $f\colon P\to\R$. 
There exists a $C^n$  extension $\bar f\colon \R\to\R$ of $f$
if, and only if, 
\begin{itemize}
\item[($W_n$)] $f$ is $C^n$ and the map $q_{f^{(i)}}^{n-i}\colon P^2\to\R$ is continuous for every $i\leq n$.
\end{itemize}
}

Theorem \ref{th:Whitney} easily follows from the general version of Whitney's Extension Theorem,
whose (quite intricate) proof can be found in \cite{Federer}, \cite{WET2}, or \cite{Wh}. 
A considerable shorter detailed proof of the specific form of Theorem \ref{th:Whitney}
can be found in  \cite{KC&JS2018}.

It should be noticed that the 
necessity part of Theorem \ref{th:Whitney} is easy to see. Specifically, 
($W_n$) must be satisfied by $f$, since 
it must be satisfied 
by any $C^n$ function $\bar f\colon \R\to\R$---this can be deduced from the well known behavior of
the reminder of the Taylor polynomials, see  \cite[prop. 3.2]{KC&JS2018}. 
Thus, the true value of the theorem lies in the sufficiency of the condition,
that is, the construction a $C^n$  extension $\bar f\colon \R\to\R$ of $f$ and the proof that, under the assumptions,
it is indeed $C^n$. 

This extension 
is defined as a weighted average of 
the maps $T^n_a f$, where the weights are given by an appropriate partition of unity of the complement of $P$.
Finding such a partition is the main difficulty for the functions of more than one variable. However, this difficulty 
almost completely vanishes for the functions of one variable, as we see below.

\begin{proof}[The construction of $\bar f$ from Theorem \ref{th:Whitney}] 
Let $f\colon P\to\R$ be as in the assumptions and let 
$H$ be the convex hull of $P$. 
We will construct 
a $C^n$ extension $\bar f\colon H\to\R$ of $f$.
This will finish the proof since, in an event when the interval $H$ is not the entire $\R$,
a further $C^n$ extension of $\bar f$ defined on $\R$ can be easily found. 

Let $\{(a_j,b_j)\colon j\in J\}$ be the family of all connected components of $H\setminus P$. 
Choose a non-decreasing $C^\infty$ map $\psi\colon\R\to\R$ 
such that $\psi=1$ on $[2/3,\infty)$ and $\psi=0$ on $(-\infty,1/3]$.
For every $j\in J$ define the following functions from $\R$ to $\R$: 
\begin{itemize}
\item the linear map $\displaystyle L_j(x)\coloneqq\frac{x-a_j}{b_j-a_j}$ (so, $L_j(a_j)=0$ and $L_j(b_j)=1$);
\item $\beta_j\coloneqq\psi\circ L_j$ and $\alpha_j\coloneqq1-\beta_j$;
\item $\bar f_j\coloneqq\alpha_jT^n_{a_j}f+\beta_jT^n_{b_j}f$. 
\end{itemize}
Then, the extension $\bar f\colon H\to\R$ of $f$
is defined by declaring simply that 
\begin{equation*}
\bar f\restriction (a_j,b_j)\coloneqq\bar f_j\restriction (a_j,b_j) \mbox{ for every $j\in J$}.
\end{equation*}

The detailed two page long proof showing that such function  $\bar f\colon H\to\R$ is, indeed, $C^n$ 
can be found in \cite{KC&JS2018}.
\end{proof}

Interestingly, it is relatively easy to deduce from Theorem \ref{th:Whitney}
its $C^\infty$ version, using the fact that the extension $\bar f$ is always 
$C^\infty$ on the complement of $P$. For the proof, see \cite[theorem 3]{Merrien}.\footnote{This 
result can be also deduced from Whitney's papers 
\cite{whitney2} and \cite[\S 12]{Wh}. See also the 1998 paper \cite{P1998}, where it is shown that
the analogous result for  functions on $\R^k$, $k\ge 2$, does not hold.}

\rem{rem:Wh}{If $P\subset \R$ is perfect, then $f\colon P\to\R$
admits a $C^\infty$ extension $\bar f\colon \R\to\R$ 
if, and only if, ($W_n$) holds for every $n\in\N$. 
}

One must be very careful when considering variations of Theorem \ref{th:Whitney}.
For example, for $n\in\N$
consider the following statement on the existence of $C^n$ extensions:
\begin{itemize}
\item[($L_n$)] Let $f\colon \R\to\R$ be $C^{n-1}$ 
and $P\subset\real$ be a perfect set 
for which the map $F\colon P^2\setminus\Delta\to\real$
defined by 
$
F(x,y)\coloneqq\frac{f^{(n-1)}(x)-f^{(n-1)}(y)}{x-y}
$
is uniformly continuous and bounded. Then $f\restriction P$ can be extended to
a $C^n$ function $\bar f\colon \R\to\R$. 
\end{itemize}
The property ($L_1$) is well known and
follows immediately from Theorem \ref{th:Whitney} used with $n=1$, since 
the continuity of $q_{f^{(1)}}^{1-1}(a,b)=f'(b)-f'(a)$ is just the continuity of $f'$,
which follows from the assumptions on the function $F$ from ($L_1$),
while the continuity of 
$q_{f^{(0)}}^{1-0}(a,b)=\frac{f(b)-f(a)}{b-a}-f'(a)$ is equivalent of the continuity of 
$F$.

In the book \cite{CPAbook}
of the first author and Janusz Pawlikowski it is claimed (as lemma 4.4.1) 
that ($L_n$) is also true for $n>1$.\footnote{The same claim is also present in \cite{CiPa}.
The error was caused by an incorrect interpretation of \cite[thm. 3.1.15]{Federer}.
Luckily, the results deduced in \cite{CPAbook} and \cite{CiPa}
from the incorrect claim remain true, as recently proved in \cite{KC&JS2018}. 
}
The next example shows that such claim is false. 

\ex{ex111}{
Let $\Cantor$ be the Cantor ternary set.
There exists a $C^1$ function $f\colon\R\to\R$ 
such that $f'\restriction \Cantor\equiv 0$ 
and for no perfect set $P\subset\Cantor$ 
there is a $C^2$ extension $\bar f\colon\R\to\R$ 
of $f\restriction P$. 
%\cite[Theorems 4.1.1(b) and 4.1.6]{CP2004}.
In particular, $f\restriction\Cantor$ contradicts ($L_2$).
}

\begin{proof}[Construction] For $n\in\N$ let 
$\J_n$ be the family of all connected components
of $\R\setminus \Cantor$ of length $3^{-n}$. 
Define $f_0\colon \R\to\R$ as 
\[
f_0(x)\coloneqq
\begin{cases}
\frac{2^{-n}}{3^{-n}} \dist(x,\Cantor) & \mbox{ if $x\in J$, where $J\in\J_n$  for some $n\in\N$, and}\\
0 
& \mbox{ otherwise.}\\
\end{cases}
\]
It is easy to see that $f_0$ is continuous,
since $f_0[J]\subset[0,2^{-n}]$ for every $J\in\J_n$. 
Define $f\colon \R\to\R$ 
via formula 
$f(x)\coloneqq\int_0^x f_0(t) \, dt$. 
Clearly $f$ is $C^1$ 
and $f'\restriction \Cantor=f_0\restriction \Cantor\equiv 0$.
We just need to verify the statement about the extension. 

To see this, notice that for every $n\in\N$ and distinct $a,b\in\Cantor$
\begin{equation}\label{ONE}
\mbox{if $|b-a|<3^{-n}$, then } \frac{|f(b)-f(a)|}{(b-a)^2}>\frac{1}{36}\left(\frac32\right)^n.
\end{equation}
Indeed, if $m\in \N$ is the smallest such 
that there is $J=(p,q)\in\J_m$ between $a$ and $b$,
then $m>n$, 
$|b-a|\leq 3 \cdot 3^{-m}$, and  
$|f(b)-f(a)|\geq \int_p^q f_0(t)\, dt= \frac{1}{2} 3^{-m}\frac{1}{2}2^{-m}$.
So, 
$ \frac{|f(b)-f(a)|}{(b-a)^2}\geq \frac{\frac{1}{4} 3^{-m}2^{-m}}{(3\cdot 3^{-m})^2}
=\frac{1}{36}\left(\frac32\right)^m>\frac{1}{36}\left(\frac32\right)^n$. 
But this means that for every perfect $P\subset\Cantor$ the map $f\restriction P$ does not 
satisfy condition ($W_2$) from Theorem~\ref{th:Whitney},
our version of Whitney's Extension Theorem, which is necessary for 
admitting a $C^2$ extension $\bar f\colon\R\to\R$ 
of $f\restriction P$.
More specifically, either $(f\restriction P)''(a)$ does not exist or else 
\begin{align*}
\displaystyle \left| q_{f\restriction P}^2(a,b)\right| &
\displaystyle =\frac{\left|f(b)-f(a)-\frac12 (f\restriction P)''(a) (b-a)^2\right|}
{(b-a)^2} \\
& \displaystyle \geq \frac{|f(b)-f(a)|}{(b-a)^2}-\frac12 (f\restriction P)''(a),
\end{align*}
that is,
$q_{f\restriction P}^2$ is not continuous at $\la a,a\ra$, as, by (\ref{ONE}), 
$\displaystyle\lim_{b\to a, b\in P} \frac{|f(b)-f(a)|}{(b-a)^2}=\infty$. 
\end{proof} 

\subsubsection*{Is there higher order Jarn\'\i k's Extension Theorem?}

Theorem \ref{th:Whitney} gives a full characterization of functions $f$ 
from perfect $P\subset\R$ into $\R$ that admit $C^n$ extensions $\bar f\colon \R\to\R$.
In this context, it is natural also to consider the following question: 
\begin{itemize}
	\item[Q:] Is there an analogous characterizations of functions $f\colon P\to\R$,  
	where $P\subset\R$ is perfect,
that admit $D^n$ extensions $\bar f\colon \R\to\R$?
\end{itemize}
Of course, any $f$ admitting $D^n$ extension $\bar f\colon \R\to\R$ must satisfy the property 
\begin{itemize}
	\item[($V_n$):] $f$ is $D^n$ and ($W_{n-1}$) from Theorem \ref{th:Whitney},
\end{itemize}
as $\bar f$ is $C^{n-1}$. 
Also,
Theorem~\ref{thmMC&KC}(a) of Jarn\'\i k immediately implies the following.
\cor{cor:Jarnik}{Let $P\subset \R$ be perfect.
A function $f\colon P\to\R$ admits a $D^1$  extension $\bar f\colon \R\to\R$
if, and only if, $f$ is $D^1$. 
}
In particular, since ($V_1)$ holds if, and only if, $f$ is $D^1$, 
the property ($V_n$) consists of a characterization for Q in case of $n=1$. 
This suggests that the property ($V_n$) is also the desired characterization for 
an arbitrary $n\in\N$.
However, this is not the case already for $n=2$, as exemplified by 
the function $f\restriction P$ from Example \ref{ex111}.
In particular, the question Q remains an open problem for $n\geq 2$, see Problem~\ref{probJarnik}. 

\subsection{Generalized Ulam-Zahorski interpolation problem}\label{sec:InterpolationHigherOrder} 
One can formulate  Ulam-Zahorski interpolation problem for any two  arbitrary classes $\F$ and $\G$ of functions from $\R$ to $\R$ (or, more generally, from a space $X$ into $Y$) as 
the following statement: 
\begin{itemize}
\item[$\UZ(\G,\F)$:] {\it For every $g\in\G$ there is an $f\in\F$ with uncountable $[f=g]$.}
\end{itemize}
Of course, if $\G'\subseteq\G$ and $\F\subseteq\F'$, then 
$\UZ(\G,\F)$ implies $\UZ(\G',\F')$. 

In this notation Zahorski's negative solution of Ulam's problem can be expressed simply as 
$\neg \UZ(C^\infty,\A)$, where $\A$ denotes the class of all real analytic functions. 
Also, Zahorski's question can be understood as an inquiry on the validity of
$\UZ(\G,\F)$ for all 
pairs $\la \G,\F\ra$ of families from 
\[
{\mathbb D}=\{C^\infty\}\cup\{C^n\colon n<\omega\}\cup\{D^n\colon n\in\N\}.
\] 
With the exception of the unknown validity of $\UZ(D^1,D^2)$
(see Problem~\ref{probInterp}), all these interpolation statements are well understood,
as summarized in the following theorem.
Recall that ${\mathbb D}$ is ordered by inclusion as follows:
\[
C^\infty \subsetneq 
\cdots \subsetneq C^{n+1} \subsetneq D^{n+1} \subsetneq  C^n \subsetneq D^n \subsetneq 
\cdots \subsetneq  C^1 \subsetneq  D^{1} \subsetneq  C^0. 
\]

%\pagebreak

\thm{th:GenUZ}{
For every $n\in \N$ with $n\geq 2$:
\begin{itemize}
\item[(a)] $C^1$ is the smallest $\F\in {\mathbb D}$ for which $\UZ(C^0,\F)$ holds.
\item[(b)] If $\F\in {\mathbb D}$ is the smallest 
for which $\UZ(D^1,\F)$ holds,
then $\F\in\{C^1,C^2\}$. 

\item[(c)] $C^2$ is the smallest $\F\in {\mathbb D}$ for which $\UZ(C^1,\F)$ holds.
\item[(d)] $C^n$ is the smallest $\F\in {\mathbb D}$ for which $\UZ(D^n,\F)$ holds.
\item[(e)] $C^n$ is the smallest $\F\in {\mathbb D}$ for which $\UZ(C^n,\F)$ holds.
\end{itemize}
}

\begin{proof}
(a) The interpolation $\UZ(C^0,C^1)$ holds by Theorem~\ref{Thm:Inter}. 
To see the negation of $\UZ(C^0,D^2)$ recall that 
Olevski\v{\i} constructed, in~\cite{Ol}, a continuous function $\varphi_0\colon [0,1]\to\R$ 
which can agree with every $C^2$ function on at most countable set.
In particular, if $\psi$ is a $C^\infty$ map from $\R$ onto $(0,1)$,
then $f_0=\varphi_0\circ \psi$ justifies 
$\neg\UZ(C^0,C^2)$. The same function $f_0$ also justifies 
$\neg\UZ(C^0,D^2)$.\footnote{The fact that Olevski\v{\i}'s function $\varphi_0$
also cannot agree with any $D^2$ function on  an uncountable set was remarked, without a proof,
by Jack B. Brown in \cite{Br1}. Our argument proves that, in fact, every function justifying 
$\neg\UZ(C^0,C^2)$ justifies also $\neg\UZ(C^0,D^2)$.}
Indeed, otherwise, there exists a $D^2$ function $f\colon \R\to\R$ such that
$[f_0=f]$ contains a perfect set $P$.
Then, by Theorem~\ref{thm:UZ}, there exists also a $C^2$ function $g\colon\R\to\R$ 
for which the set $Q\coloneqq[f=g]\cap P$ is uncountable.
Then $Q\subset [f=g]\cap [f_0=f]\subset [f_0=g]$, 
which is impossible, since $f_0$ justifies $\neg\UZ(C^0,C^2)$. 

(b)  The interpolation $\UZ(D^1,C^{1})$ holds either by  Theorem~\ref{Thm:Inter} or Theorem~\ref{thm:UZ}. 
If $\UZ(D^1,D^2)$ does not hold, see Problem~\ref{probInterp}, then 
clearly $\D= C^1$. 
If $\UZ(D^1,D^2)$ holds, then by Theorem~\ref{thm:UZ}, $\UZ(D^1,C^2)$ holds as well. 
Clearly $\UZ(D^1,D^3)$ does not hold, since this would imply $\UZ(C^1,D^3)$, contradicting (c).

(c) The interpolation $\UZ(C^1,C^2)$ is proved by Olevski\v{\i} in~\cite{Ol}. 
(See also \cite[thm 6]{Br2}.)
The negation of $\UZ(C^1,C^3)$ is clearly justified by a function $f_1\colon\R\to\R$
given as $f_1(x)\coloneqq \int_0^x f_0(t) \, dt$, where $f_0$ is as in part (a).
It also justifies $\UZ(C^1,D^3)$, what can be deduced from Theorem~\ref{thm:UZ}
similarly as above. 

(d) $\UZ(D^n,C^{n})$ holds by Theorem~\ref{thm:UZ}. The negation of $\UZ(D^n,D^{n+1})$
is justify by a $C^n$ function $f_n$ from part (e).

(e) Olevski\v{\i} constructed, in~\cite[thm 4]{Ol}, a $C^2$ function $\varphi_2\colon [0,1]\to\R$ 
which can agree with every $C^3$ function on at most countable set.\footnote{In fact it is proved 
in~\cite{Ol} that $[\varphi_2=f]$ is at most countable for every $\alpha\in(0,1)$ and a $C^{2+\alpha}$ map $f\colon [0,1]\to\R$ 
(i.e., such that $f$ is $C^2$ and $f''$ is of H\"older class $\alpha$). 
Similarly he proves that his map $\varphi_1$ cannot be interpolated by any $C^{1+\alpha}$ map.
}
Similarly as above $f_2=\varphi_2\circ \psi$ justifies 
$\neg\UZ(C^2,C^3)$ and, using  Theorem~\ref{thm:UZ}, also $\neg\UZ(C^2,D^3)$.
For $n>2$ a function $f_n$ justifying $\neg\UZ(C^n,D^{n+1})$
is obtained as an $(n-2)$th antiderivative of $f_2$ (i.e., $f_n^{(n-2)}=f_2$). 
\end{proof}

We will finish Section~\ref{sec:InterpolationHigherOrder}  with the following theorem, which was heavily used in the above
proof. The theorem can be deduced from \cite[thm. 4]{Whitney1951} and from 
the proof of \cite[thm. 3.1.15]{Federer}. However, our argument below is considerably different 
from the proofs presented there. 

\thm{thm:UZ}{For every $n\in\N$, perfect $P\subset \R$, and $D^n$ function 
$f\colon\R\to\R$ there exists a $C^n$ function 
	$g\colon\R\to\R$ for which the set $[f=g]\cap P$ is uncountable.
} 

For $n=1$ this can be deduced from the proof of Theorem~\ref{Thm:Inter} 
presented above.\footnote{Since $f$ is differentiable, we can find a perfect subset $P_0$ of $P$
such that $f\restriction P_0$ is Lipschitz. Then, as in
Proposition~\ref{propMain}, we can find perfect subset $Q$ of $P_0$ such that
the conclusion of the proposition holds. This is the only fact used in the proof of Theorem~\ref{Thm:Inter}. 
}
Also, for $n\geq 2$ the theorem can be deduced from 
the (complicated) lemma \cite[lem. 3.7]{KC&JS2018}.
Instead, we will provide below a short argument 
based on 
Theorem~\ref{th:Whitney} and 
the following lemma, which is of independent interest. 

\lem{lem:UZ}{Let $\psi\colon P^2\setminus\Delta\to\R$ be continuous, where $P\subset\R$ is perfect.
If $\delta_1,\delta_2\colon P^2\cap\Delta\to\R$ are continuous,
$\psi_1=\psi\cup\delta_1$ is continuous with respect to the first variable,
and $\psi_2=\psi\cup\delta_2$ is continuous with respect to the second variable,
then $\delta_1=\delta_2$. 
}

\begin{proof}
By way of contradiction, assume that $\delta_1\neq \delta_2$.
Then, there exists an 
$\e>0$ and open non-empty set $U\subset P$ 
such that%
\footnote{Indeed, if $r\in P$ is such that $\delta_1(r,r)\neq \delta_2(r,r)$ and $\e=|\delta_1(r,r)-\delta_2(r,r)|/3$,
then the set 
$U=\{p\in P\colon |\delta_1(p,p)-\delta_1(r,r)|<\e\ \&\ |\delta_2(p,p)-\delta_2(r,r)|<\e\}$ is as needed. 
}
\begin{equation}\label{eq9}
\mbox{  $|\delta_1(p,p)-\delta_2(q,q)|>\e$ for every 
$p,q\in U$.}
\end{equation}
Since $\psi_1$ is continuous with respect to the first variable,
for every $q\in U$ there exists an $n_q\in\N$ such that $|\psi_1(q,q)-\psi_1(p,q)|<\e/2$ 
for every $p\in P$ with $|p-q|<1/n_q$. 
Since $U$ is of second category, there exists an $n\in \N$ such that
$Z=\{q\in U\colon n_q=n\}$ is dense in some non-empty subset $V$ of $U$.
Choose $p\in V$. Since $\psi_2$ is continuous with respect to the second variable,
there exists an open subset $W$ of $V$ containing $p$ and such that
$|\psi_2(p,p)-\psi_2(p,q)|<\e/2$ 
for every $q\in W$.
Choose $q\in W\cap Z$ such that $0<|p-q|<1/n$.
Then $|\psi_1(q,q)-\psi_1(p,q)|<\e/2$, as $|p-q|<1/n=1/n_q$.
In particular,
\[
|\delta_1(p,p)-\delta_2(q,q)|\leq |\psi_1(p,p)-\psi_1(p,q)|+|\psi_2(p,q)-\psi_2(q,q)|<\e,
\]
contradicting (\ref{eq9}).
\end{proof}

\begin{proof}[Proof of Theorem~\ref{thm:UZ}] 
Let $h=f\restriction P$ and define function 
$\Psi\colon P^2\setminus\Delta\to\R$ via formula
\[
\Psi(a,b)\coloneqq
\sum_{k=0}^{n}|q_{h^{(k)}}^{n-k}(a,b)|,
\]
where functions $q_{h^{(k)}}^{n-k}$ are as in Theorem~\ref{th:Whitney}.
Also, let $F\colon P^2\setminus\Delta\to\R$ be given as
$F(a,b)\coloneqq \Psi(a,b)+\Psi(b,a)$. 
By a theorem of Morayne from \cite{Morayne85} applied to $F$,
there exists a perfect $Q\subset P$ for which $F\restriction Q^2\setminus \Delta$ is uniformly 
continuous. Decreasing $Q$, if necessary, we can also assume that $h^{(n)}$ (and so, also $q_{h^{(n)}}^{0}$)
is continuous.
Clearly, $F\restriction Q^2\setminus \Delta$ has a uniformly continuous extension $\bar F\colon Q^2\to\R$.
We claim that
\begin{equation}\label{BrPhi}
\bar F\restriction\Delta\equiv 0.
\end{equation}
To see this, we will use Lemma~\ref{lem:UZ} to the maps $\psi\coloneqq\Psi\restriction Q^2\setminus \Delta$
and $\delta_1,\delta_2\colon P^2\cap\Delta\to\R$, where $\delta_2\equiv 0$ and 
$\delta_1=\bar F\restriction P^2\cap\Delta$. 
The map $\psi_2=\psi\cup\delta_2$ is continuous with respect to the second variable,
since each $q_{h^{(i)}}^{n-i}$ is continuous with respect to the second variable
as long as $h^{(i)}$ has $D^{n-i}$ extension onto $\R$.
(For an easy argument in case when $n-i>0$
see e.g. \cite[prop. 3.2(i)]{KC&JS2018}.) 
The map 
$\psi_1=\psi\cup\delta_1$ is continuous with respect to the first variable,
since $\bar F-\psi_2$ is continuous with respect to the second variable
and $\psi_1(a,b)=(\bar F-\psi_2)(b,a)$ for every $\la a,b\ra\in Q^2$. 
Hence, by Lemma~\ref{lem:UZ}, $\delta_1=\delta_2$.
Thus, for every $a\in Q$ we have
$$\bar F(a,a)=\lim_{b\to a,\, b\in Q} (\psi(a,b)+\psi(b,a))=\delta_1(a,a)+\delta_2(a,a)=0,
$$ proving (\ref{BrPhi}). 

Finally, by (\ref{BrPhi}), 
$$\sum_{k=0}^{n}|q_{h^{(k)}}^{n-k}(a,b)|+\sum_{k=0}^{n}|q_{h^{(k)}}^{n-k}(a,b)|=\bar F(a,b)$$ on $Q^2$,
so it 
is continuous. 
It is easy to see (compare \cite[lem. 3.5]{KC&JS2018}) that this implies that 
$q_{h^{(i)}}^{n-i}$ is continuous on $Q^2$ for every $i\leq n$.
In particular, $f\restriction Q=h\restriction Q$ satisfies the assumptions of 
Theorem \ref{th:Whitney} and so, it admits a $C^n$ extension $g\colon\R\to\R$.
Therefore, $[f=g]\cap P$ contains an uncountable $Q$, as needed.
\end{proof}

% % % % % % % % % % % % % % % % % % % % % %
% % % % % % % % % % % % % % % % % % % % % %
% % % % % % % % % % % % % % % % % % % % % %
% % % % % % % % % % % % % % % % % % % % % %

\subsection{Smooth functions on $\R$ and joint continuity for maps on $\R^2$} 

Let $\mathcal{H}$ be a class of maps $h \in \mathbb{R}^\mathbb{R}$, each $h$ identified with its graph
$h= \{ \langle x, h(x) \rangle\colon x \in \mathbb{R} \}\subset\R^2$. 
We say that a function $f\colon \mathbb{R}^2 \rightarrow \mathbb{R}$ is 
{\em $\mathcal{H}$-continuous}\/ provided $f\restriction h$ is continuous for every $h \in \mathcal{H}$. 
Also, $h$ is said to be
{\em ${\mathcal{H}}^{*}$-continuous}\/ whenever, for every $h \in \mathcal{H}$, the restrictions 
$f\restriction h$ and $f \restriction h^{-1}$ are continuous, where $h^{-1} = \{ \langle h(x),x \rangle\colon x \in \mathbb{R} \}$. 
Thus, in ${\mathcal H}^*$-continuity, we examine the restrictions of $f$ to the functions $h \in \mathcal{H}$ treated as functions ``from $x$ to $y$'' and as functions ``from $y$ to $x$.'' 
Consider the following general question:
\begin{itemize}
\item[(H)] For which classes ${\mathcal{H}}\subset\mathbb{R}^\mathbb{R}$, the ${\mathcal{H}}^*$-continuity of $f\colon \mathbb{R}^2 \to\R$ 
implies its joint continuity?
\end{itemize}
For the class $\A$ of real analytic functions the question (H) has a negative answer, as 
noted in 1890 by %Ludwig Scheeffer 
Ludwig Scheeffer (1859-85)
(see \cite{Sche} or \cite{Ros})  
and in 1905 by Henri Lebesgue~\cite[pp. 199--200]{Lebesgue}. 
On the other hand, in 1948 text~\cite[pp. 173--176]{Luz}
Nikolai Luzin (1883-1950) proves that $(C^0)^*$-continuity implies joint continuity. 

\begin{figure}[h!]
	\centering
	\begin{tabular}{c}
		\includegraphics[width=0.32\textwidth]{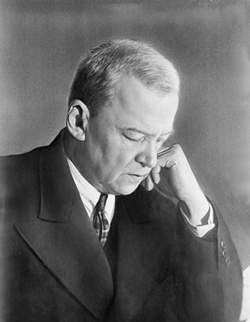} 
	\end{tabular}
	\caption{Nikolai Luzin.}
	\label{pic_NL}
\end{figure}
% % % %
% % % %
% % % %

The final answer to the question (H) for $\H$ in the collection $\{C^n(\R)\colon n<\omega\}$ was given
by Arthur Rosenthal (1887-1959) in his 1955 paper~\cite{Ros}:

\thm{Rosen}{If $f\colon\mathbb{R}^2 \rightarrow \mathbb{R}$ is $(C^1)^*$-continuous, then it is also continuous.
However, there exist discontinuous $(C^2)^*$-continuous functions $g\colon \R^2\to\R$. 
}

More on this subject can be found in a survey \cite{CiMi} of the first author and David Miller. 

\begin{figure}[h!]
	\centering
	\begin{tabular}{cc}
		\includegraphics[width=0.28\textwidth]{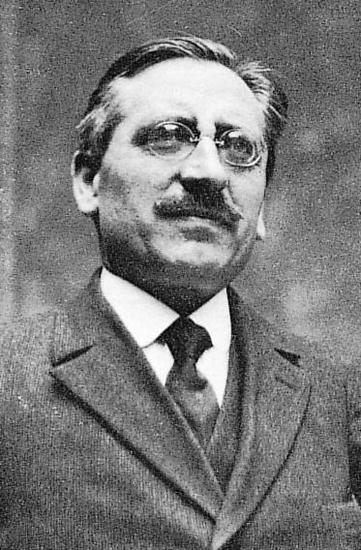} \hspace{.25cm} & \hspace{.25cm} \includegraphics[width=0.31\textwidth]{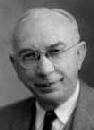} 
	\end{tabular}
	\caption{Henri Le\'on Lebesgue and Arthur Rosenthal.}
	\label{pic_LR}
\end{figure}

% % % % % % % % % % % % % % % % % % % % % %
% % % % % % % % % % % % % % % % % % % % % %

\section{Some related results independent of ZFC}\label{sec:independence}

It is a common mathematical knowledge that the smoother a function is,
the more regular is its behavior. Thus, one could expect, that 
there will not be many statements
about the smooth functions (from $\R$ to $\R$)
that cannot be decided within the standard  axioms ZFC of set theory. 
Nevertheless, there are quite a few 
results, loosely related to the preceding material, 
that fall under such category. The goal of this section is to describe them. 

\begin{figure}[h!]
	\centering
	\begin{tabular}{cc}
		\includegraphics[width=0.3\textwidth]{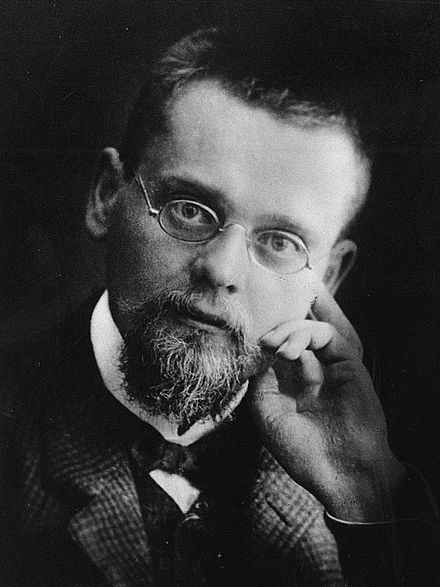} \hspace{.25cm} & \hspace{.25cm} \includegraphics[width=0.285\textwidth]{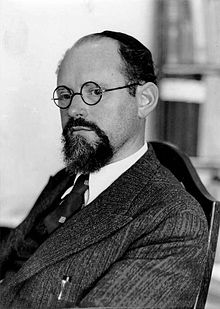} 
	\end{tabular}
	\caption{Ernst Friedrich Ferdinand Zermelo and Abraham Halevi ``Adolf'' Fraenkel.}
	\label{pic_ZF}
\end{figure}

\subsection{Set-theoretical background} 

It is assumed that the reader of this section is familiar with the standard notation and commonly known results of modern set theory, as presented either in \cite{CiBook} or \cite{kunen}. Just to give a brief basic overview, recall that, given a set $X$, its cardinality is denoted by $|X|$. 
The symbol  $\omega$ stands for the cardinality of $\N$ and  $\continuum\coloneqq 2^\omega$ is the cardinality of $\R$. A function shall  be  identified with its graph. The famous Continuum Hypothesis (CH) states that there is no set whose cardinality is strictly between that of 
the integers, $\omega$, and that of the real numbers, $\continuum$. 
CH was advanced by Georg Cantor (1845-1918) in 1878, \cite{cantor}, 
and a problem of its truth or falsehood was the first of Hilbert's 23 problems presented in 1900
at the International Congress of Mathematicians (see \cite{h1,h2}).
It turns out that CH is independent of ZFC---the standard system of axioms of set theory,  including the axiom of choice, introduced by
Ernst Friedrich Ferdinand Zermelo (1871--1953) and Abraham Halevi ``Adolf'' Fraenkel (1891--1965).

The (relative) consistency of CH with ZFC was proved in 1940 paper~\cite{KGodel} by Kurt Friedrich G\"{o}del (1906--1978);
the independence of CH from ZFC (i.e., the consistency of $\neg$CH with ZFC) 
was 
proved in 1963 by Paul Cohen (1934-2007), see~\cite{Cohen1,Cohen2}.

Martin's Axiom (MA), introduced by Donald A. Martin (1940-) and Robert M. Solovay (1938-) in 1970 paper~\cite{MAx}, is a statement that is independent of ZFC. It is implied by CH, but it is independent of ZFC+$\neg$CH.  
Roughly, MA says that all cardinal
numbers less than $\continuum$ behave like $\omega$.
MA has quite the number of  interesting combinatorial, analytic, and topological consequences (see, e.g., \cite{fremlin}).

\begin{figure}[h!]
	\centering
	\begin{tabular}{ccc}
		\includegraphics[width=0.257\textwidth]{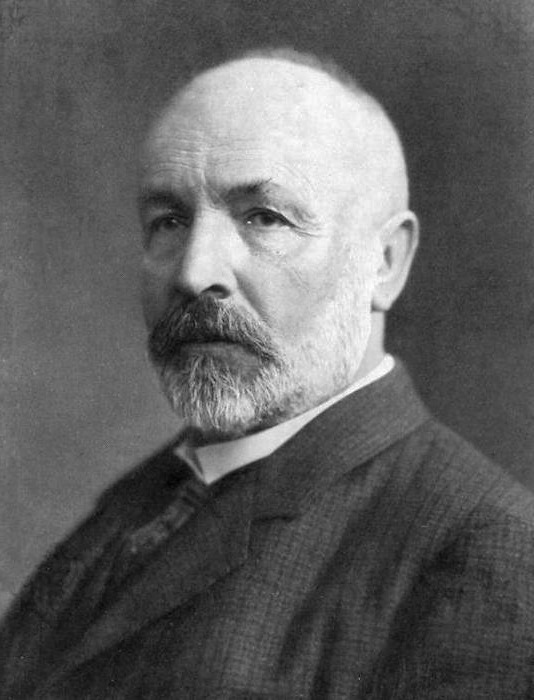} \hspace{.1cm} & \hspace{.1cm} \includegraphics[width=0.26\textwidth]{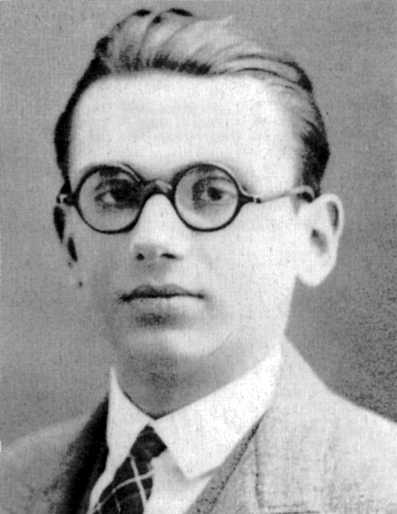} \hspace{.1cm} & \hspace{.1cm} \includegraphics[width=0.25\textwidth]{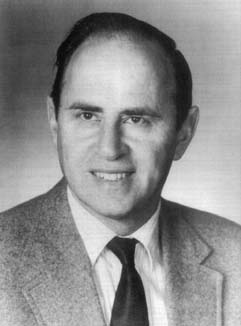}
	\end{tabular}
	\caption{From left to right: Georg Cantor, Kurt G\"{o}del, and Paul Cohen.}
	\label{pic_CGC}
\end{figure}

\subsection{Consistency results related to the interpolation problems} 

By Theorem~\ref{thm:UZ}, 
for every $n\in\N$ and $f\in D^n$ there exists a $g\in C^n$
such that the set $[f=g]$ contains a perfect set. 
A further 
natural question in this context, 
examined by the first author %Krzysztof C. Ciesielski 
and Janusz Pawlikowski in \cite{CiPa,CPAbook},
is about the smallest cardinality $\kappa$ 
such that each $f\in D^n$ is covered by at most $\kappa$-many $C^n$ functions. 
More precisely, we like to know the value of $\cov(D^n,C^n)$, where, for 
$\F,\G\in{\mathbb D}$, 
\[
\cov(\F,\G)\coloneqq\min
\left\{\kappa\colon (\forall f\in \F)
(\exists \G_0\subset \G)\  |\G_0|\leq\kappa\ \&\ f\subset\bigcup\G_0\right\}\!.
\]

An easy ZFC result in this direction is as follows. 

\prop{propCOV}{$\omega<\cov(D^n,C^n)\leq\continuum$ and $\cov(C^{n-1},D^n)= \continuum$ for every $n\in\N$.
}

\begin{proof}
The inequalities $\cov(D^n,C^n)\leq\continuum$ and $\cov(C^{n-1},D^n)\leq \continuum$
are 
obvious, 
since any function $f\colon\R\to\R$
can be covered by $\continuum$-many 
$C^n$ maps (e.g. constant), one for each point $\la x,f(x)\ra\in f$. 

The inequality $\cov(C^0,D^1)\geq \continuum$ is ensured by any function $g_0\in C^0$ 
which has infinite derivative on an uncountable set $Z\subset\R$,
as then for any $f\in D^1$ the set $[f\cap g_0]\cap Z$ is at most countable. 
We can take as $g_0$ the 
Pompeiu's function $g$ from Proposition~\ref{pr:Pom}. 
For general $n\in N$ define functions $g_{n-1}$ inductively, by starting with $g_0$ as above and 
putting $g_{n}(x)=\int_0^x g_{n-1}(t)\, dt$.
Then,  $g_{n}\in C^n$ and $[f\cap g_n]\cap Z$ is at most countable for every $f\in D^n$,
ensuring that 
$\cov(C^{n-1},D^n)\geq \continuum$.

To see that $\cov(D^n,C^n)>\omega$ we need to find an $h_n\in D^n$ which cannot be covered
by countably many $C^n$ maps. For $n=1$ this is witnessed by 
the inverse of a Pompeiu's function from Proposition~\ref{pr:Pom}.
Indeed, if $h_1$ is this map, then for any $C^1$ function $g$ the set $[f=g]$ must be nowhere dense and so,
by the Baire Category Theorem, countably many of such sets cannot cover $\R$. 
For general $n\in N$ define functions $h_n$ inductively, by starting with the above $h_1$ and 
putting $h_{n+1}(x)=\int_0^x h_n(t)\, dt$.
Then the Baire Category Theorem once again ensures that $h_n$ cannot  be covered by 
countably many maps from $C^n$. 
\end{proof}

By Proposition~\ref{propCOV}, CH implies that $\cov(D^n,C^n)=\continuum=\omega_1$. 
Also, even under $\neg$CH, Martin's Axiom implies that $\cov(D^n,C^n)=\continuum$. 
This follows from the argument we used to show that $\cov(D^n,C^n)>\omega$, since, under MA, the union of less than $\continuum$-many 
nowhere dense sets does not cover $\R$. 

Nevertheless, the next theorem shows that it is consistent with ZFC that $$\cov(D^n,C^n)=\omega_1<\continuum.$$
This result can be found \cite{CiPa,CPAbook}.
However, the proofs given in both these sources are incorrect for $n\geq 2$, as 
shown by the authors in paper \cite{KC&JS2018},
which contains also a corrected argument for  the theorem. 

Recall that CPA, {\em the Covering Property Axiom}, is consistent with ZFC.
It holds in the iterated perfect set model.

\thm{CPth1}{{\rm CPA } implies that $\cov(D^n,C^n)=\omega_1<\continuum$ for every $n\in\N$.}

Actually, \cite{KC&JS2018} contains a stronger result: {\em
under {\rm CPA}, for every $n\in\N$ there exists an $\F_n\subset C^n$ of cardinality $\omega_1<\continuum$
that almost covers every $f\in D^n$, in a sense that $f\setminus \bigcup\F_n$ 
has cardinality $\leq\omega_1$.} For $n=0$ this was proved earlier in~\cite{CGNSprep2017}.

Interestingly, Theorem~\ref{CPth1} implies (consistently) the in\-ter\-po\-la\-tion theo\-rems\break 
$\UZ(D^n,C^n)$, which we discussed earlier. Indeed, if $f\in D^n$ is covered by the graphs of $<\continuum$-many functions $g\in C^n$, then for one of these functions $g$ the set $[f=g]$ must be uncountable.

\subsection{Covering $\R^2$ by the graphs of few $C^1$ maps}\label{secR2Cov}

In this subsection, for a function $f\colon \R\to\R$ the symbol $f^{-1}$ will stand 
for the inverse relation, that is,
$f^{-1}=\{\la f(x),x\ra\colon x\in\R\}$. 

Wac\l aw Sierpi\' nski (1882-1969) showed in \cite[Property P1]{sierpinskiCH}  (see, also, \cite{komtot}) that CH is equivalent to the fact that 
there exists a family $\F$ of 
countably many functions from $\R$ to $\R$ (they cannot be ``nice'') such that $\R^2= \bigcup_{f\in\F} (f\cup f^{-1})$. 
The sets $A=\bigcup\F$ (with each vertical section countable) 
and $B=\R^2\setminus A$ (with each horizontal section countable) 
form what is known as {\em Sierpi\' nski's decomposition}. 
To see this result, notice that if CH holds, then $\R$ can be enumerated,  
with no repetitions, as 
$\{r_{\alpha}\colon \alpha < \omega_1\}$.
Also, for every $\alpha<\omega_1$, the set $\{\xi\colon \xi\leq\alpha\}$ 
can be enumerated, with possible repetitions, as $\{\xi(\alpha,n)\colon n<\omega\}$. 
Then $\F$ can be defined as the family of all functions $f_n\colon \R\to\R$, $n<\omega$,
defined as $f_n(r_\alpha)=r_{\xi(\alpha,n)}$.
Conversely, assume that $\mathfrak{c} \ge \omega_2$ and,
by way of contradiction, that there exists a 
Sierpi\' nski's decomposition $\{A,B\}$ of $\R^2$ as above.
Pick $X \subset \R$ of cardinality $\omega_1$ and 
$y\in \R\setminus \bigcup_{x\in X} A_x$, where
$A_x=\{y\colon \la x,y\ra\in A\}$.  
Such a choice is possible, since $\bigcup_{x\in X} A_x$ has cardinality $\leq\omega_1$ (as each  $A_x$ is countable),
while $|\R|=\continuum>\omega_1$.
Also, we can choose $x\in U\setminus B^y$, where
$B^y=\{x\colon \la x,y\ra\in B\}$, since $|B|=\omega_1>|B^y|$. 
But then, $\la x,y\ra\in \R^2\setminus(A\cup B)$, a contradiction. 

It is an easy generalization of Sierpi\' nski's argument (see, e.g., \cite{komtot}) that there exists family $\F$ of cardinality $\kappa$
of functions from $\R$ to $\R$ with 
$\R^2= \bigcup_{f\in\F} (f\cup f^{-1})$
if, and only if, $\continuum\leq \kappa^+$.

\begin{figure}[h!]
	\centering
	\includegraphics[width=0.35\textwidth]{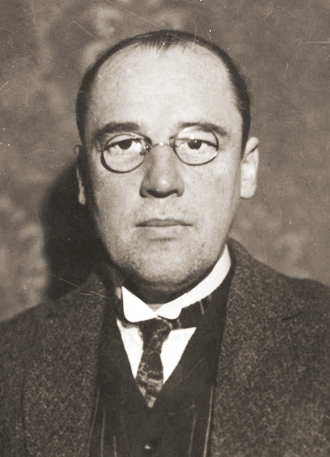}
	\caption{Wac\l aw Sierpi\' nski.}
	\label{pic_Sier}
\end{figure}

None of these results uses continuous functions. In fact, for a countable family $\F$ of continuous functions,
the set $\bigcup_{f\in\F} (f\cup f^{-1})$ is of first category, so it cannot be equal $\R^2$.
Nevertheless, it is consistent with ZFC that $\R^2= \bigcup_{f\in\F} (f\cup f^{-1})$
for a family $\F$ of less than $\continuum$-many continuous functions. In fact, these functions 
can be even $C^1$! 
This follows from the following theorem 
of the first author %Krzysztof C. Ciesielski 
and Janusz Pawlikowski proved in \cite{CiPa}. (Compare also~\cite{CPAbook}.)
Note, that 
an earlier, weaker version of the theorem was proved by Juris Stepr\={a}ns~\cite{St2}.

\thm{CPth2}{{\rm CPA } implies that there exists a family $\F$ of size $\omega_1<\continuum$
of $C^1$ functions such that $\R^2= \bigcup_{f\in\F} (f\cup f^{-1})$.}

Notice that $C^1$ is the best possible smoothness 
for such result, since there is no family $\F\subset D^2$ of size $<\continuum$ with
$\R^2= \bigcup_{f\in\F} (f\cup f^{-1})$.
This is the case, since there exists a continuous injection $h$ from a compact perfect set $P$ into $\R$ 
such that both $h$ and its inverse have infinite second derivative at every point of the domain.
(See \cite{CiPa} or~\cite[example 4.5.1]{CPAbook}.)
This implies that $h\cap (f\cup f^{-1})$ is at most countable for every $f\in D^2$. 
That is, $h\not\subset \bigcup_{f\in\F} (f\cup f^{-1})$ for every $\F\subset D^2$ of size $<\continuum$.  See also Problem~\ref{prA}. 

\subsection{Big continuous and smooth images of sets of cardinality $\continuum$}

In Section~\ref{secR2Cov} we took a property that 
a small family of $\F$ of arbitrary functions from $\R$ to $\R$ can cover $\R^2$ and investigated to what extend the functions in $\F$ can be, consistently, continuous or smooth. 
In a 1983 paper \cite{Mi} Arnold W. Miller considered the similar regularization of the family $\F$ in 
the following statement which, of course, holds for $\F=\R^\R$.  

\begin{itemize}
	\item[$Im^*(\F)$:] For every $S\in[\real]^\continuum$ there is an $f\in\F$ such that $f[S]=[0,1]$. 
\end{itemize}
In particular, he proved there 

\thm{th:AMiller}{It is consistent with ZFC, holds in the iterated perfect set model, that $Im^*(C^0)$ holds. }

This proved that the statement $Im^*(C^0)$ is independent of ZFC axioms, since 
$Im^*(C^0)$ is false under CH and, more generally, MA, see e.g.~\cite{Mi,CPAbook}.\footnote{A (generalized) Luzin's 
set cannot be mapped continuously onto $[0,1]$.
}
A considerably simpler proof of Theorem~\ref{th:AMiller} was given by 
the first author %Krzysztof C. Ciesielski 
and Janusz Pawlikowski in \cite{CiPa,CPAbook}, showing that 
$Im^*(C^0)$ follows from 
the Covering Property Axiom CPA. (In fact, its simplest version 
CPA$_{\rm cube}$.)  

The next natural question is 
wether $Im^*(\F)$ can be also consistent 
for a family $\F\in{\mathbb D}$ strictly smaller than $C^0$.  
Formally, the answer is negative (see e.g. \cite{CiNi}), 
since 
$Im^*(\D^1)$  is false for any $S\in[\real]^\continuum$  of Lebesgue measure zero and $f\in\D^1$, 
as every $f\in\D^1$ satisfies Luzin's condition (N), that is,  
maps Lebesgue measure zero sets onto sets of measure zero. 
Nevertheless, it is easy to see that for $\F=C^0$ the statement 
$Im^*(\F)$ is equivalent to 
\begin{itemize}
	\item[$Im(\F)$:] For every $S\in[\real]^\continuum$ there is $f\in\F$ such that $f[S]$ contains a perfect set. 
\end{itemize}
At the same time, $Im(C^0)$ is equivalent to $Im(C^\infty)$, as indicated in the following theorem
of the first author and Togo Nishiura, see \cite{CiNi}. 

\thm{thAAA}{The properties $Im^*(\C^0)$, $Im(\C^0)$, and $Im(\C^\infty)$ are equivalent (in ZFC). 
In particular, each of these properties is independent of ZFC.}

Finally, let us notice that $Im(\A)$ is false, where $\A$ denotes the class of all real analytic functions. 
A counterexample for $Im(\A)$ is provided in \cite{CiNi}. 

We finish this article by including a section on open problems and potential directions of research.
% % % % % % % % %% % % % % % % % %
% % % % % % % % %% % % % % % % % %
% % % % % % % % %% % % % % % % % %
% % % % % % % % %% % % % % % % % %
% % % % % % % % %% % % % % % % % %
% % % % % % % % %% % % % % % % % %

\section{Final remarks and open problems}\label{sect6}

We would like to begin this section by emphasizing the fact that there is no simpler characterization of being a derivative than the trivial one: 
\begin{quote}
	{\em $f$ is a derivative if, and only if, there exists a function $F$ for which $f=F'$.} 
\end{quote}
Thus, it would certainly be of interest to solve the following.
\pr{pr:der}{
	Find a non-trivial characterization of the derivatives, that is, the functions $h \in \mathbb{R}^\mathbb{R}$ such that $h=f'$ for some $f \in \mathbb{R}^\mathbb{R}$.}

For  more  of  this  problem,  see  the  1947  paper  \cite{Za}  of  
Zahorski or the monograph \cite{Bru1978} of Bruckner. 
Notice, that Chris Freiling gives in \cite{Freiling} that such 
simpler characterization does not exist.

The next problem, related to our discussion in Section~\ref{sec:Nice}, can be found in \cite{Szuca} and \cite{BrCi}. 

\pr{pr:comp}{If $f$ is 
a composition of finite numbers of derivatives from $[0,1]$ into itself, must the graph of $f$ be connected in $\R^2$? 
}

The following problem, related to the discussion in Section~\ref{sec:dynamics}, comes from \cite[problem 1]{Ci112}). 
Notice that there is no function $h$ as in the problem which could be extended to a $C^1$ map from $\R$ to $\R^2$, see \cite[thm. 3.1]{Ci112}.

\pr{prPeano}{Does there exist a compact perfect $P\subset\R$
and a map $h$ from $P$ onto $P^2$ such that $h$ is 
$D^1$ (i.e., $h$ has $D^1$ coordinates)?}

The next problem, on a $D^n$ analog of 
Whitney's $C^n$ Extension Theorem~\ref{th:Whitney}, comes from Section~\ref{sec:higherExt}.
 
\pr{probJarnik}{Find, for every $n\geq 2$, a characterization of all $D^n$  functions $f$ from perfect $P\subset\R$
into $\R$ that  admit $D^n$ extensions $\bar f\colon \R\to\R$.}

Since every $D^n$ map $\bar f$ is also $C^{n-1}$, any $D^n$-extendable function $f$ 
must satisfy property ($W_{n-1}$) from Theorem~\ref{th:Whitney}. 
By Theorem \ref{thmMC&KC}(a), for 
$n=1$ this is also sufficient condition. 
However, Example \ref{ex111} shows that this is not strong enough condition for $n\geq 2$.

The following problem comes from our discussion of Zahorski-Ulam problem presented in 
Section~\ref{sec:InterpolationHigherOrder}. 

\pr{probInterp}{Is the following interpolation true? 
\begin{itemize}
\item[$\UZ(D^1,D^2)$:] {\it For every $g\in D^1$ there is an $f\in D^2$ with uncountable $[f=g]$.}
\end{itemize}
}

We know, by Theorem~\ref{CPth1}, that it is consistent with ZFC (follows from CPA) that $\real^2$ can be covered 
by less than $\continuum$-many graphs of $C^1$ maps.
Also, it is consistent with ZFC (follows from CH) that $\real^2$ cannot be covered 
by less than $\continuum$-many graphs of $C^0$ maps.
The next problem asks how the existence of such coverings by $C^0$  and $C^1$ maps are related. 

\pr{prA}{Can it be proved, in ZFC, that if $\R^2= \bigcup_{f\in\F} (f\cup f^{-1})$
for a family $\F\subset C^0$ of size $<\continuum$,
than the same is true for some $\F\subset C^1$ of size $<\continuum$?
How about the family $D^1$ in the same setting?}

\section*{About the authors}

\textbf{Krzysztof C. Ciesielski} received his Master and Ph.D. degrees in Pure Mathematics from Warsaw University, Poland, in 1981 and 1985, respectively. He works at West Virginia University since 1989. In addition, since 2006 he holds a position of Adjunct Professor in the Department of Radiology at the University Pennsylvania. He is author of three books and over 130 journal research articles. Ciesielski's research interests include both pure mathematics (real analysis, topology, set theory) and applied mathematics (image processing, especially image segmentation). He is an editor of Real Analysis Exchange, Journal of Applied Analysis, and Journal of Mathematical Imaging and Vision.

\medskip

\textbf{Juan B. Seoane--Sep\'ulveda} received his first Ph.D. at the Universidad de C\'adiz (Spain) jointly with Universit\"{a}t Karlsruhe (Germany) in 2005. His second Ph.D. was earned at Kent State University (Kent, Ohio, USA) in 2006. His main interests include real analysis, set theory, Banach space geometry, and lineability. He has authored two books and over 120 journal research papers. He is currently a professor at Universidad Complutense de Madrid (Spain) and an editor of Real Analysis Exchange.

\section*{Acknowledgments}

We would like to express our gratitude to Prof. Andrew M. Bruckner for his invaluable advise and encouragement towards this work. J.B. Seoane-Sep\'ulveda was supported by grant MTM2015-65825-P.

%%%%%%%%%%
%%%%%%%%%%
% REFERENCES %
%%%%%%%%%%
%%%%%%%%%%

\section*{Photographs source and citations}

The photographs appearing in this manuscript can be found in the following online locations:

\begin{itemize}
	\item Photograph \ref{pic_cauchy}: 
	\href{https://es.wikipedia.org/wiki/Augustin_Louis_Cauchy}{https://es.wikipedia.org/wiki/Augustin\_Louis\_Cauchy}
	
		\item Photograph \ref{pic_VD}: 
		\href{http://www.villinovolterra.it/en/la-storia/}{http://www.villinovolterra.it/en/la-storia/}
	 and \break \href{https://es.wikipedia.org/wiki/Jean_Gaston_Darboux}{https://es.wikipedia.org/wiki/Jean\_Gaston\_Darboux}
	\item Photograph \ref{pic_Baire}: \href{https://es.wikipedia.org/wiki/Rene-Louis_Baire}{https://es.wikipedia.org/wiki/Ren\'e-Louis\_Baire}
	
	\item Photograph \ref{pic_DenPom}:\\ 
	\href{https://upload.wikimedia.org/wikipedia/ru/1/16/Arnaud_Denjoy.jpeg}
	{\small{https://upload.wikimedia.org/wikipedia/ru/1/16/Arnaud\_Denjoy.jpeg}}
	 and \break 
	 \href{https://upload.wikimedia.org/wikipedia/ru/4/4c/Dimitrie-Pompeiu-small.jpg}
	 {\small{https://upload.wikimedia.org/wikipedia/ru/4/4c/Dimitrie-Pompeiu-small.jpg}}
	\item Photograph \ref{pic_BW}:\\ 
	\href{http://www.sil.si.edu/DigitalCollections/hst/scientific-identity/fullsize/SIL14-B5-07a.jpg} 
	{\scriptsize{http://www.sil.si.edu/DigitalCollections/hst/scientific-identity/fullsize/SIL14-B5-07a.jpg}}\\ 
	and \break 
	\href{https://es.wikipedia.org/wiki/Karl_Weierstrass}{https://es.wikipedia.org/wiki/Karl\_Weierstrass}
	
	 \item Photograph \ref{pic_LT}: \\ 
	 \href{https://es.wikipedia.org/wiki/Bartel_Leendert_van_der_Waerden} 
	 {https://es.wikipedia.org/wiki/Bartel\_Leendert\_van\_der\_Waerden} 
	 and \break 
	 \href{https://es.wikipedia.org/wiki/Teiji_Takagi}
	 {https://es.wikipedia.org/wiki/Teiji\_Takagi}
	 
	 \item Photograph \ref{pic_RW}:\\ 
	 \href{https://en.wikipedia.org/wiki/Frigyes_Riesz} 
	 {https://en.wikipedia.org/wiki/Frigyes\_Riesz} 
	 and \break 
	 \href{https://es.wikipedia.org/wiki/Hassler_Whitney}
	 {https://es.wikipedia.org/wiki/Hassler\_Whitney}
	 
	 \item Photograph \ref{pic_jarnik}:\\ 
	 \href{https://web.math.muni.cz/biografie/obrazky/jarnik_vojtech.jpg} 
	 {https://web.math.muni.cz/biografie/obrazky/jarnik\_vojtech.jpg} 
	 
	 \item Photograph \ref{pic_UZ}: \\
	 \href{http://aprender-mat.info/historyDetail.htm?id=Ulam}
	 {http://aprender-mat.info/historyDetail.htm?id=Ulam}
	  and \break 
	  \href{http://ms.polsl.pl/trysektor/materialy/profZahorski.pdf}
	  {http://ms.polsl.pl/trysektor/materialy/profZahorski.pdf}

	 \item Photograph \ref{pic_ME}:  
	 \href{https://cms.math.ca/notes/v35/n3/Notesv35n3.pdf}
	 {https://cms.math.ca/notes/v35/n3/Notesv35n3.pdf}
	 
	 \item Photograph \ref{pic_NL}: 
	 \href{http://theor.jinr.ru/~kuzemsky/Luzinbio.html}
	 {http://theor.jinr.ru/\~{}kuzemsky/Luzinbio.html}

 \item Photograph \ref{pic_LR}:\\ 
\href%{https://en.wikipedia.org/wiki/Henri_Lebesgue}%
{https://www.goodreads.com/author/show/2403467.Henri_Lebesgue}
 {https://www.goodreads.com/author/show/2403467.Henri\_Lebesgue}
and\break 
\href{http://histmath-heidelberg.de/homo-heid/rosenthal.htm}
{http://histmath-heidelberg.de/homo-heid/rosenthal.htm}
 
  \item Photograph \ref{pic_ZF}: 
  \href{https://de.wikipedia.org/wiki/Ernst_Zermelo}
  {https://de.wikipedia.org/wiki/Ernst\_Zermelo}
  and\break 
  \href{https://de.wikipedia.org/wiki/Adolf_Abraham_Halevi_Fraenkel}
  {https://de.wikipedia.org/wiki/Adolf\_Abraham\_Halevi\_Fraenkel}

 \item Photograph \ref{pic_CGC}: 
\href {https://es.wikipedia.org/wiki/Georg\_Cantor} 
{https://es.wikipedia.org/wiki/Georg\_Cantor}, \\
\href{https://es.wikipedia.org/wiki/Kurt_Godel} 
{https://es.wikipedia.org/wiki/Kurt\_G\"odel}, and \\
\href{http://www-history.mcs.st-andrews.ac.uk/PictDisplay/Cohen.html}
{http://www-history.mcs.st-andrews.ac.uk/PictDisplay/Cohen.html}
 
 \item Photograph \ref{pic_Sier}:  
\href{https://en.wikipedia.org/wiki/Waclaw_Sierpinski}
{https://en.wikipedia.org/wiki/Waclaw\_Sierpinski}
 
\end{itemize}

\end{document}